\documentclass[a4paper, 12pt]{amsart}
\usepackage{amsmath}
\usepackage{amssymb, latexsym, amsthm}

\newtheorem{Thm}{Theorem}[section]
\newtheorem{Prop}[Thm]{Proposition}
\newtheorem{Cor}[Thm]{Corollary}
\newtheorem{Lem}[Thm]{Lemma}

\theoremstyle{definition}
\newtheorem*{definition}{Definition}
\newtheorem*{remark}{Remark}

\numberwithin{equation}{section}

\begin{document}
\renewcommand{\d}[0]{\text{\tiny{$\bullet$}}}
\newcommand{\Coim}{\mathrm{Coim}}
\newcommand{\Z}[0]{\mathbb{Z}}
\newcommand{\Q}[0]{\mathbb{Q}}
\newcommand{\F}[0]{\mathbb{F}}
\newcommand{\N}[0]{\mathbb{N}}
\renewcommand{\O}[0]{\mathcal{O}}
\newcommand{\m}[0]{\mathrm{m}}
\newcommand{\Tr}{\mathrm{Tr}}
\newcommand{\Hom}[0]{\mathrm{Hom}}
\newcommand{\Gal}[0]{\mathrm{Gal}}
\newcommand{\Res}[0]{\mathrm{Res}}
\newcommand{\id}{\mathrm{id}}
\newcommand{\cl}{\mathrm{cl}}
\newcommand{\mult}{\mathrm{mult}}
\newcommand{\adm}{\mathrm{adm}}
\newcommand{\tr}{\mathrm{tr}}
\newcommand{\pr}{\mathrm{pr}}
\newcommand{\Ker}{\mathrm{Ker}}
\newcommand{\ab}{\mathrm{ab}}
\newcommand{\sep}{\mathrm{sep}}
\newcommand{\triv}{\mathrm{triv}}
\newcommand{\alg}{\mathrm{alg}}
\newcommand{\ur}{\mathrm{ur}}
\newcommand{\Coker}{\mathrm{Coker}}
\newcommand{\Aut}{\mathrm{Aut}}
\newcommand{\Ext}{\mathrm{Ext}}
\newcommand{\Iso}{\mathrm{Iso}}
\newcommand{\M}{\mathcal{M}}
\newcommand{\GL}{\mathrm{GL}}
\newcommand{\Fil}{\mathrm{Fil}}
\newcommand{\an}{\mathrm{an}}
\renewcommand{\c}{\mathcal }
\newcommand{\W}{\mathcal W}
\newcommand{\R}{\mathcal R}
\newcommand{\crys}{\mathrm{crys}}
\newcommand{\st}{\mathrm{st}}
\newcommand{\CM}{\mathrm{CM\Gamma }}
\newcommand{\CV}{\mathcal{C}\mathcal{V}}
\newcommand{\G}{\mathrm{G}}
\newcommand{\Map}{\mathrm{Map}}
\newcommand{\Sym}{\mathrm{Sym}}
\newcommand{\Spec}{\mathrm{Spec}}
\newcommand{\Gr}{\mathrm{Gr}}
\newcommand{\I}{\mathrm{Im}}
\newcommand{\Frac}{\mathrm{Frac}}
\newcommand{\To}{\longrightarrow}
\newcommand{\MF}{\mathrm{MF}}
\newcommand{\Aug}{\mathrm{Aug}}
\newcommand{\wt}{\widetilde}
\newcommand{\op}{\mathrm}
\newcommand{\Ad}{\op{Ad}}
\newcommand{\ad}{\op{ad}}
\newcommand{\into}[0]{\rightarrowtail}
\newcommand{\onto}[0]{\twoheadrightarrow}
\renewcommand{\iff}[0]{\Longleftrightarrow}
\newcommand{\wh}[0]{\widehat}
\newcommand{\fr}[0]{\frak}

\title[$p$-extensions of higher local fields]
{Galois groups of $p$-extensions of higher local 
fields}
\author{Victor Abrashkin}
\address{Department of Mathematical Sciences, 
Durham University, Science Laboratories, 
Lower Mountjoy, Durham DH1 3LE, United Kingdom \ \&\ Steklov 
Institute, Gubkina str. 8, 119991, Moscow, Russia
}
\email{victor.abrashkin@durham.ac.uk}
%\date 

\begin{abstract} 
Suppose $\mathcal K$ is $N$-dimensional local field of characteristic $p>2$, 
$\mathcal G_{<p}$ is the maximal 
quotient of $\mathcal G==\mathop{Gal}(\mathcal K_{sep}/\mathcal K)$ 
of period $p$ and nilpotent class $<p$ 
and $\mathcal K_{<p}\subset \mathcal K_{sep}$ is such that 
$\mathop{Gal}(\mathcal K_{<p}/\mathcal K)=\mathcal G_{<p}$. 
We use nilpotent Artin-Schreier theory 
to identify $\mathcal G_{<p}$ with 
the group $G(\mathcal L)$ obtained from a profinite 
Lie $\mathbb F _p$-algebra $\mathcal L$ 
via the Campbell-Hausdorff composition law. The canonical 
$\mathcal P$-topology 
on $\mathcal K$ is used to define a dense Lie subalgebra 
$\mathcal L^{\mathcal P}$ in $\mathcal L$. The 
algebra $\mathcal L^{\mathcal P}$ can be provided with a 
system of $\mathcal P$-topological 
generators and its $\mathcal P$-open subalgebras correspond to all 
$N$-dimensional extensions of $\mathcal K$ in $\mathcal K_{<p}$. 
These results are applied to higher local fields $K$ of characteristic 0 
containing nontrivial  $p$-th root of unity. 
If $\Gamma =\mathop{Gal}(K_{alg}/K)$ we introduce similarly the quotient  
$\Gamma _{<p}=G(L)$, a dense $\mathbb F _p$-Lie 
algebra $L^{\mathcal P}\subset L$, 
and describe the structure of $L^{\mathcal P}$ in terms 
of generators and relations. The general result is illustrated by 
explicit presentation of $\Gamma _{<p}$ modulo subgroup of third commutators. 
\end{abstract}
\maketitle

\section*{Introduction} \label{S0} 

Let $p>2$ be a fixed prime number.

\subsection{Higher local fields} \label{S0.1} 
The concept of higher local field $K$ of dimension $N\geqslant 1$ was introduced 
as an essential ingredient 
of the theory of higher adeles in the study of 
arithmetic properties of algebraic varieties. 
In dimension 0 we just require that $K$ is finite of characteristic $p$. 
If $N\geqslant 1$ then 
$K$ is a complete discrete valuation field with the residue 
field which is isomorphic to some 
$(N-1)$-dimensional local field of {\sc characteristic} $p$. 
(In this paper we work with fields, which have most 
interesting arithmetic properties.)This residue field 
will be called the 
first residue field $K^{(1)}$ of $K$. Similarly, we 
obtain next residue fields,  
the last (or $N$-th) residue field is necessarily finite and will be always 
denoted by $k\simeq\F _{p^{N_0}}$. For example, 1-dimensional fields appear as 
either finite extensions of $\Q _p$ or fields of 
formal Laurent series in one variable with coefficients in a finite field $k$. 
Basics of the theory of such fields including highly important concept of 
special topology 
($\c P$-topology) together with 
classification results 
 can be found in 
\cite{Zh1, Zh2}, cf.\,also Sect.\,\ref{S2} below.

One of first considerable achievements of the theory of  
higher local fields was the 
construction of higher dimensional generalization 
of local class field theory, cf. \cite{Ka1, Ka2, Ka3, Pa1, Pa2, Pa3} 
and (for explicit aspects of the theory) 
\cite{Vo, AJ}. 
In this setting abelian extensions of $N$-dimensional fields are described 
(in a functorial way) in terms 
of the appropriate Milnor $K_N$-groups. 
The group $\Gamma _K=\op{Gal}(K_{sep}/K)$ is soluble and its 
most interesting part appears as  
the Galois group $\Gamma _K(p)$ of the maximal $p$-extension of $K$. 
Its structure (as well as of any other $p$-group) 
can be described in terms of generators and relations.  
(Notice that in higher dimensions there are substantial problems with a choice of 
a set of generators, cf. below.) 
Any minimal (topological) system of generators of $\Gamma _K(p)$ 
comes from the lifts 
of any topological $\F _p$-basis of 
$\Gamma _K^{ab}/(\Gamma _K^{ab})^p$. 

\subsection{Review of 1-dimensional case} \label{S0.2}
Suppose $K$ is 1-dimensional. In this case (according to class field theory) 
generators of $\Gamma _K(p)$ 
come from any 
$\F _p$-basis of $K^*/K^{*p}$. This basis can be chosen 
in a natural way if we fix a choice of  
uniformizing element of $K$. For example, 
suppose  $K\simeq \F _p((t))$ and for 
all $a\in\Z ^0(p):=\{a\in\Z _{>0}\ |\ \op{gcd}(a,p)=1\}\bigcup\{0\}$, 
$T_a\in K_{sep}$ are such that $T_a^p-T_a=t^{-a}$. For $b\in\Z ^0(p)$, let 
$\tau _b\in \Gamma _K(p)$ be such that $\tau _b(T_a)-T_a=\delta _{ab}$ 
(the Kronecker symbol). Then $\{\tau _a\ |\ a\in\Z ^0(p)\}$ is 
a minimal system of generators in $\Gamma _K(p)$. 

The structure of $\Gamma _K(p)$ 
was described around 1960's as follows:  
\medskip 

--- if $\op{char}\,K=p$ or 
$\op{char}K=0$ and $K$ contains no non-trivial $p$-th roots of 
unity the group $\Gamma _K(p)$ is profinite free (I.\,Shafarevich); 
\medskip 

--- if $K$ contains a non-trivial $p$-th root of unity then 
$\Gamma _K(p)$ has a minimal system of generators 
containing $[K:\Q _p]+2$  elements 
and one (explicitly known) relation (S.\,Demushkin), cf.\,\cite{Se2, Se3, Ko}. 
(This result leads to a complete description of $\Gamma _K$, cf.\cite{JW}.) 
\medskip 

There is no a straight way to extend the above results to 
higher local fields for the following reasons. 

First, there is no any reasonable choice of generators in $\Gamma _K(p)$. 
To illustrate this suppose 
$N=2$ and $K=\F _p((t_2))((t_1))$ is 2-dimensional 
local field of iterated Laurent formal series. 
The extension $K(T)$ such that $T^p-T=t_1^{-1}(1+t_2+
\dots +t_2^n+\dots )$ 
is not contained in the composit of all $K(T_n)$, where 
$T_n^p-T_n=t_1^{-1}t_2^n$, $n\geqslant 0$. 
As a result, the lifts of elements of the Galois groups of the 
elementary field extensions $K(T_{a_1a_2})$, where 
$T^p_{a_1a_2}-T_{a_1a_2}=t_1^{-a_1}t_2^{-a_2}$, generate only very small piece 
of $\Gamma _K(p)$. 
This also can be seen at the level of class field theory, where the 
abelian extensions of $N$-dimensional local field 
$K$ are described via the 
quotients of the $K$-group $K_N(K)$. This group has no natural system of 
generators if $N>1$, but it contains a dense subgroup $K_N^{top}(K)$;  
this subgroup 
can be described via (topological) generators and can be taken instead of 
$K_N(K)$ when studying finite 
abelian extensions of $K$. 

Another concern is related to the strategy used by Demushkin (in  
the 1-dimensional case).  Let $K$ be a 1-dimensional local field of 
characteristic 0 containing a non-trivial $p$-th root of unity. 
For $s\geqslant 1$, let $C^{(p)}_s$ be the $s$-th term of the 
$p$-central series of $\Gamma _K(p)$. 
Then we can use the interpretation of the abelian quotient 
$\Gamma _K(p)/C_2^{(p)}$ in terms of class field theory. 
Applying formalism of the Galois cohomology we can describe explicitly 
the action of this quotient on $C_2^{(p)}/C_3^{(p)}$:  
this involves calculations with Hilbert symbol. 
As a result we can describe the group theoretic structure of 
$\Gamma _K(p)/C_3^{(p)}$ in terms of a 
specially chosen minimal system of generators and 
one (explicitly given) relation. Luckily, this allows us to recover 
the structure of $\Gamma _K(p)$ by choosing special lifts of generators 
which satisfy the simplest possible lift of that relation.  

The above strategy was applied in the case of local fields 
of dimension 2 in \cite{BV}. 
At that time explicit aspects of higher local class field theory, in particular, 
formulas for the Hilbert symbol, 
were just developed by the second author 
of that paper. The authors used the subgroup $K_2^{top}(K)$ of 
$K_2(K)$ and attempted (following the 
1-dimensional strategy) to find the structure of a dense subgroup in 
$\Gamma _K(p)/C_3^{(p)}$. 
The paper \cite{BV} 
justifies 
that in higher dimensions the Demushkin strategy requires enourmous calculations 
but don't give very much information about $\Gamma _K(p)$. 
The whole approach should be profoundly revisited at least for the following reason. 
When we use class field theory and afterwards apply explicit formulas 
for the Hilbert symbol, we actually move in two opposite directions. 
For this reason, it makes sense to avoid the use of class field theory and to 
proceed within the frames of Kummer (or Artin-Schreier) theory from the very beginning. 
Another important concern is that (abelian)  
class field theory is not sufficient for 
understanding the structure of 
$\Gamma _K(p)$ better than just modulo $C_3^{(p)}$.  

In \cite{Ab1, Ab2} the author initiated the study of 
$\Gamma _K(p)$ modulo the subgroup of $p$-th commutators 
for fields of characteristic $p$ 
via so-called nilpotent Artin-Schreier theory. Later we applied   
this theory together with the Fontaine-Wintenberger field-of-norms functor  
to study the case of 1-dimensional local fields 
$K$ with non-trivial $p$-th roots of unity, cf. 
\cite{Ab12, Ab13, Ab14}. As a  result, we obtained the description of  
$\Gamma _{<p}:=\Gamma _K/\Gamma _K^pC_p$ 
in terms of a specially chosen system of generators 
which satisfy one relation. 
(Here $C_p=C_p(\Gamma _K)$ is the closure of the subgroup of 
$p$-th commutators in $\Gamma _K$.) 
This could be considered as an alternative approach to Demushkin's result 
in the 1-dimensional case. 
Actually, we obtained much more in the above papers: 
our result gives also an 
explicit description of the images of all ramification subgroups 
$\Gamma _K^{(v)}$, $v\geqslant 0$, in $\Gamma _{<p}$, cf. also 
Sect.\,\ref{S0.4}.)

\subsection{Main results} \label{S0.3} 
In this paper we develop a techniques allowing us to 
study the structure of 
$\Gamma _{<p}=\Gamma _K/\Gamma _K^pC_p(\Gamma _K)$ in 
terms of generators and relations. We consider the cases where 
either $K=\c K$ has characteristic $p$ or $K$ has characteristic 0 and 
a non-trivial $p$-th root of unity $\zeta _1\in K$.  
In both cases we introduce 
a dense subgroup $\Gamma _{<p}^{\c P}$ of $\Gamma _{<p}$ with  
new $\c P$-topological structure (related to the  
$\c P$-topology on $K$). The subgroup $\Gamma _{<p}^{\c P}$ still 
allows us to study finite extensions of $K$ in $K_{<p}$ and 
admits  
a description in terms of $\c P$-topological generators and relations.
\medskip 

Describe the content of the  paper in more details. .

a) For an $N$-dimensional local field $\c K$ of characteristic $p$ we apply 
nilpotent Artin-Schreier theory to fix an identification 
$\pi: \c G_{<p}\simeq G(\c L)$. Here $\c G_{<p}=\c G/\c G^pC_p(\c G)$ 
is the maximal quotient of 
$\c G=\op{Gal}(\c K_{sep}/\c K)$ of period $p$ and nilpotent class $<p$, 
$\c L$ is a profinite Lie $\F _p$-algebra and $G(\c L)$ is the 
profinite $p$-group 
obtained from $\c L$ via the 
Campbell-Hausdorff composition law. 
The identification $\pi $ is defined uniquely up to conjugation 
after choosing a 
suitable element $e\in\c L\otimes\c K$. 
\medskip 

b) We use the $\c P$-topology on $\c K$ to define the   
Lie subalgebra $\c L^{\c P}$ in $\c L$. This is 
$\c P$-topological algebra provided with a system of  
$\c P$-topological generators. The algebra 
$\c L^{\c P}$ is dense in $\c L$, i.e. the profinite completion of $\c L^{\c P}$ 
coincides with $\c L$. 
\medskip 

c) With respect to (defined up to conjugation) identifications 
of nilpotent Artin-Schreier theory $\pi :\c G_{<p}\simeq G(\c L)$ 
the algebra $\c L^{\c P}$ gives rise to 
a class of conjugated subgroups 
$\c G_{<p}^{\c P}:=\pi ^{-1}(\c L^{\c P})
%\in\op{cl}^{\c P}\c G_{<p}
$; 
the profinite completions of $\c G^{\c P}_{<p}$ coincide with $\c G_{<p}$. 
\medskip 

d) The subgroups $\c G_{<p}^{\c P}$ have  
$\c P$-topological systems of generators and could be used to study 
$N$-dimensional local field extensions $\c K'$ of $\c K$ in $\c K_{<p}$. 
More precisely, $\c H$ is an open subgroup in $\c G_{<p}$ 
(with respect to the Krull topology) 
iff 
$\c H^{\c P}:=\c G^{\c P}_{<p}\cap\c H$ is a  
$\c P$-open subgroup of finite index in $\c G^{\c P}_{<p}$.  
We have also 
$(\c G_{<p}:\c H)=(\c G^{\c P}_{<p}:\c H^{\c P})=[\c K':\c K]$, 
where $\c K'=\c K_{<p}^{\c H}$. 
In particular, $\c K'/\c K$ is Galois iff $\c H^{\c P}$ is 
normal in $\c G_{<p}^{\c P}$; in this case $\op{Gal}(\c K'/\c K)=
\c G^{\c P}_{<p}/\c H^{\c P}$.
\medskip 

e) Suppose $t=(t_1,\dots ,t_N)$ is a system of local 
parameters in $\c K$, $\m _{\c K}$ is the maximal ideal in 
the $N$-valuation ring $\c O_{\c K}$ of 
$\c K$, $\omega \in\m _{\c K}$, and for $1\leqslant m\leqslant N$, 
$h_{\omega }^{(m)}\in\Aut\,\c K$ are such that 
$h_{\omega }^{(m)}(t_i)=t_iE(\omega ^p)^{\delta _{mi}}$, 
where 
$E(X)$ is the Artin-Hasse exponential. Then all lifts 
of $h_{\omega }^{(m)}$, $1\leqslant m\leqslant N$, to $\c K_{<p}$ form 
a subgroup $\c G_{\omega }\subset\Aut\,\c K_{<p}$ 
containing $\c G_{<p}$. Let $\Gamma _{\omega }$ be the maximal quotient 
of $\c G_{\omega }$ of period $p$ and nilpotent class $<p$. 
If $\bar{\c G}$ is the image of $\c G_{<p}$ in $\Gamma _{\omega }$ 
we obtain the following short exact sequence of profinite $p$-groups 
$$1\To \Bar{\c G}\To \Gamma _{\omega }\To \langle h_{\omega }^{(1)}\rangle 
^{\Z /p}\times \dots \times \langle h_{\omega }^{(N)}\rangle ^{\Z /p}\To 1\, ,$$
the corresponding exact sequence of Lie $\F _p$-algebras 
(here $\Gamma _{\omega }=G(L_{\omega })$)
$$0\To \bar{\c L}\To L_{\omega }\To \prod _{1\leqslant m\leqslant N}
\F _ph_{\omega }^{(m)} \To 0\, ,$$
and define the appropriate dense subalgebra $L^{\c P}_{\omega }$ such that  
$$0\To \bar{\c L}^{\c P}\To L_{\omega }^{\c P}\To 
\prod _{1\leqslant m\leqslant N}
\F _pl_{\omega }^{(m)} \To 0\, .$$

f) We apply methods from \cite{Ab12, Ab13} to describe the 
structure of the Lie algebras $\bar{\c L}^{\c P}\otimes k$ 
and $L^{\c P}_{\omega }\otimes k$. In particular, for 
$1\leqslant m\leqslant N$, we 
obtain a recurrent procedure to recover the operators 
$\op{ad}\,\bar l^{(m)}_{\omega }$, where $\bar l_{\omega }^{(m)}$ are 
lifts of $l_{\omega }^{(m)}$ to $L_{\omega }$, and find explicit formula for 
$[\bar l_{\omega }^{(m_1)}, \bar l_{\omega }^{(m_2)}]\in\bar{\c L}$. 
%For the first algebra, we specify the structure of 
%$\Ker (\c L^{\c P}\To \bar{\c L}^{\c P})$ 
%in terms of generators $D_{an}$ of $\c L^{\c P}_k$: 
%we define the weight function $\op{wt}$ on generators $D_{an}$ and 
%prove that this kernel is the ideal of elements of total weight $\geqslant p$. 
%For the second algebra, we specify the lifts 
%$l_{\omega }^{(m)}\in L_{\omega }^{\c P}$ of 
%$h_{\omega }^{(m)}$ and arrange the recurrent procedure 
%to determine the operators $\ad\,l_{\omega }^{(m)}$ on $\bar{\c L}^{\c P}$ 
%and the commutators $[l_{\omega }^{(m_1)}, 
%l_{\omega }^{(m_2)}]\in\bar{\c L}^{\c P}$.  
These results are illustrated 
via explicit description of 
the structure of the Lie algebra $L_{\omega }^{\c P}$ modulo 
the ideal of third commutators.

g) We apply the results from f) to the explicit description of 
$\Gamma _{<p}=\Gamma /\Gamma ^pC_p(\Gamma )$, 
where $\Gamma $ is the Galois group of $N$-dimensional 
local field $K$ containing a non-trivial $p$-th root of unity $\zeta _1$. 
More precisely, we introduce a canonical class 
of conjugated dense subgroups $\Gamma ^{\c P}_{<p}$ in $\Gamma _{<p}$  
with $\c P$-topological systems of generators. Then we apply 
Scholl's construction of the field-of-norms functor  
to identify $\Gamma _{<p}^{\c P}$ with 
$\Gamma _{\omega }^{\c P}$, where 
$\omega \in\m _K$ is defined in terms related to the 
$p$-th root of unity $\zeta _1$. This result is illustrated 
in the case where $K=\Q _p(\zeta _1)\{\{x\}\}$. 
\medskip 

\subsection {Final remarks} \label{S0.4}
 a) In 1-dimensional case the Demushkin relation 
 depends only on the subgroup $\mu (K)$ of 
 roots of unity in $K^*$ and the degree $[K:\Q _p]$. 
 In the general case the structure of $\Gamma _{<p}$ 
 depends only on a special power series 
 constructed from $\zeta _1\in K$; this series appears in the $p$-adic 
 Hodge theory as the period of $\mathbb{G}_m$. In particular, the group 
 structure on $\Gamma _{<p}$ is a very weak invariant of the field $K$. 
 
%b) As a result of a), the structure of 
%$\Gamma _K(p)$ ``almost'' did not reflect the structure of $K$. 
%There was general opinion 
%(author's communications with I.Shafarevich, A.Weil, P.Deligne, H.Koch) 
%that the group structure on $\Gamma _K(p)$ should be related to the additional 
%structure given by the decreasing filtration of ramification subgroups 
%$\{\Gamma _K(p)^{(v)}\}_{v\geqslant 1}$.
%\medskip 

 b) In 1-dimensional case $\Gamma _K(p)$ (as well as $\Gamma _K$) 
 has very important 
 additional structure given by the decreasing filtration 
 of ramification subgroups $\Gamma _K(p)^{(v)}$, $v\geqslant 1$. 
 According to \cite{Ab11} the group $\Gamma _K(p)$ together with the 
 additional structure given by the 
 ramification filtration is an absolute invariant of $K$, 
 cf.\,also \cite{Mo, Ab6} in the context of the whole group $\Gamma _K$. 
 The papers \cite{Ab12,Ab13} contain the description of the group structure of 
 $\Gamma _{<p}$ together with the induced ramification filtration. 
 
 c) It would be natural to assume that most interesting 
 structures on $\Gamma _{<p}$ appear as completions of 
 structures defined in terms of the subgroups $\Gamma _{<p}^{\c P}$. 
 (In particular, we see that the group structure on 
 $\Gamma _{<p}$ is induced from $\Gamma _{<p}^{\c P}$.)
 As a result, such structures can be studied and described in terms of 
 generators and relations. In particular, it will be natural 
 to expect that the ramification subgroups introduced for 
 higher local fields in \cite{Zh3, Ab7, Ab8} satisfy this assumption. 
 In particular, in the case of 2-dimensional higher local field 
 $\c K$ of characteristic $p$ the description of ramification subgroups 
 in $\Gamma _{<p}$ modulo subgroup of third commutators from \cite{Ab8} 
 satisfies this assumption. 
 \medskip 
 
 d) Recently we found more substantial and natural 
 way to study the ramification filtration of $\Gamma _{<p}$ 
 in the 1-dimensional case, cf.\,\cite{Ab15}. 
 We expect that the techniques of generators and relations provided by this paper 
 will allow us to develop more substantial and clear 
 approach to  
 the proof of the local analog of the Grothendieck conjecture for all 
 higher local fields. 
 \medskip 
 
 e) There is still an open question in the description of $\Gamma _{<p}$:  
 we have not yet found explicitly the commutators 
 $[\bar l_{\omega }^{(m_1)}, \bar l_{\omega }^{(m_2)}]$. There is a 
 strong evidence that there are lifts $\bar l_{\omega }^{(m)}$ which 
 commute one with another: we verified this fact   
 modulo $C_4(\Gamma _{\omega })$ by direct computation. The existence of commuting 
 elements in sufficiently large Galois groups may have some relation to 
 anabelian geometry, cf. \cite{Bg}. 
 \medskip 
 
 f) Notice 
the paper \cite{Wb} where the case of the Galois group of 
2-dimensional fields with the first residue field of {\sc characteristic} 0 
was considered. We are not considering here such fields, but this result 
is not very far from the Demushkin one. The Galois group here appears 
as a profinite group with finitely many generators and one relation, i.e. is  
a group of Poincare type. 
\medskip

\subsection{Notation} \label{S0.5} 
Let $G$ be a topological group. For $s\geqslant 1$, denote 
by $C_s(G)$ the closure of its subgroup 
of $s$-th commutators. Here $C_1(G )=G $ and for $s\geqslant 2$, 
$C_s(G )$ is the closure of the commutator 
subgroup $(G ,C_{s-1}(G))$. Similarly, if $L$ is a (topological) Lie algebra 
over some ring $R$ then $C_s(L)$ is the closure of its $R$-submodule of 
commutators of order $\geqslant s$. 
If $M$ and $S$ are $R$-modules we denote very often by 
$M_S$ the extension of scalars 
$M\otimes _RS$. 
\medskip

\section{Constructions of nilpotent Artin-Schreier theory} \label{S1} 

In this section we review basic results of nilpotent Artin-Schreier theory, 
cf.\,\cite{Ab1,Ab2}. 
This theory allows us to work with $p$-extensions 
of fields of characteristic $p$ with Galois groups of nilpotent 
class $<p$. In these notes we use the simplest case of the theory 
involving Galois groups of period $p$. In other words, if 
$\op{char} K=p$ and $\Gamma =\Gal (K_{sep}/K)$ 
our approach allows us to work efficiently with subfields of 
 $K_{<p}:=K_{sep}^{\Gamma _{<p}}$, 
where $\Gamma _{<p}=\Gamma /\Gamma ^pC_p(\Gamma )$.

\subsection{Groups and Lie algebras of nilpotent class $<p$} \label{S1.1} 

The basic ingredient of the nilpotent Artin-Schreier 
theory is the equivalence of the category of 
$p$-groups of nilpotent class $s_0<p$ and the 
category of Lie $\Z _p$-algebras of the same nilpotent class. 
In the case of objects killed by $p$ this 
equivalence can be explained as follows. 

Let $L$ be a Lie  $\F _p$-algebra of nilpotent class $<p$, i.e. $C_p(L)=0$. 

Let $\frak A$ be an enveloping algebra  of $L$. Then there is a natural embedding 
$L\subset \frak A$, the elements of $L$ generate the augmentation ideal $J$ of $\frak A$ 
and we have a morphism of algebras $\Delta :\frak A\To \frak A\otimes \frak A$ 
uniquely determined by the 
conditions $\Delta (l)=l\otimes 1+1\otimes l$ for all $l\in L$.   
The Poincare-Birkhoff-Witt Theorem then implies:

--- $L\cap J^p=0$; 

--- $L\,\op{mod}\,J^p=\{a\,\op{mod}
\,J^p\ |\ \Delta (a)\equiv a\otimes 
1+1\otimes a\,\op{mod}\,(J\otimes 1+1\otimes J)^p\}\, ;$

--- the set $\wt{\exp}(L)\,\op{mod}J^p$ is identified with the set of all 
''diagonal elements`` $\op{mod}\,\op{deg}\,p$, i.e. with 
the set of all $a\in 1+J\,\op{mod}\,J^p$ such that 
$\Delta (a)
\equiv a\otimes a\,\op{mod}(J\otimes 1+1\otimes J)^p$ 
(here  $\wt{\exp}(x)=\sum _{0\leqslant i<p}x^i/i!$ 
is the truncated exponential). 
\medskip 

In particular,  there is a natural embedding 
$L\subset \frak A\,\op{mod}\,J^p$ and in terms of this embedding 
the Campbell-Hausdorff formula appears as    
$$(l_1,l_2)\mapsto l_1\circ l_2=
l_1+l_2+\frac{1}{2}[l_1,l_2]+\dots ,\ \ l_1,l_2\in L\, ,$$
where $\wt{\exp}(l_1)\,
\wt{\exp}(l_2)\equiv \wt{\exp}(l_1\circ l_2)\,\op{mod}\,J^p$.   
This composition law provides the set $L$ with 
a group structure and we denote this group by $G(L)$.   
The group $G(L)$  has period $p$ and nilpotent class $<p$. 
The correspondence 
$L\mapsto G(L)$ induces the   
equivalence of the category of $p$-groups of 
period $p$ and nilpotent class $s<p$ 
and the category of Lie 
$\Z /p$-algebras of the same nilpotent class $s$. 
This  
equivalence is naturally extended to the similar  categories of 
pro-finite Lie algebras and pro-finite $p$-groups.

\subsection{Nilpotent Artin-Schreier theory} \label{S1.2} 

Let $L$ be a finite Lie $\F _p$-algebra 
of 
nilpotent class $<p$. Consider the extensions of scalars $L_{ K}$ and 
$L_{sep}:=L_{K_{sep}}$. Then  the elements of 
$\Gamma =\Gal (K_{sep}/K)$ and  the Frobenius 
$\sigma $ act on $L_{sep}$ through the second factor, 
$L_{sep}|_{\sigma =\id}=L$ and $(L_{sep})^{\Gamma }=L_{K}$. 
If $e\in G(L_{K})$ then  the set 
$$\c F(e)=\{f\in G(L_{sep})\ |\ \sigma (f)=e\circ f\}$$  
is not empty and for any fixed $f\in \c F(e)$, the map  
$\tau\mapsto (-f)\circ \tau (f)$ is a continuous group homomorphism 
$\pi _f(e):\Gamma \To G(L)$.  The correspondence $e\mapsto\pi _f(e)$ has 
the following properties: 
\medskip 

a) if $f'\in\c F(e)$ then $f'=f\circ l$, where $l\in G(L)$; in particular, 
$\pi _f(e)$ and $\pi _{f'}(e)$ are conjugated via $l$; 
\medskip 

b) for any continuous group homomorphism $\pi :\Gamma \To G(L)$, 
there are $e\in G(L_{K})$ and 
$f\in \c F(e)$ such that $\pi _f(e)=\pi $;
\medskip 

c) for appropriate elements $e,e'\in G(L_{K})$,  
$f\in\c F(e)$ and $f'\in \c F(e')$, we have 
$\pi _f(e)=\pi _{f'}(e')$ iff 
there is an $x\in G(L_{K})$ such that $f'=x\circ f$ and, therefore,   
$e'=\sigma (x)\circ e\circ (-x)$. 
\medskip 

In the case of a profinite Lie algebra $L=\varprojlim\limits _{\alpha }L_{\alpha }$, 
where all $L_{\alpha }$ are finite Lie $\F _p$-algebras, consider 
$e=\varprojlim\limits _{\alpha }e_{\alpha }\in L_{K}$, where all 
$e_{\alpha }\in L_{\alpha  K}$. Then there is 
$f=\varprojlim\limits _{\alpha }f_{\alpha }\in \varprojlim\limits_{\alpha }
\c F(e_{\alpha })\subset L_{sep}$ (where all 
$f_{\alpha }\in\c F(e_{\alpha })$) and 
$\pi _f(e)=\varprojlim\limits _{\alpha }\pi _{f_{\alpha }}
(e_{\alpha })$ maps $\Gamma $ to 
$G(L)=\varprojlim\limits _{\alpha }G(L_{\alpha })$. 

\subsection{The diagonal element and 
abelian Artin-Schreier theory}\label{S1.3} 

Let $\bar K=K/(\sigma -\id )K$ and $M=\Hom _{\F _p\text{-lin}}(\bar K,\F _p)$. 
If $\bar K$ is 
provided with discrete topology 
(as an inductive limit of finite dimensional 
$\F _p$-subspaces), its dual $M$ has the pro-finite topology and 
$$M_{\bar K}=
\Hom _{\F _p\text {-lin}}(\bar K, \bar K)\, .$$
Let  
$\Pi :K\To\bar K$ be the natural projection 
and $e\in M_K=\Hom (\bar K,K)$ be such that $(\id _M\otimes \Pi )e=\id _{\bar K}$. 
Equivalently, let $S$ be a section of $\Pi $ and 
$e:=e_S=(\id _{M}\otimes S)\id _{\bar K}$.  

In notation of Sect.\,\ref{S1.2} the identification of the abelian Artin-Schreier theory 
$\pi ^{ab}:\Gamma _{<2}:=\Gamma /\Gamma ^pC_2(\Gamma )\simeq M$ 
can be obtained as follows: 
\medskip 

--- choose $f\in M_{sep}:=M_{K_{sep}}$ such 
that $\sigma f-f=e_S$ and for any 
$\tau\in \Gamma _{<2}$, let 
$\tau f-f=\pi ^{ab}(\tau )\in M_{sep}|_{\sigma =\id }= M$.  
\medskip 

\begin{remark} 
In the above formulas (and in similar situations below if there 
is no risk of confusion) 
we use the simpler notation  $\sigma $ and $\tau $ instead 
of $\id _M\otimes\sigma $ and $\id _{M_{sep}}\otimes \tau $. 
\end{remark}

The map $\pi ^{ab}$ does not depend on a choice of $f$: 
\medskip 

---  if $f_1\in M_{sep}$ is such that $\sigma f_1-f_1=e_S$ 
then $f_1-f\in M_{sep}|_{\sigma =\id }= M$ and 
$\tau f_1-f_1=\tau f-f$. 
\medskip 

The map $\pi ^{ab}$ also does not depend on a choice of $S$:  
\medskip 

--- if $S'$ is another section then 
there is $g\in M_K$ such that 
$e_{S'}-e_S=\sigma g-g$. Therefore, 
$f':=f+g$ satisfies the relation 
$\sigma f'-f'=e_{S'}$ and, 
$\tau f'-f'=\tau f-f$.

\subsection{Identifications $\pi _f(e):\Gamma _{<p}\simeq G(L)$}\label{S1.4}\ \

Let $\wt {L}$ be a free Lie $\F _p$-algebra with generating 
module $M$ and $L=\wt{L}/C_p(\wt{L})$. (Note that $M$ is a profinite limit of its 
finite quotients and $\wt{L}$ is the corresponding profinite 
limit of finite Lie algebras.) 
Consider the natural projection 
$$\op{pr}\otimes\Pi :L\otimes _{\F _p}K\To L/C_2(L)
\otimes_{\F _p}\bar K=M_{\bar K}$$ 
and 
set $\c E(L_K)=\{e\in L_K\ |\ (\op{pr}\otimes\Pi )e=\id _{\bar K}\}$. 
Agree to denote the image of $e$ in $M_K$ by $e_S$, where 
$S$ is the appropriate section of $\Pi $, cf. Sect.\,\ref{S1.3}. 

Choose $f\in \c F(e)$ 
and consider the group homomorphism $\pi _f(e):
\Gamma _{<p}\To G(L)$ such that 
for any $\tau\in \Gamma _{<p}$, $\pi _f(e)(\tau )=(-f)\circ \tau (f)$. 
Then $\pi _f(e)$ is a group isomorphism (use that $\Gamma $ 
is a free pro-$p$-group and 
 $\pi _f(e)\,\op{mod}\Gamma ^pC_2(\Gamma )$ is isomorphism by 
Sect.\,\ref{S1.3}). 

If $f'$ is another element from $\c F(e)$ then there is $l\in G(L)$ such that 
$f'=f\circ l$ and 
$\pi _{f'}(e)(\tau )=(-f')\circ \tau (f')=
(-l)\circ \pi _{f}(e)(\tau )\circ l$ is conjugated to $\pi _f(e)$. 
Study how $\pi _f(e)$ depends on a choice of $e\in\c E(L_K)$.     

\begin{Prop} \label{P1.1} If $e,e'\in\c E(L_K)$ then 
there is $x\in L_K$ and a section $A$ of the natural projection 
$\op{pr}:L\To L/C_2(L)=M$ such that 
$$e'=\sigma (x)
\circ (\c A\otimes \id _K)e\circ (-x)\, ,$$ 
where $\c A\in\op{Aut}_{\text{Lie}}L$ is a unique extension of $A$. 
\end{Prop}

\begin{proof} 

Let $\{l_{\alpha }\ |\ \alpha\in\c I\}$ be an 
$\F _p$-basis of $\bar K$. 
Let $\hat l_{\alpha }$, $\alpha\in\c I$,  be the dual 
(topological) basis for $M$, i.e. 
for any $\alpha _1,\alpha _2\in \c I$, 
$\hat l_{\alpha _1}(l_{\alpha _2})=\delta _{\alpha _1\alpha _2}$. 

Then we have the sections $S$ and $S'$ of $\Pi $ such that 
$e_S=\sum\limits _{\alpha }\hat l_{\alpha }\hat\otimes S(l_{\alpha })$ and 
$e_{S'}=\sum\limits _{\alpha }\hat l_{\alpha }\hat\otimes S'(l_{\alpha })$.

Apply induction on $r\geqslant 1$ to prove the existence of $x_r\in L_K$ and a 
section $A_r$ of the projection $L\To M$ such that 
 $$e'\equiv \sigma (x_r)\circ (\c A_r\otimes \id _K)
 e\circ (-x_r)\,\op{mod}\,C_{r+1}(L_K)\, ,$$
 where $\c A_r\in\op{Aut}_{\text{Lie}}L$ is such that $\c A_r|_M=A_r$. 
 
If $r=1$ take $A_1=\id _M$ and 
$x_1=\sum\limits _{\alpha }\hat l_{\alpha }\otimes x_{1\alpha }$, where 
all $x_{1\alpha }\in K$ are such that $S'(l_{\alpha })-S(l_{\alpha })
 =\sigma (x_{1\alpha })-x_{1\alpha }$. 
 
 If $r\geqslant 1$ and the required $x_r$ and $A_r$ exist then 
 there is $l_{r+1}\in C_{r+1}(L_K)$ such that 
 $e'\equiv\sigma x_r\circ (\c A_r\otimes\id _K)e
 \circ (-x_r)\circ l_{r+1}\,\op{mod}\,C_{r+2}(L_K)$. 
 
 Using that $K=\op{Im}(S)\oplus (\sigma -\id )K$ we can present $l_{r+1}$ as 
 $$l_{r+1}=l'+\sigma x'-x'$$
 where $l'=\sum\limits _{\alpha }c_{\alpha }\otimes S(l_{\alpha })$, 
 all $c_{\alpha }\in C_{r+1}(L)$ 
 and $x'\in C_{r+1}(L_K)$. 
 It remains to set 
 $A_{r+1}(\hat l_{\alpha })=A_r(\hat l_{\alpha })+c_{\alpha }$ and 
 $x_{r+1}=x_r+x'$. 
 The proposition is proved.  
\end{proof}

\begin{Cor} \label{C1.2} With above  
notation there is $f'\in\c F(e')$ such that 
for any $\tau\in\Gamma _{<p}$, $\pi _{f'}(e')(\tau )=\c A(\pi _f(e)(\tau ))$.
\end{Cor}

\begin{proof}
 Let $f'=x\circ (\c A\otimes \id _{sep})f$, then $f'\in\c F(e')$. 
 Indeed, 
$$\sigma (f')=\sigma x\circ 
(\c A\otimes\id _{sep})\sigma (f)=
 \sigma (x)\circ (\c A\otimes\id _K)e_S\circ (\c A\otimes\id _{sep})f$$
$$=\sigma (x)\circ (\c A\otimes\id _K)e\circ (-x)\circ f'=e'\circ f'\, .$$ 
 Therefore, for any $\tau\in \Gamma _{<p}$, $\pi _{f'}(e')(\tau )$ is equal to 
$$ 
(-f')\circ \tau (f')=
(\c A\otimes \id _{sep})((-f)\circ \tau (f))=\c A(\pi _f(e)(\tau ))\, .$$ 
\end{proof}

By the above corollary, a  
choice of $e\in\c E(L_K)$ determines the class $\pi _e$ of conjugated 
identifications $\{\pi _f(e)\ |\ f\in\c F(e)\}$ of 
$\Gamma _{<p}$ with $G(L)$. 
When $e$ is replaced by another $e'\in\c E(L_K)$ the new class 
of conjugated identifications 
$\pi _{e'}$ is obtained from $\pi _e$ via the composition with some automorphism 
$\c A=\c A(e,e')\in\op{Aut}_{\op{Lie}}(L)$ such that 
$\c A\equiv \id _L\,\op{mod}\,C_2(L)$.

\subsection{Compatibility with field extensions} \label{S1.5} 
 
Suppose $K'$ is a field extension of $K$ in $K_{sep}$. 
Consider the above defined objects: 
$M$, $L$, $e\in\c E(L_K)$,  
$f\in \c F(e)$ and $\pi =\pi _f(e):\Gamma _{<p}\simeq G(L)$ 
introduced in the context of the field $K$.  
Let $\Gamma '_{<p}$, $M'$, $L'$, $e'\in \c E(L'_{K'})$, 
$f'\in\c F(e')$ and 
$\pi ':\Gamma '_{<p}\simeq G(L')$ 
be the similar objects for the field $K'$. 

The embedding $\op{Gal}(K_{sep}/K')\To \op{Gal}(K_{sep}/K)$ induces 
a natural group homomorphism $\Theta :\Gamma '_{<p}\To \Gamma _{<p}$,  
which can be described 
in terms of the identifications $\pi $ and 
$\pi '$ as follows. 
\medskip

Consider $e\otimes _K1\in 
L_K\otimes _KK'=L_{K'}\supset M_{K'}=\op{Hom}(\bar K, K')$.

\begin{Prop} \label{P1.3} 
There is a morphism of Lie algebras 
$\c A:L'\To L$ and $x\in L_{K'}$ such that 

{\rm a)}  $e\otimes _K1=\sigma (x')\circ 
(\c A\otimes\id _{\c K})e'\circ (-x')$;
\medskip

{\rm b)} for any $\tau '\in\Gamma _{<p}'$, 
$\pi (\Theta (\tau '))=\c A(\pi '(\tau '))$;  
\medskip 

{\rm c)} if $K'\subset K_{<p}$ then $\pi (\Gal (K_{<p}/K'))=
\c A(L')$.
\end{Prop}

\begin{proof} 
Let $\{l'_{\alpha }\ |\ \alpha\in\c I'\}$ 
be an $\F _p$-basis of $\bar K'=K'/(\sigma -\id )K'$. 
Let $\hat l'_{\alpha }$, $\alpha\in\c I'$, 
be the dual (topological) 
basis for $M'$. Then for a suitable section 
$S'$ of $\Pi ':K'\To \bar K'$, we have 
$e_{S'}=\sum\limits_{\alpha }\hat l'_{\alpha }
\otimes S'(l'_{\alpha })$ and 
$\{S'(l'_{\alpha })|\ \alpha\in \c I '\}$ is a 
basis of $\op{Im}(S')\subset K'$. 
Proceeding similarly to the proof of Prop.\ref{P1.1} prove the 
existence of $x'\in L_{K'}$ and $\wt{l}_{\alpha }\in L$ such that 
$$e\otimes _K1=\sigma (x')\circ 
\left (\sum\limits _{\alpha }
\wt{l}_{\alpha }\otimes S'(l'_{\alpha })\right )
\circ (-x')\, .$$ 
If $A':M'\To L$ is a linear map such that 
for all $\alpha $, it holds 
$A'(\hat l'_{\alpha })=\wt{l}_{\alpha }$, the above relation appears 
in the following form 
$$
e\otimes _K1=\sigma (x')\circ (\c A'\otimes\id _{K'})e'
\circ (-x')\, ,
$$
where $\c A'$ is a unique 
morphism of Lie algebras $L'\To L$ such that $\c A'|_{M'}=A'$. 
As a result, the both   
$(-x')\circ f$ and $(\c A'\otimes\id _{sep})f'$ 
belong to $\c F((\c A'\otimes\id _{\c K'})e')\subset L_{sep}$.
So,  
there is $l\in L$ such that 
$$(-x')\circ (f\otimes _K1)=(\c A'\otimes\id _{sep})f'\circ l\, .$$

If $x=x'\circ l$ and $\c A=\op{Ad}\,l\cdot \c A'
\in \Hom _{\text{Lie}}(L',L)$ then the above equality can be rewritten as 
$$(-x)\circ (f\otimes _K1)=(\c A\otimes\id _{sep})f'\, .$$

In particular, we have 
$e\otimes _K1=\sigma (x)\circ (\c A\otimes\id _{K'})e'
\circ (-x)$ and for any $\tau '\in\Gamma '_{<p}$, it holds 
$\pi (\Theta (\tau '))= 
(-f)\circ \tau '(f)= 
(\c A\otimes\id _{K'})((-f')\circ \tau '(f'))=
\c A(\pi '(\tau '))$. 
 The proposition is proved. 
\end{proof}

\subsection{Lifts of $\phi\in\Aut K$}
\label{S1.6} \ \ 

As earlier, $e\in \c E(L_K)$, $f\in \c F(e)$, $\pi =\pi _f(e)
:\Gamma _{<p}\simeq G(L)$.  

Suppose $\phi\in\op{Aut}K$. 
We are going to describe (the lifts) 
$\phi _{<p}\in\Aut K_{<p}$ such that $\phi _{<p}|_K=\phi $. 

Let $\phi _*e:=(\id \otimes\phi )e\in L_K$. 
As earlier, for any given $\phi _{<p}$, 
establish the existence of 
$\c A=\c A(\phi _{<p})\in\op{Aut}_{\op{Lie}}L$ 
and $C=C(\phi _{<p})\in L_K$ such that 
\begin{equation} \label{E1.1} \phi _*e=
\sigma (C)\circ (\c A\otimes\id _K)
 e\circ (-C)\, .
\end{equation}

Let $\frak M(\phi )$ be the set of all pairs $(C,\c A)$ satisfying \eqref{E1.1}.  
Let 
$$\kappa :
\frak M(\phi )\To \{\phi _{<p}\in\Aut K_{<p}\ |\ \phi _{<p}|_{K}=\phi \}$$ 
be the map defined as follows. 

If $(C,\c A)\in\frak M(\phi )$ then $g:=C\circ (\c A
\otimes \id )f\in \c F(\phi _*e)$. If $\phi '_{<p}$ is a  
lift of $\phi $ then $(\id \otimes\phi '_{<p})f\in\c F(\phi _*e)$. 
Then for some $l\in L$, 
$$g=(\id \otimes\phi '_{<p})(f\circ l)=
(\id \otimes\phi '_{<p})(\id \otimes \pi ^{-1}l)f=
(\id \otimes (\phi '_{<p}\cdot\pi ^{-1}l))f\, $$
and the composition 
$\phi _{<p}:=\phi '_{<p}\cdot \pi ^{-1}l$ is a lift of $\phi $. 
It is easy to see that the lift $\phi _{<p}$ 
does not depend on the above choice of $\phi '_{<p}$. As a result, 
we can set $\kappa (C,\c A)=\phi _{<p}$.

 \begin{Prop} \label{P1.4} \ \ 

{\rm a)} If $\kappa (C,\c A)=\phi _{<p}$ then for any $\tau\in\Gamma _{<p}$, 
$\pi (\Ad (\phi _{<p})\tau)=\c A(\pi (\tau ))$.

{\rm b)} The map $\kappa $ is a bijection.
 \end{Prop}
 
 \begin{proof}  
With above notation 
$\pi _g(\phi _*e)(\tau )=(-g)\circ \tau (g)=\c A(\pi (\tau ))$. 
On the other hand, 
$g=(\id _L\otimes \phi _{<p})f$ implies that   
$$\pi _{g}(\phi _*e)(\tau )=
\phi _{<p}\left ((-f)\circ \phi_{<p}^{-1}\tau\phi _{<p}\,(f)\right )
  =\pi (\Ad (\phi _{<p})\tau )\, .$$  
The part b) is implied by the following three facts: 
\medskip 

$b_1$) {\it the map $\kappa $ is injective}.
\medskip 

Indeed, let $\kappa (C_1,\c A_1)=\kappa (C_2,\c A_2)$. 
Then 
$$C_1\circ (\c A_1\otimes\id )f=
C_2\circ (\c A_2\otimes\id )f\,.$$
This implies that 
$(\c A_1^{-1}\otimes\id )((-C_2)\circ C_1)\circ f=
(\c A_1^{-1}\c A_2\otimes\id )f$ and, therefore, for any 
$\tau\in\Gamma _{<p}$, it holds $\pi (\tau )=(\c A_1^{-1}\c A_2)(\pi \tau )$ 
(use that $(\c A_1^{-1}\otimes\id)((-C_2)\circ C_1)\in L_K$). As a result, 
$\c A_1^{-1}\c A_2=\id _{L}$ and $(C_1,\c A_1)=(C_2,\c A_2)$. 
\medskip 

$b_2$)\ 
$\{\phi _{<p}\ |\ \phi _{<p}|_{K}=\phi \}$ 
{\it is a principal homogeneous space over $\Gamma _{<p}$ with 
respect to the action $\phi _{<p}\mapsto 
\phi _{<p}\cdot\tau $,  $\tau\in\Gamma _{<p}$};
\medskip 

$b_3)$\ {\it the appropriate action of $\tau\in \Gamma _{<p}$ on the pair $(C,\c A)$ 
appears in the form 
$(C,\c A)\mapsto (C',\c A')$, where for $l_{\tau }:=\pi (e)\tau $, 
$C'=C\circ l_{\tau }$ and 
for any} $l\in L$, $\c A'(l)=(-l_{\tau })\circ \c A(l)\circ l_{\tau }
=(\Ad l_{\tau }\cdot\c A)(l)$. 
\medskip 

The proof of $b_2)$ and $b_3)$ is straightforward. 
For more details cf.\,\cite{Ab14}.
\end{proof}

The above formalism allows us to use the identification $\pi =\pi _f(e)$ to work 
with the group of all lifts $\phi _{<p}\in\Aut\,K_{<p}$ 
of automorphisms $\phi\in\Aut\,K$. Note that: 
\medskip 

-- if $(C',\c A')$ and $C'',\c A'')$ correspond to the lifts 
$\phi '_{<p}$ and, $\phi ''_{<p}$ 
of, resp., $\phi '$ and $\phi ''$, then 
the couple $(C'\circ (\c A'\otimes\id _K)C'', \c A'\c A'')$ corresponds to the lift 
$\phi '_{<p}\phi ''_{<p}$ of $\phi '\phi ''$;  
\medskip 

-- the elements $\pi (\tau )=l_{\tau }\in G(\c L)$ appear as a special case of a lift of 
$\id _K$ and correspond to the pairs $(l_{\tau }, \op{Ad}\,l_{\tau })$, where 
$\op{Ad}\,l_{\tau }:l\mapsto (-l_{\tau })\circ l\circ l_{\tau }$.

\section{Higher local fields and $\c P$-topology}\label{S2}

\subsection{Higher local fields}\label{S2.1}

Let $K$ be an $N$-dimensional local field, i.e.

-- if $N=0$ then $K$ is finite; 

-- if $N\geqslant 1$ then $K$ is a complete discrete valuation field 
such that its residue field is $(N-1)$-dimensional. 
\medskip 

If $N\geqslant 1$ 
then the residue field of $K$ is 
the 
{\sc first residue field} of $K$. It will be usually denoted by $K^{(1)}$. 
The corresponding 
valuation ring $O^{(1)}_{K}$ is the {\sc first valuation ring}. 
We agree 
to set by induction for all $1<m\leqslant N$, $K^{(m)}=K^{(m-1)(1)}$ --- 
this is the {\sc $m$-th residue field} of $K$. 
Note that $K^{(N)}$ is $0$-dimensional and, therefore, finite. 

Define the $N$-{\sc valuation ring} $\c O_K$ of $K$ by induction on $N$ as follows. 
If $N=0$ set $\c O_K=K$. If $N\geqslant 1$ and 
$\op{pr}:O^{(1)}_K\To K^{(1)}$ is the natural projection then 
set $\c O_K=\op{pr}^{-1}\c O_{K^{(1)}}$.

If $N\geqslant 1$ then $\pi :=\{\pi _1,\dots ,\pi _N\}$ is a 
{\sc system of local parameters} in $K$ if: 

-- $\pi _1$ is (the first) uniformizer in $K$;

 -- $\pi _2,\dots ,\pi _N\in O^{(1)}_{K}$ and their projections  
$\bar \pi _2,\dots ,\bar \pi _N$ to  
$K^{(1)}$ form a system of local parameters for $K^{(1)}$.

If $[E:K]<\infty $ then the structure of $N$-dimensional 
field on $K$ is uniquely extended to $E$ and vice versa. 
\medskip 
 
If $\op{char}\,K=p\ne 0$ and $\pi =\{\pi _1,\dots ,\pi _N\}$ is a 
system of local parameters in $K$ then 
$K$ 
appears as a field of iterated formal Laurent series 
$K=k((\pi _N))\dots ((\pi _1))$, where $k=K^{(N)}\simeq \F _q$ with $q=p^{N_0}$.  
This is a part of the classification result,\,\cite{Zh1}. 
In this case $K^{(1)}$ is identified with the subfield 
$k((\pi _N))\dots ((\pi _2))$ 
of $K$. More formally, there is a system of local parameters 
$\bar\pi :=\{\bar\pi _2,\dots ,\bar\pi _N\}$ 
in $K^{(1)}$ such that their lifts to $K$ form a subset  
$\{\pi _2,\dots ,\pi _N\}$ of $\pi $. We use the notation 
$\iota _{\bar\pi }$ for the corresponding embedding  of $K^{(1)}$ into $K$. 

If $\op{char} K=0$ we always assume 
that $\op{char}K^{(1)}=p>0$ -- such fields have most interesting 
arithmetical structure. 
The appropriate classification result for 
such fields $K$ can be presented as follows.

Let  
$K^{(1)}=k((\bar\pi _N))\dots ((\bar\pi _2))$ where  
$\bar \pi :=\{\bar \pi _2, \dots ,\bar \pi _N\}$ are 
local parameters for $K^{(1)}$. 
 The elements $\bar \pi _2,\dots ,\bar \pi _N$ form 
a $p$-basis in $K^{(1)}$, \cite{Bth}.  We can 
use this $p$-basis to construct an (absolutely unramified) lift 
$ K^{(1)}_{\bar \pi }$ of $K^{(1)}$ to characteristic 0,\,
\cite{Bth}. Recall that $K^{(1)}_{\bar \pi }$ is the fraction field of the ring 
$\varprojlim\limits_{m\in\N }O_m(K^{(1)})$, 
where 
$$O_m(K^{(1)})=W_m(\sigma ^{m-1} 
K^{(1)})[\pi _2,\dots ,\pi _N]\subset W_m(K^{(1)})$$  
are the lifts of $K^{(1)}$ modulo $p^m$. 
(Here $\pi _2,\dots ,\pi _N$ are the Teichmuller
representatives of $\bar \pi _2,\dots ,\bar \pi _N$.) 
The field $K^{(1)}_{\bar \pi }$ has a natural structure of $N$-dimensional 
local field of characteristic 0 with local parameters 
$\{p,\pi _2,\dots ,\pi _N\}$. Now the classification 
result from \cite{Zh1} states 

{\it $K$ is a finite field extension of $K^{(1)}_{\bar \pi }$. }
\medskip 

In particular, we obtain an analogue 
$\iota _{\bar\pi }:K_{\bar\pi }^{(1)}\To K$ of the above 
defined embedding $K^{(1)}\To K$ in the characteristic $p$ case. 

 Note also that,
 
 --- there is $\pi _1\in K$ such that $\{\pi _1,\pi _2,\dots ,\pi _N\}$ is 
 a system of local parameters in $K$;
 \medskip 
 
 --- the field $K$ contains a (unramified 1-dimensional) local field 
 $F_{ur}=\op{Frac}\,W(k)$;
 \medskip 
 
 --- the classification result from \cite{Zh1} states also the existence of 
 a finite totally ramified extension $F'$ of $F_{ur}$ 
 such that $K\subset F'K_{\bar \pi }^{(1)}$. 
\medskip 

\subsection{Definition and basic properties of $\c P$-topology.} \label{S2.2}\ \ 

The topology  
on $N$-dimensional local field $K$ (we refer to  
it as the $\c P$-topology) can be introduced as follows, \cite {Pa2, Zh1, MZh}. 
\medskip 

If $N=0$ then the $\c P$-topology on $K$ is  discrete. 
\medskip 

Suppose $N\geqslant 1$ and $\pi =\{\pi _1,\dots ,\pi _N\}$ 
is a system of local parameters in $K$. 
Then the $\c P$-topology on $K$ is introduced by 
induction on $N$ via the following properties: 
\medskip 

\begin{enumerate}
\item 
any $\xi\in K$ can be uniquely presented 
as $\c P$-convergent series 
$$
\xi =\sum _{a}[\alpha _{\bar a}]\pi _1^{a_1}\dots \pi _{N}^{a_N}
$$
where the indices $a=(a_1,\dots ,a_N)\in\Z ^N$, all $\alpha _{a}\in k$,  
$[\alpha _{a}]=\alpha _{a}$ if $\op{char}K=p$ and 
$[\alpha _{\bar a}]$ are the Teichmuller representatives 
of $\alpha _{a}$ 
in 
$W(k)\subset K^{(1)}_{\bar\pi }$ if $\op{char}K=0$;
\medskip

\item 
the $\c P$-convergence property of 
$\xi $ means the existence of integers  
$A_i(a_1,\dots a_{i-1})\in\Z $ with $1\leqslant i\leqslant N$,  
such that: 
{\it if 
$\alpha _{a}\ne 0$ then} 
$$a_1\geqslant A_1, 
a_2\geqslant A_2(a_1), \ldots , 
a_N\geqslant A_N(a_1,\dots ,a_{N-1})\, ;$$
\medskip 

\item   if $\bar{\pi }:=\{\bar\pi _2,\dots ,\bar\pi _n\}$ is 
a system of local parameters 
for $K^{(1)}$ used to define the $\c P$-topology on $K^{(1)}$ 
then $\iota _{\bar\pi }$ induces 
the map  
$$s_{\pi }:\sum _{\bar a=(a_2,\dots ,a_N)}
[\alpha _{a}]\bar\pi _2^{a_2}\dots \bar\pi _N^{a_N}
\mapsto \sum _{\bar a=(0,a_2,\dots ,a_N)}[\alpha _{a}]\pi _2^{a_2}
\dots \pi _N^{a_N}$$
which is a  
$\c P$-continuous (set-theoretic) section  
$s:K^{(1)}
\To O_K^{(1)}$ of the natural 
projection $O_K^{(1)}\To K^{(1)}$. 
\end{enumerate}
\medskip

The above properties imply that: 
\medskip 

--- a)\ a base of $\c P$-open subsets $\c U_{\pi }(K)$ in $K$ consists of 
the subsets 
$\sum _{b\in\Z }\pi _1^{b}s _{\pi }(U_b)$, 
where all $U_b\in\c U_{\bar\pi }(K^{(1)})$ and 
for $b\gg 0$, $U_b= K^{(1)}$;
\medskip 
 
--- b) a base $\c C_{\pi }(K)$ of sequentially compact (closed) subsets 
in $K$ consists of  
$\sum _{b\in\Z }\pi _1^{b}s _{\pi }
(C_b)$ such that all $C_b\in\c C_{\bar\pi }(K^{(1)})$  
and for $b\ll 0$, $C_b=0$; 
\medskip 

--- c) if $\pi '=\{\pi _1',\dots ,\pi _N'\}$ is another system 
of local parameters for $K$ then the appropriate 
analogs of the above properties a)-c) also hold (i.e. 
the concept of $\c P$-topology 
does not depend on the original choice of local parameters in $K$);
\medskip 

--- d) if $[E:K]=n$ and the identification of $K$-vector spaces $E=K^n$ 
is induced by a choice of some 
$K$-basis in $E$ then $\{U^n\ |\ U\in\c U_{\pi }(K)\}$ is a base of 
$\c P$-open subsets in $L$; 
similarly, $\{C^n\ |\ C\in\c C_{\pi }(K)\}$ is a base of 
sequentially $\c P$-compact subsets in $E$;  
\medskip 

--- e) if $C_1,C_2\subset \c C_{\pi }(K)$ then $C_1C_2$ is also 
sequentially compact (i.e. there is $C\in\c C_{\pi }(K)$ 
such that $C_1C_2\subset C$);
\medskip 

--- f) $K$ is a $\c P$-topological additive group but not a 
$\c P$-topological field; however,  
$K=\varinjlim\limits _{C\in\c C_{\pi }(K)}C$ and the 
multiplication $C\times K\To K$ is $\c P$-continuous 
(i.e. 
for any $U\in\c U_{\pi }(K)$ there is an $U'\in\c U_{\pi }(K)$ 
such that $CU'\subset U$).
\medskip 

Note that the subset of $K$ consisting of the series $\xi $ from above item (1) 
satisfying the condition:
\medskip 

{\it if $\alpha _a\ne 0$ then $a_1\geqslant A_1, 
a_2\geqslant A_2(a_1), \ldots , 
a_N\geqslant A_N(a_1,\dots ,a_{N-1})$ 
(with a fixed choice of 
$A_i(a_1,\dots ,a_{i-1})$, $1\leqslant i\leqslant N$)}
\medskip 
\newline 
is sequentially compact. The family of all such 
subsets (with a fixed choice of a system of 
local parameters $\pi =\{\pi _1,\dots ,\pi _N\}$) forms 
the base $\c C_{\pi }(K)$. 

We can similarly describe the base $\c U_{\pi }(K)$: 
\medskip 

{\it $U\in\c U_{\pi }(K)$ iff there are  
$B_1, B_2(a_1),\dots , B_N(a_1,\dots ,a_{N-1})\in\Z $ 
(depending on $U$) such that 
$\xi\in U$ are characterized by the condition:
\medskip

if $a_1<B_1$, $a_2<B_2(a_1)$,  
\dots , $a_N<B_N(a_1,\dots ,a_{N-1})$ then $\alpha _a=0$.} 
\medskip

\subsection{$\c P$-topology in characteristic $p$} \label{S2.3} \ \ 

Assume that $K=\c K$ has characteristic $p$ and has a system of 
local parameters $t=\{t_1,\dots ,t_N\}$. 
We will use the simpler notation 
$\c U(\c K)$ and $\c C(\c K)$ instead of $\c U_{t}(\c K)$ and $\c C_{t}(\c K)$ 
when working with this fixed system of local parameters $t$.  
Note that all  
$C\in\c C(\c K)$ and $U\in\c U(\c K)$ are $k$-linear vector spaces 
(where $k=\c K^{(N)}$) and their elements 
appear as (some) formal $k$-linear combinations of the monomials 
$t^{a}:=t_1^{a_1}\dots t_N^{a_N}$, where all $a=(a_1,\dots ,a_N)\in\Z ^N$. 
\medskip 

Let $\c I=\c I(\c K)$ and $\c J=\c J(\c K)$ be the sets of indices such that 
$\c C(\c K)=\{C_{\alpha }\ |\ \alpha\in\c I\}$ and 
$\c U(\c K)=\{U_\beta\ |\ \beta\in\c J\}$. 
\medskip 

It is easy to see by induction on $N$ that 
for any $\alpha\in\c I$ and $\beta\in\c J$, 
\begin{equation} \label{E1.2}
\dim _{k}C_{\alpha }/C_{\alpha }\cap U_{\beta }<\infty \,.
\end{equation}
Note that $C_{\alpha }/C_{\alpha }\cap U_{\beta }$ are provided with the 
$k$-bases consisting of the monomials $t^{a}\,\op{mod}\,C_{\alpha }\cap U_{\beta }$ 
such that $t^{a}\in C_{\alpha }\setminus U_{\beta }$. These bases 
are compatible with respect to different choices of $\alpha $ and $\beta $. 
Therefore, $\{t^{a}\ |\ t^{a}\in C_{\alpha }\}$ is a 
$\c P$-topological $k$-basis in $C_{\alpha }$,   
a base of $\c P$-open neibourhoods in $C_{\alpha }$ 
consists of $k$-vector subspaces 
containing almost all elements of this basis and 
$C_{\alpha }=\varprojlim\limits _{\beta }
C_{\alpha }/C_{\alpha }\cap U_{\beta }$.

Let $\c N_{\alpha\beta }=
(C_{\alpha }/C_{\alpha }\cap U_{\beta })^D$ be the dual $k$-vector space 
for $C_{\alpha }/C_{\alpha }\cap U_{\beta }$. Then 
$\dim _k\c N_{\alpha\beta }<\infty $, 
$\c N_{\alpha\beta }^D=C_{\alpha }/C_{\alpha }\cap U_{\beta }$ and  
the spaces $\c N_{\alpha\beta }$ are provided with compatible (dual) 
$k$-bases 
$$\{T_{a}\ |\ t^{-a}\in C_{\alpha }\setminus U_{\beta }\}\, ,$$ 
where   
for any $t^{-b}$ with $b\in\Z ^N$,   
$T_{a}(t^{-b})=\delta _{ab}$\,. 

Note that 
$\c N_{\alpha }^{\c P}:=\Hom _{\c P\text{-cont}}(C_{\alpha },k)
=\varinjlim\limits_{\beta }
\Hom (\c N^D_{\alpha\beta },k)=\varinjlim\limits _{\beta }
\c N_{\alpha\beta }\, $
has a $k$-basis 
$\{T_{a}\ |\ t^{-a}\in C_{\alpha }\}$. 
Therefore, $C_{\alpha }$ is the set of all formal  
$k$-linear combinations of the appropriate monomials $t^{a}$ and 
we have a natural identification 
$C_{\alpha }=\c N_{\alpha }^{\c P\,D}$.  
\medskip 

Consider the $\F _p$-vector spaces $\c E^{\c P}_{\alpha }:=
\Hom _{\c P\text{-cont}}(C_{\alpha },\F _p)$.  
Then  
$$\c E^{\c P}_{\alpha }\otimes k=
\Hom _{\F _p\text{-lin},\c P\text{-cont}}(C_{\alpha }, k)
=\oplus _{n\in\Z /N_0}\c N^{\c P}_{\alpha }\otimes _kk^{(n)}\,,$$
where $k^{(n)}$ is the  
(twisted) $k$-module $k\otimes _{\sigma ^n}k$. In this identification 
the Frobenius 
$\sigma $ acts through the second factor 
on the left-hand side and shifts $\Z/N_0$-summands  
by +1 on the right-hand side. 

In particular, the extensions of scalars 
$\c E^{\c P}_{\alpha }\otimes k$ have the $k$-bases 
$$\{T_{an}=T_{a}\otimes 1^{(n)}\ |\ t^{-a}
\in C_{\alpha }, n\in\Z /N_0\}$$ 
which are compatible on $\alpha\in\c I$. 
So,  
$\{T_{an}\ |\ a\in\Z ^N, n\in\Z /N_0\}$ 
is a topological basis 
for $\c E^{\c P}\otimes k$, where 
$\c E^{\c P}:=\Hom _{\c P\text{-cont}}(\c K,\F _p)=
\varprojlim\limits_{\alpha} \c E^{\c P}_{\alpha}$.  

\begin{Prop}\label{P2.1} 
Let 
 $\c E:=\Hom (\c K,\F _p)$. Then $\c E=\c E^{\c P\,DD}$ 
 (the double dual $\F _p$-vector space for $\c E^{\c P}$).
\end{Prop}

\begin{proof} We verify this on the level 
of extensions of scalars as follows:  
$$\Hom (\c K,\F _p)\otimes k=\varprojlim\limits_{\alpha }
\Hom (C_{\alpha },\F _p)\otimes k=$$
$$\varprojlim\limits_{\alpha }\underset{n\in\Z /N_0}\oplus 
\Hom _{k}(C_{\alpha }, k)\otimes _kk^{(n)}=
\varprojlim\limits_{\alpha }\underset{n\in\Z /N_0}\oplus 
\c N_{\alpha }^{DD}\otimes _kk^{(n)}$$
$$=\varprojlim\limits_{\alpha }
\left (\underset{n\in\Z /N_0}\oplus\c N_{\alpha }\otimes _kk^{(n)}\right )^{DD}
=\varprojlim\limits_{\alpha }(\c E^{\c P}_{\alpha }
\otimes k)^{DD}=\c E^{\c P\,DD}\otimes k\,.$$
\end{proof} 

\begin{Cor}\label{C2.2}
 The vector space $\Hom (\c K,\F _p)$ is 
 the profinite completion 
 of its subspace $\Hom _{\c P\text{-cont}}(\c K,\F _p)$.
\end{Cor}

\begin{proof}
 Use that if $L$ is a $\F _p$-linear space then $L^{DD}$ is 
 canonically isomorphic to the profinite completion of $L$. We sketch briefly the proof 
 of this fact extracted from \cite{CB}. 
 
 Suppose $L^D=\varinjlim Y_{\alpha }$, where all $Y_{\alpha }$ are finite dimensional 
 vector subspaces in $L^D$. Then 
 $L^{DD}=\varprojlim Y_{\alpha }^D$. 
 Note that $Y_{\alpha }\mapsto \op{Ann}(Y_{\alpha })\subset L$ 
 is a bijection between the set of finite dimensional subspaces in $L^D$ and the set of 
 finite codimensional subspaces in $L$, and 
 $Y_{\alpha }^D\simeq L/\op{Ann}(Y_{\alpha })$. 
 Therefore, $L^{DD}\simeq \varprojlim L/Z_{\alpha }$ where 
 $Z_{\alpha }=\op{Ann}Y_{\alpha }$ runs over the set of all finite 
 codimensional subspaces in $L$. 
\end{proof}

\medskip

\section{The group $\c G_{<p}^{\c P}$} \label{S3}

\subsection{Frobenius and $\c P$-topology} \label{S3.1} \ \ 

Let $\c K$ be an $N$-dimensional local field of characteristic $p$. 
The quotient $\bar{\c K}=\c K/(\sigma -\id )\c K$ 
can be provided with the induced $\c P$-topological structure such that 
the projection $\Pi :\c K\To\bar{\c K}$ is open. 
Choose a system of local parameters $t=\{t_1,\dots ,t_N\}$ in $\c K$ and let 
$\c C(\c K)=\{C_{\alpha }\ |\ \alpha\in\c I\}$ and 
$\c U(\c K)=\{U_{\beta }\ |\ \beta\in\c J\}$ 
be the corresponding bases of sequentially compact and 
open subsets in $\c K$ from Sect.\,\ref{S2.3}. 
Then the images $\bar C_{\alpha }=\Pi (C_{\alpha })$ and 
$\bar U_{\beta }=\Pi (U_{\beta })$ 
form the corresponding bases 
for $\bar{\c K}$. 

Choose $\alpha _0\in k$ with the absolute trace $\op{Tr}_{k/\F _p}\alpha _0=1$.

Define $\F _p$-linear operators 
$\c S,\c R:\c K\To \c K$ as follows. 
\medskip 

Suppose $\alpha\in k^*$. 

If $a\in\Z ^N$, 
$a>\bar 0=(0,\dots ,0)\in\Z ^N$ then set $\c S(t^{a}\alpha )=0$ and $\c R(t^{a}\alpha )=
-\sum _{i\geqslant 0}\sigma ^i(t^{a}\alpha )$.

For $a=\bar 0$, set $\c S(\alpha )=\alpha _0\op{Tr}_{k/\F _p}\alpha $  
and  $\c R(\alpha )=\sum _{0\leqslant j<i<N_0}
(\sigma ^j\alpha _0)\sigma ^i\alpha $. 

If $a=-a_1p^m<\bar 0$ with $a_1\in\Z ^+_N(p)$ set 
$\c S(t^a\alpha )=t^{-a_1}\sigma ^{-m}\alpha $ and 
$\c R(t^a\alpha )=\sum _{1\leqslant i\leqslant m}\sigma ^{-i}(t^a\alpha )$. 
\medskip 

For $b =\hskip-5pt\sum\limits _{a\in\Z ^N}
\alpha _{a}t^{a}\in\c K$, set 
$\c S(b)=\hskip-5pt\sum\limits _{a\in\Z ^N}
\c S(\alpha _{a}t^{a})$,  
$\c R(b)=\hskip-5pt\sum\limits _{a\in\Z ^N}
\c R(\alpha _{a}t^{a})$. 

The proof of the following proposition is straightforward. 
It uses just that $\sigma :\c K\To \c K$ is $\c P$-continuous  and $\c K$ is a 
$\c P$-topological group with respect to addition. 
\begin{Prop} \label{P3.1} 

{\rm a)} $\c R$ and $\c S$ are $\c P$-continuous. 

{\rm b)} For any $b\in\c K$,  
 $b=\c S(b)+(\sigma -\id )\c R(b)$.  
\end{Prop}

Notice that $\c S^2=\c S$, $\c R^2=\c R$ and $\c R\c S=\c S\c R=0$. 
In particular, Prop.\,\ref{P3.1} implies
that the elements $b\in\c K$  
can be uniquely presented  modulo $(\sigma -\id )\c K$ in the following form 
\begin{equation} \label{E3.1} 
\sum\limits _{a\in\Z _N^+(p) }\gamma _{a}
t^{-a}+\gamma _{\bar 0}\alpha _0\,
\end{equation} 
where all $\gamma _{a}\in k$ and $\gamma _{\bar 0}\in\F _p$. 
We have also the following proposition. 

\begin{Prop} \label{P3.2} 
 
{\rm a)} The morphism $\Pi (b)\mapsto \c S(b)$, where $b\in\c K$, 
defines a $\c P$-continuous section   
$S_{t,\alpha _0}:\bar{\c K}\To \c K$ of $\Pi $ 
 such that $S_{t,\alpha _0}\Pi (\alpha t^{-a})=\alpha t^{-a}$ 
if $a\in\Z _N^+(p)$, $\alpha\in k$, 
and $S_{t,\alpha _0}(k/(\sigma -\id )k)=\F _p\alpha _0\subset k$.
\medskip 

{\rm b)} For a $\c P$-continuous section $S$ of $\Pi $ there is a 
$\c P$-continuous map $R_S:\c K\To\c K$ such that 
for any $b\in\c K$, 
$b =S(\Pi (b))+(\sigma -\id )R_S(b)$.
\end{Prop} 

\begin{proof} 
 Item a) follows from Prop.\,\ref{P3.1}. For item b), just notice that 
 $$b=\c S(b)+(\sigma -\id )\c R(b)=
 S\Pi (b)+(\c S-S\Pi )b+(\sigma -\id )\c R(b)=$$
 $$S\Pi (b)+(\sigma -\id )(\c R(\c S-S\Pi) +\c R)(b)=
 S\Pi (b)+(\sigma -\id )\c R(b-S\Pi b)$$
 and 
 $R_S:=\c R\,(\id -S\Pi )$ is $\c P$-continuous. 
\end{proof}

\subsection{$\c P$-topological module $\bar{\c K}^{D\c P}$.}\label{S3.2}  

Proceed similarly to Sect.\,\ref{S2.3} by setting for all 
$\alpha\in\c I$ and $\beta\in\c J$: 
\medskip 

1) $\bar{\c K}^D_{\alpha\beta }:=
\Hom (\bar C_{\alpha }/\bar C_{\alpha }\cap\bar U_{\beta },\F _p)=
\Hom _{\c P\text{-cont}}(\bar C_{\alpha }
/\bar C_{\alpha }\cap\bar U_{\beta },\F _p)\,;$
\medskip 

2) $\bar{\c K}^{D\c P}_{\alpha }:=
\varinjlim\limits _{\beta }\bar{\c K}^D_{\alpha\beta }=
\Hom _{\c P\text{-cont}}(\bar C_{\alpha },\F _p)\, ;$
\medskip 

3) $\bar{\c K}^{D\c P}:=\varprojlim\limits_{\alpha }\bar{\c K}^{D\c P}_{\alpha }
=\Hom _{\c P\text{-cont}}
(\bar{\c K},\F _p)\, .$
\medskip

\begin{remark} 
 1) For any $\alpha $ and $\beta $, $\dim _k\bar C_{\alpha }/(\bar C_{\alpha }
 \cap U_{\beta })<\infty $. 
 \medskip 
 
 2) $\bar{\c K}_{\alpha }^D=
\Hom (\bar C_{\alpha },\F _p)=
\left (\bar{\c K}_{\alpha }^{D\c P}\right )^{DD}\,$, in particular, $\bar{\c K}_{\alpha }^D$ is the profinite 
completion of $\bar{\c K}_{\alpha }^{D\c P}$.
\medskip 

3) It follows from 2) that $\bar{\c K}^D=
\varprojlim\limits _{\alpha }
\left (\bar{\c K}_{\alpha }^{D\c P}\right )^{\,DD}=
\left (\bar{\c K}^{D\c P}\right )^{DD}\,$ is 
the profinite completion of $\bar{\c K}^{D\c P}$.
\end{remark}

Define the $\c P$-topology on $\bar{\c K}^{D\c P}$ as 
the projective limit of discrete topologies on all $\bar{\c K}^{D\c P}_{\alpha }$. 

Set for any 
$\alpha\in \c I$ and $\beta\in\c J$, 
\medskip 

$U^D_{\alpha }:=
\{u^D\in\bar{\c K}^{D\c P}\ |\ u^D(\bar C_{\alpha })=0\}=\op{Ann}(\bar C_{\alpha })$.  
\medskip 

$C^D_{\beta }=\{c^D\in\bar{\c K}^{D\c P}\ |\ 
c^D(\bar U_{\beta })=0\}=\op{Ann}(\bar U_{\beta })\, .$ 
\medskip 

Then in $\bar{\c K}^{D\c P}$: 
\medskip

a)  ${\c U}(\bar{\c K}^{D\c P}):=\{U _{\alpha }^D\ |\ \alpha\in\c I\}$ 
is a base of open neighborhoods;
\medskip 
 
b) ${\c C}(\bar{\c K}^{D\c P})=\{C_{\beta }^D\ |\ \beta\in\c J\}$ 
is a base of sequentially compact subsets; 
\medskip  

c) $\bar{\c K}^{D\c P}=
\varinjlim\limits _{\beta }C^D_{\beta }$ and for any $\beta\in\c J$, 
$C^D_{\beta }=\varprojlim\limits _{\alpha }
C^D_{\beta }/C^D_{\beta }\cap U^D_{\alpha }$;
\medskip 

d) for any $\alpha\in\c I$, $\bar{\c K}^{D\c P}_{\alpha }=
\bar{\c K}^{D\c P}/U^D_{\alpha }$.
\medskip 

These properties are implied easily via the following observations.

Let $\{\alpha _i\ |\ 1\leqslant i\leqslant N\}$ be a basis of $k$ over $\F _p$. 
Consider the set 
\begin{equation} \label{E3.2}
 \{\alpha _it^{-a}\ |\ 1\leqslant i\leqslant N, 
 a\in\Z ^+_N(p)\}\cup\{\alpha _{\bar 0}\}\, .
\end{equation}

Then:
\medskip 

--- for any $\alpha \in\c I$, there is a 
subset of \eqref{E3.2} which forms a $\c P$-topological 
basis of $\bar{C}_{\alpha }$;
\medskip 

--- for any $\beta \in\c J$, there is a subset of \eqref{E3.2} which forms a 
$\c P$-topological basis of $\bar{U}_{\beta }$. 
\medskip 

Let 
\begin{equation} \label{E3.3} 
\{D_{a}^{(i)}\ |\ a\in\Z _N^+(p), 
1\leqslant i\leqslant N\}\cup\{D_{\bar 0}\}
\end{equation} 
be the dual system of elements of 
$\bar{\c K}^{D\c P}$ for system \eqref{E3.2}.

Then for any $\alpha \in\c I$, there is a subset of \eqref{E3.3} which forms 
a $\c P$-topological $\F _p$-basis of $U^D_{\alpha }$. Similarly, for any  
$\beta \in\c J$, there is a subset of \eqref{E3.3} which 
forms a $\c P$-topological $\F _p$-basis 
of $C_{\beta }^D$. 

As a result, the pairing 
$\bar{\c K}\times \bar{\c K}^{D\c P}\To\F _p$ is a perfect 
pairing of $\c P$-topological modules. This pairing identifies $\bar{\c K}^{D\c P}$ with 
$\Hom _{\c P{\text{-cont}}}(\bar{\c K},\F _p)$.  
\medskip 

Consider the presentations of elements from  
$\c S(\bar{\c K})\subset\c K$ in the form \eqref{E3.1}. 
For $a\in\Z _N^+(p)$ and $n\in\Z /N_0$, 
let $D_{an}\in\bar{\c K}^{D\c P}\otimes _{\F _p}k$ 
be such that 
$D_{an}(\Pi (\gamma _{a}t^{-a}))=\sigma ^n\gamma _{a}$ 
and $D_{an}(\Pi (\alpha _{\bar 0}))=0$. 
If $D_{\bar 0}\in\bar{\c K}^{D\c P}$ is the element appeared in 
\eqref{E3.3}, then  
$D_{\bar 0}(\gamma _{a}t^{-a})=0$ 
and $D_{\bar 0}(\alpha _{\bar 0})=1$. 

The elements of the set 
\begin{equation} \label{E3.4} 
\c D:=\{D_{an}\ |\ a\in\Z ^+_N(p), n\in\Z /N_0\}\cup\{D_{\bar 0}\}
\end{equation}
form a $\c P$-topological basis for $\bar{\c K}^{D\c P}_k$. In particular: 
\medskip 

\begin{enumerate} 
 \item 
the elements of $\bar{\c K}^{D\c P}_k$ appear uniquely   
as $\c P$-convergent series 
$$\sum _{\substack{a\in\Z _N^+(p)\\ n\in\Z /N_0}}
\gamma _{an}D_{an}
+\gamma _{\bar 0}D_{\bar 0}\, ,$$
where all $\gamma _{an}$ and $\gamma _{\bar 0}$ run over $k$;
\medskip 

\item  the appropriate subsets of \eqref{E3.4} provide us with compatible 
$k$-bases for $\bar{\c K}^D_{\alpha\beta }$ and $\bar{\c K}^{D\c P}_{\alpha }$;
\medskip

\item  the elements of $\bar{\c K}^{D\c P}$ can be presented uniquely  
as $\c P$-convergent series 
$$\sum _{\substack{a\in\Z _N^+(p)\\ n\in\Z /N_0}}
\sigma ^n(\gamma _{a})D_{an}
+\gamma _{\bar 0}D_{\bar 0}\, ,$$
where $\gamma _{\bar 0}\in \F _p$ and for $a\ne\bar 0$,  
$\gamma _{a}\in k$. 
\end{enumerate}
\medskip 

\begin{remark}
 The condition of $\c P$-convergence in (1) means that 
 for any $\alpha\in\c I$, 
 $\{\gamma _{an}\ne 0\ |\ t^{-a}\in \bar C_{\alpha }\}$ is finite. 
 Similar condition holds in (3) (where $\gamma _{an}$ should be replaced by 
 $\sigma ^n\gamma _{a}$).  
\end{remark}

Let $\otimes ^{\c P}$ be the 
$\c P$-topological tensor product. Consider 
$$\bar{\c K}^{D\c P}_{\c K}:=
\Hom _{\c P\text{-cont}}(\bar{\c K}, \c K)
=\bar{\c K}^{D\c P}\otimes ^{\c P}_{\F _p}\c K=
\bar{\c K}^{D\c P}_k\otimes ^{\c P}_k\c K\, .$$
The following property is straightforward.

\begin{Prop} \label {P3.3} 
The elements of $\bar{\c K}^{D\c P}_{\c K}$ can be presented uniquely  
as $\c P$-convergent sums  
$\sum \limits _{a\in Z^N}m_{a}t^{-a}$ 
with coefficients $m_{a}\in \bar{\c K}^{D\c P}_k$.    
\end{Prop}

 \begin{remark} 
The condition of $\c P$-convergency in Prop.\,\ref{P3.3} 
means that for any $\alpha\in\c I$, $\beta\in\c J$,  
$\{m_{a}\ne 0\ |\ t^{-a}\notin U_{\beta }, 
m_{a}\notin U^D_{\alpha }\}$ 
is finite. 
\end{remark}

\subsection{Lie algebras $\c L$ and $\c L^{\c P}$} \label{S3.3} \ \ 

Let $\wt{\c L}$ be a free profinite Lie algebra over $\F _p$ 
with the (profinite) module 
of free generators $\bar{\c K}^D=\Hom (\bar{\c K}, \F _p)$.  
Let $\c L=\wt{\c L}/C_p(\wt{\c L})$. Then $\c L$ is a projective limit of 
finite Lie $\F _p$-algebras generated by the finite quotients of $\bar{\c K}^D$. 

Let $\wt{\c L}^{\c P}$, resp., $\c L^{\c P}$,  be the Lie subalgebra in 
$\wt{\c L}$, resp. in $\c L$,  
generated by the elements of $\bar{\c K}^{D\c P}=\Hom _{\c P\text{-cont}}
(\bar{\c K},\F _p)\subset\bar{\c K}^D$. 
Then $C_p(\wt{\c L}^{\c P})=\wt{\c L}^{\c P}\cap C_p(\wt{\c L})$ 
(use that the profinite completion of $\bar{\c K}^{D\c P}$ is $\bar{\c K}^D$) and 
$\c L^{\c P}=\wt{\c L}^{\c P}/C_p(\wt{\c L}^{\c P})$. 
 
Note that $\c L^{\c P}$ inherits 
the $\c P$-topological structure from $\bar{\c K}^{D\c P}$ (use the 
topology of tensor product on  
$\sum\limits _{1\leqslant i<p}
\left (\bar{\c K}^{D\c P}\right )^{\otimes i}\supset\c L^{\c P}$), and the 
profinite completion of $\c L^{\c P}$ coincides with $\c L$.

Introduce the Lie algebras $\c L_{\alpha }$ 
with generators $\bar{\c K}^D_{\alpha }=\Hom (\bar C_{\alpha },\F _p)$
and $\c L^{\c P}_{\alpha }$ 
with generators $\bar{\c K}_{\alpha }^{D\c P}=
\Hom _{\c P\text{-cont}}(\bar C_{\alpha },\F _p)$. 
Then $\c L_{\alpha }^{\c P}$ has discrete topology, its  
profinite completion 
coincides with $\c L_{\alpha }$, 
$\c L=\varprojlim\limits _{\alpha }\c L_{\alpha }$ and $\c L^{\c P}=
\varprojlim\limits_{\alpha }\c L^{\c P}_{\alpha }$.
If $\c L_{\alpha\beta }$ is the subalgebra in $\c L_{\alpha }$ 
generated by $\bar{\c K}^D_{\alpha\beta }$  then $\c L_{\alpha\beta }$ is finite and 
$\c L_{\alpha }^{\c P}=\varinjlim \limits _{\beta }
\c L_{\alpha\beta }$.

The elements of the 
Lie algebra $\c L^{\c P}_k:=\c L^{\c P}\otimes _{\F _p}k$ 
appear as convergent $k$-linear 
combinations of the Lie monomials of the form 
$$\sum\limits _{D_1,\dots ,D_r}
\gamma _{D_1,\dots D_r}[\dots [D_1,\dots ],D_r]\, .$$ 
where all $D_1,\dots ,D_r$ belong to \eqref{E3.4}. 
The condition of convergency means that for any $\alpha\in\c I$, 
all but finitely many of these monomials have the zero image 
in $\c L^{\c P}_{\alpha }$. 

We can describe similarly the enveloping algebra of $\c L^{\c P}$. Namely, 
let $\frak A^{\c P}$ and $\frak A_{\alpha }^{\c P}$ be enveloping algebras for 
$\c L^{\c P}$ and, resp., $\c L_{\alpha }^{\c P}$ taken modulo $p$-th powers of 
the corresponding augmentation ideals. 
Then $\frak A^{\c P}=\underset{\alpha }\varprojlim 
\frak A^{\c P}_{\alpha }$ and $\frak A_{\alpha ,k}^{\c P}$ consists of all 
polynomials of total degree 
$<p$ in the subset of variables $D_{an}$ from \eqref{E3.4} satisfying the condition 
$t^{-a}\notin C_{\alpha }$. In other words, 
the elements of $\frak A^{\c P}$ are 
characterized in the algebra $\frak A$ 
as $\c P$-continuous polynomials on $\bar{\c K}$ with values in $\F _p$ of 
total degree $<p$. 
Of course, $\frak A$ and $\c L$ 
can be recovered as the profinite completion of 
$\frak A^{\c P}$ and, resp., $\c L^{\c P}$.

\subsection{Class of conjugated subgroups 
$\op{cl}^{\c P}(\c G_{<p})$}\label{S3.4} \ \ 

Let $\c G  =\Gal (\c K_{sep}/\c K)$ be the absolute 
Galois group of the field $\c K$. 

If $\c G _{<2}:=\c G/\c G^pC_2(\c G)$ is the 
maximal abelian quotient of period $p$ of $\c G $  
then the classical Artin-Schreier duality 
$\bar{\c K}\times \c G _{<2}\To \F _p$   
allows us to identify 
$\c G_{<2}$ with $\bar{\c K}^D=\left (\bar{\c K}^{D\c P}\right )^{DD}$ 
and to introduce  
a dense subgroup 
$\c G^{\c P}_{<2}:=\bar{\c K}^{D\c P}$ in $\c G _{<2}$. 
Note that with respect to this identification, 
the elements $D_a^{(i)}\in\bar{\c K}^{D\c P}$ from 
\eqref{E3.3} appear as elements of $\c G_{<2}$ such that if 
$T_{bj}\in\c K_{sep}$ are such that 
$T^p_{bj}-T_{bj}=\alpha _jt^{-b}$ (in notation from Sect.\,\ref{S2.3}) 
then $D_a^{(i)}(T_{bj})=T_{bj}+\delta _{ab}\delta _{ij}$. 
Similarly, the elements $D_{an}\in\bar{\c K}^{D\c P}_k$ 
from \eqref{E3.4} act as follows: if $b\in\N $, 
$\op{gcd}(b,p)=1$ and 
$T_b\in\c K_{sep}$ is such that $T_b^q-T_b=t^{-b}$ then for $0\leqslant m<N_0$, 
$D_{an}(T^{p^m}_b)=T^{p^m}_b+\delta _{ab}\delta _{nm}$. 

The identification of local class field theory 
$\c G_{<2}\simeq K_N(\c K)/p$ induces the identification 
$\c G_{<2}^{\c P}\simeq K_N^{top}(\c K)/p$, 
where 
$K_N^{top}$ is the topological version of the functor $K_N$, cf.\,e.g. 
\cite{Pa1, Pa2}. 
The subgroup $\c G_{<2}^{\c P}$ 
is considerably smaller than $\c G_{<2}$ but its profinite completion  
recovers the whole $\c G_{<2}$. In particular, $\c G_{<2}^{\c P}$ 
can be used instead of $\c G_{<2}$ when studying  
finite (abelian) extensions of $\c K$ inside 
$\c K_{<2}=\c K_{<p}^{C_2(\c G_{<p})}$.  
Our target is to introduce an analog of $\c G^{\c P}_{<2}$ in the 
case of $p$-extensions of nilpotent class $<p$. 

From now on we will consider only $e\in
\c E(\c L^{\c P}_{\c K}):=\c L^{\c P}_{\c K}\cap
\c E(\c L_{\c K})$. 
 Under this assumption if $S$ is a section of the projection 
$\Pi :\c K\To \bar{\c K}$ such that 
$e\,\op{mod}\,C_2(\c L^{\c P}_{\c K})=e_S$  
then $S$ is $\c P$-continuous.  

As earlier, choose $f\in\c F(e)$ and set $\pi =\pi _f(e):\c G_{<p}\simeq G(\c L)$. 

\begin{definition} $\op{cl}^{\c P}(\c G_{<p})$ is the class 
of conjugated subgroups of $\c G_{<p}$ containing 
$\pi _f(e)^{-1}G(\c L^{\c P})$.
\end{definition}

\begin{Thm} \label{T3.4} 
 The class $\op{cl}^{\c P}(\c G_{<p})$ does not depend on 
the choices of $e\in\c E(\c L_{\c K}^{\c P})$ and $f\in\c F(e)$. 
\end{Thm}

\begin{proof}
 Suppose $e'\in\c E(\c L_{\c K}^{\c P})$, $f'\in\c F(e')$ and set $\pi '=\pi _{f'}(e')$. 
 We must prove that $\pi ^{-1}G(\c L^{\c P})$ and 
 $(\pi ')^{-1}G(\c L^{\c P})$ are conjugated 
 in $\c G_{<p}$. 

\begin{Lem} \label{L3.5} 
 There is an $x\in\c L^{\c P}_{\c K}$ and a $\c P$-continuous section 
 $A:\bar{\c K}^{D\c P}\To\c L^{\c P}$ of the natural 
projection $\c L^{\c P}\To \c L^{\c P}/C_2(\c L^{\c P})=
\bar{\c K}^{D\c P}$ 
such that 
$$e'=\sigma (x)\circ (\c A\otimes _{\F _p}^{\c P}
\id _{\c K})e\circ (-x)\, ,$$
where $\c A\in\op{Aut}_{\text{Lie}}\c L^{\c P}$ is 
such that $\c A|_{\bar{\c K}^{D\c P}}=A$. 
\end{Lem}

\begin{proof} [Proof of lemma]
The proof appears as a $\c P$-topological version of 
the proof of Prop.\,\ref{P1.1}, where we use 
the ($\c P$-continuous) operators 
from Prop.\,\ref{P3.1} and \ref{P3.2}. 
 
Let $\{l_{\alpha }\ |\ \alpha\in\c I\}$ be a $\c P$-topological  
$\F _p$-basis of $\bar{\c K}$. 
Let $\hat l_{\alpha }$, $\alpha\in\c I$,  be the dual 
($\c P$-topological) $\F _p$-basis for $\bar{\c K}^{D\c P}$, i.e. 
for any $\alpha _1,\alpha _2\in \c I$, 
$\hat l_{\alpha _1}(l_{\alpha _2})=\delta _{\alpha _1\alpha _2}$. 
Then for the corresponding sections $S$ and $S'$, we have 
the $\c P$-convergent series 
$e_S=\sum\limits _{\alpha }\hat l_{\alpha }\otimes ^{\c P}S(l_{\alpha })$ and 
$e'_{S'}=\sum\limits _{\alpha }\hat l_{\alpha }\otimes ^{\c P}S'(l_{\alpha })$. 

Apply induction on $r\geqslant 1$ to prove the 
existence of $x_r\in \c L^{\c P}_{\c K}$ and a 
section $A_r$ of the projection $\c L^{\c P}\To \bar{\c K}^{D\c P}$ such that 
 $$e'\equiv \sigma (x_r)\circ (\c A_r\otimes \id _{\c K})
 e\circ (-x_r)\,\op{mod}\,C_{r+1}(\c L^{\c P}_{\c K})\, ,$$ 
 with the appropriate $\c P$-continuous automorphism $\c A_r$ of $\c L^{\c P}$. 
 
If $r=1$ take $A_1=\id _M$ and 
$x_1=\sum\limits _{\alpha }\hat l_{\alpha }\otimes ^{\c P}x_{1\alpha }$, where 
all $x_{1\alpha }=\c R(S'(l_{\alpha })-S(l_{\alpha }))\in\c K$ and 
$\c R$ is the operator from Prop.\,\ref{P3.1}. Note that 
$x_1\in\c L^{\c P}_{\c K}$ because 
$\c R$ is $\c P$-continuous. 
 
 If $r\geqslant 1$ and such $x_r$ and $A_r$ exist then 
 there is $l_{r+1}\in C_{r+1}(\c L^{\c P}_{\c K})$ such that 
 $e'\equiv\sigma x_r\circ (\c A_r\otimes ^{\c P}\id _{\c K})e
 \circ (-x_r)\circ l_{r+1}\,\op{mod}\,C_{r+2}(\c L^{\c P}_{\c K})$. 
 
 Let $l_{r+1}=\sum _{\alpha }c_{\alpha }\otimes ^{\c P}b_{\alpha }$ 
 with all $c_{\alpha }\in C_{r+1}(\c L^{\c P})$ 
 and $b_{\alpha }\in\c K$. Then 
 $$l_{r+1}=l'+\sigma (x')-x'\, ,$$ 
 where 
 $l'=\sum\limits _{\alpha }c_{\alpha }\otimes ^{\c P}(S\Pi )(b_{\alpha })$ 
 and $x'=\sum\limits _{\alpha }c_{\alpha }\otimes ^{\c P}R_S(b_{\alpha })$ 
 belong to $C_{r+1}(\c L^{\c P}_{\c K})\subset 
 \c L^{\c P}_{\c K}$. 
 It remains to set 
 $A_{r+1}(\hat l_{\alpha })=A_r(\hat l_{\alpha })+c_{\alpha }$ and 
 $x_{r+1}=x_r+x'$. 
 The lemma is proved.  
\end{proof} 

\begin{remark} The main reason why 
the proof of Prop.\,\ref{P1.1} works in the $\c P$-topological context 
is that $\c L^{\c P}_{\c K}=\c L^{\c P}\otimes ^{\c P}\c K$ 
is stable with respect to the action of 
$\c P$-continuous operators on the factor $\c K$. 
\end{remark}

Continue the proof of the theorem. 

Denote by the same symbol $\c A$ 
the (unique) extension of $\c A$ to $\c L$ (use that 
$\c L$ is a profinite completion of $\c L^{\c P}$). 
Then $f''=x\circ (\c A\otimes \id _{sep})f\in\c F(e')$ 
and for any  
$\tau\in\c G _{<p}$, $\pi _{f''}(e')(\tau )=(\c A\cdot \pi )\tau $. 

Then $(\c A\cdot\pi )(\c G_{<p}^{\c P})=\c A(\c L^{\c P})=\c L^{\c P}$. 
This implies that $\pi ^{-1}G(\c L^{\c P})=\pi _{f''}(e')^{-1}G(\c L^{\c P})$. 
But $f',f''\in\c F(e')$ implies that 
$(\pi ')^{-1}G(\c L^{\c P})$ and 
$\pi _{f''}(e')^{-1}G(\c L^{\c P})$ are conjugated in $\c G _{<p}$.  
\end{proof}

\subsection{Galois $\c P$-correspondence} \label{S3.5} 
As earlier, consider 
$\bar{\c K}^D$, $\bar{\c K}^{D\c P}$, $e\in\c E(\c L^{\c P}_{\c K})$, 
$f\in \c F(e)$ and $\pi :=\pi _f(e):\c G _{<p}\simeq G(\c L)$. 
Suppose $\c K'$ is a finite field extension of $\c K$ in $\c K_{sep}$.  
Let $\c G '_{<p}$, $\bar{\c K}'^D$, $\bar{\c K}^{'D\c P}$, 
$e'\in\c E(\c L_{\c K'}^{'\c P})$, 
$f'\in\c F(e')$ and 
$\pi '=\pi _{f'}(e'):\c G '_{<p}\simeq G(\c L')$ 
be the similar objects for the field $\c K'$. 

The natural morphism of 
profinite groups $\Theta :\c G '_{<p}\To \c G _{<p}$ 
can be described 
in the terms of identifications $\pi $ and $\pi '$ by Prop.\,\ref{P1.3}. 
It admits the following $\c P$-version. 

\begin{Prop} \label{P3.6}  
Suppose $\c G^{'\c P}_{<p}\in\op{cl}^{\c P}\c G'_{<p}$. Then:
\medskip 

{\rm a)} there is  
$\c G^{\c P}_{<p}\in\op{cl}^{\c P}\c G_{<p}$ such that 
$\Theta (\c G^{'\c P}_{<p})\subset \c G^{\c P}_{<p}\, ;$ 
\medskip 

{\rm b)}  
$(\c G^{\c P}_{<p}:\Theta (\c G^{'\c P}_{<p}))=
(\c G_{<p}:\Theta (\c G'_{<p}))$, i.e.   
$\Theta (\c G^{'\c P}_{<p})=\Theta (\c G'_{<p})\cap\c G^{\c P}_{<p}$. 
\end{Prop}

 \begin{proof} a) We can assume that  $f'\in\c F(e')$ is such that 
 $\c G^{'\c P}_{<p}=\pi '^{-1}G(\c L^{'\c P})$. 
 Then we can apply the $\c P$-topological version of the proof of Prop.\,\ref{P1.3} 
 to establish the existence of $\c P$-continuous 
$\c A\in\Hom _{\op{Lie}}(\c L^{'\c P},\c L^{\c P})$ and $x'\in \c L^{\c P}_{\c K'}$ such that 
\begin{equation} \label{E3.5} e\otimes ^{\c P}_{\c K}1_{\c K'}=
\sigma (x')\circ (\c A\otimes\id _{\c K})e'\circ (-x')\,.  
\end{equation}
 
From \eqref{E3.5} it follows that both 
$(-x')\circ f$ and $(\c A\otimes\id _{\c K'})f'$ belong to 
$\c F((\c A\otimes\id )e')
%\subset\c A(\c L')_{sep}
\subset \c L_{sep}$. 
Therefore, 
there is $l\in \c L$ such that 
$$(-x')\circ f=(\c A\otimes\id _{sep})f'\circ l\, .$$

As a result, for any $\tau '\in\c G'_{<p}$, 
$\pi (\Theta (\tau '))=(-l)\circ \c A(\pi '(\tau '))\circ l$, and 
$$\pi (\Theta (\c G^{'\c P}_{<p}))=(-l)\circ \c A(\c L^{'\c P})\circ (-l)
\subset (-l)\circ \c L^{\c P}\circ l\, .$$
Equivalently, for $g=\pi ^{-1}(l)\in \c G_{<p}$, we have 
$$\Theta (\c G^{'\c P}_{<p})\subset (-g)
\circ \pi ^{-1}(\c L^{\c P})\circ g\in\op{cl}^{\c P}
\c G_{<p}\, .$$ 
So, we can take 
$\c G_{<p}^{\c P}=\Ad (g)(\pi ^{-1}\c L^{\c P})$. 
\medskip

b)   
Assume that in the notation from a), $l=0$. This guarantees  
$\Theta (\c G'^{\c P}_{<p})\subset \c G^{\c P}_{<p}$ and $\pi =\c A\cdot\pi '$, where 
$\pi =\pi _f(e):\c G_{<p}\simeq G(\c L)$,  
$\pi '=\pi _{f'}(e'):\c G'_{<p}\simeq G(\c L')$ and 
$\c A:\c L'\To \c L$ is induced by $\Theta $. 

Let $p^n=[\c K':\c K]=
(\c G_{<p}:\Theta (\c G'_{<p}))$. 
\medskip 

$\bullet $\ {\it The case $[\c K':\c K]=p$.}
\medskip 

Here $\c K'/\c K$ is Galois of degree $p$, $(\c L:\c A(\c L'))=p$,  
$\c A(\c L')$ is an ideal in $\c L$, 
and $\c A(\c L')=C_2(\c L)+L$, where $L\subset \bar{\c K}^D$ 
is of index $p$.
\medskip 

Let 
$\c A(\c L'^{\c P})=C_2({\c L}^{\c P})+L^0\subset \c A(\c L')$, where 
in notation from Sect.\,\ref{S3.2},
$$L^0\subset 
\bar{\c K}^{\c PD}=\underset{\alpha }\varprojlim\,
\Hom _{{\c P}\text{-cont}}(\bar C_{\alpha },\F _p)\, .$$

Let $\c K'=\c K(T')$, where 
$T^{\,\prime p}-T'=c\in\c K$. Then  
$\c K'=\c K_{<p}^{H}$, with  $H=\Theta (\c G'_{<p})$ and 
$\pi (H)=G(C_2(\c L)+L)$. Therefore, 
$L\subset \bar{\c K}^D$ is characterized by 
the trivial action on $T'$ or, equivalently, $L=\op{Ann}(\bar c)$, where 
$\bar c\in\bar{\c K}$ is the image of $c$ under the natural projection 
$\Pi :\c K\To\bar{\c K}$. 

We can assume that for some index $\alpha _0$, $\Pi (c)\in \bar C_{\alpha _0}$, 
because $\bar{\c K}$ is the union of all $\bar C_{\alpha }=\Pi (C_{\alpha })$. 
As a result: 

-- the $\c P$-subgroup 
$H^{\c P}$ appears in the form $\pi ^{-1}G(C_2(\c L^{\c P})+L^0)$;

-- 
$L^0$ is the preimage 
of a subspace $L^0_{\alpha _0}
\subset \Hom _{\c P\text{-cont}}(\bar C_{\alpha _0}, \F _p)$;

-- $L^0_{\alpha _0}$  
consists of all {\sc finite} $\F _p$-linear combinations of 
the elements $D_a^{(i)}\in \bar C_{\alpha _0}$ from \eqref{E3.3} 
which annihilate $\bar c$. This means that $L^0_{\alpha _0}$ is of index $p$ in 
$\Hom _{\c P\text{-cont}}(\bar C_{\alpha _0}, \F _p)$,  
$L^0$ is of index $p$ in $\bar{\c K}^{\c PD}$, $\c A(\c L'^{\c P})$ is of index $p$ in 
$\c L^{\c P}$ and b) is proved in the case $n=1$. 
\medskip 

{\it Inductive step.}

Suppose $n\geqslant 2$ and b) is proved for field extensions of degree 
$p^{n-1}$. 

Consider the tower 
$\c K\subset \c K_1\subset \c K'$, $[\c K_1:\c K]=p$, 
$[\c K':\c K]=p^{n-1}$. 

Using similar notation for $\c K'$ and $\c K_1$ we have:
\begin{enumerate} 
 \item 
the fields tower 
$\c K\subset \c K_1\subset \c K'\subset \c K_{<p}\subset \c K_{1,<p}\subset \c K'_{<p}$, 

\item 
the compatible identifications: 

-- $\c G_{<p}\simeq G(\c L)\subset \c G_{1,<p}\simeq G(\c L_1)
\subset \c G'_{<p}\simeq G(\c L')$, 

-- $\c G^{\c P}_{<p}\simeq G(\c L^{\c P})\subset \c G^{\c P}_{1,<p}\simeq G(\c L_1)
\subset \c G'^{\c P}_{<p}\simeq G(\c L'^{\c P})$ 

\item 
the natural group homomorphisms:

-- $\Theta _1:\c G_{1,<p}\To \Theta _1(\c G_{1,<p})\subset \c G_{<p}$, 
\medskip 

-- $\Theta ':\c G'_{<p}\To \Theta '(\c G'_{<p})\subset\c G_{1,<p}$
\medskip 

-- $\Theta _1: \Theta '(\c G_{<p})\To \Theta (\c G'_{<p})
\subset \Theta _1(\c G_{1,<p})\subset \c G_{<p},$ 

\item 
the restrictions of the above $\Theta $, $\Theta '$, $\Theta _1$ 
to the corresponding $\c P$-subgroups satisfy analogs of 
relations from c). 
\end{enumerate}

Note that 
$\Ker\,\Theta _1=\op{Gal}(\c K_{1,<p}/\c K_{<p})=J$ is the 
profinite closure of 
$\c J^{\c P}:=\c G^{\c P}_{1,<p}\cap J=\Ker\,\Theta _1|_{\c G^{\c P}_{1,<p}}$. 
Therefore, $\Theta _1(\c G^{\c P}_{1,<p})=
\c G_{1,<p}/\Ker\,J^{\c P}$. 

Similarly, 
$\Theta _1$ induces a group epimorphic map 
$\Theta '(\c G_{<p})\To \Theta (\c G'_{<p})$ with the kernel $J$, 
the corresponding epimorphism 
$\Theta '(\c G^{\c P}_{<p})\To \Theta (\c G'^{\c P}_{<p})$ has the kernel 
$J^{\c P}$ and 
$\Theta (\c G'^{\c P}_{<p})=\Theta '(\c G^{\c P}_{<p})/J^{\c P}$. 

Therefore,  
$(\c G^{\c P}_{1,<p}:\Theta '(\c G^{\c P}_{<p}))=(\Theta _1(\c G^{\c P}_{1,<p}):
\Theta (\c G'^{\c P}_{<p}))$. By the inductive assumption, 
this index equals $p^{n-1}$. Finally, using the case $n=1$ we obtain 
$(\c G^{\c P}_{<p}:\Theta (\c G'^{\c P}_{<p}))=p^n$. 
\end{proof} 

\begin{definition} If $\c H=\Theta (\c G'_{<p})$ then we 
set $\c H^{\c P}=\c H\cap\c G^{\c P}_{<p}$.  
\end{definition}

Clearly, the conjugacy class of $\c H^{\c P}$ in 
its profinite completion $\c H$ is well defined.

\begin{Cor} \label{C3.7}
{\rm a)} Any extension $\c K'$ of $\c K$ in $\c K_{<p}$ (in the category of 
$N$-dimensional local fields) appears in the form 
$\c K_{<p}^{\c H}$, where $\c H$ is the profinite completion of a $\c P$-closed subgroup 
$\c H^{\c P}$ of some $\c G_{<p}^{\c P}\in\op{cl}^{\c P}\c G_{<p}$.

{\rm b)} In the above notation, 
$\c K'$ is Galois over $\c K$ iff the subgroup $\c H^{\c P}$ of 
$\c G_{<p}^{\c P}$ is normal, and  
$\op{Gal}(\c K'/\c K)=\c G^{\c P}_{<p}/\c H^{\c P}$. 
\end{Cor}

It remains to characterize the subgroups $\c H^{\c P}\subset\c G^{\c P}_{<p}$ such that 
$\c K^{\c H}$ is an extension of $\c K$ in $\c K_{<p}$. 

\begin{Prop} \label{P3.8}
 Let $H\subset \c G_{<p}^{\c P}$ be a subgroup. Then 
 $H=\c H^{\c P}$, where $\c K^{\c H}=\c K'$ 
 is $N$-dimensional field  extension of $\c K$ iff 
 \medskip 
 
 {\rm a)} $(\c G^{\c P}_{<p}:H)<\infty $;
 \medskip 
 
 {\rm b)} $H$ is $\c P$-open in $\c G_{<p}^{\c P}$. 
\end{Prop}

\begin{proof} If $\c K'$ is field extension of $\c K$ in the category of 
$N$-dimensional fields then $[\c K':\c K]=(\c G_{<p}:\c H)=
(\c G^{\c P}_{<p}:\c H^{\c P})<\infty $ and by Prop.\,\ref{P3.6} 
$\c H^{\c P}$ is $\c P$-closed. It is also $\c P$-open as a closed subgroup of 
finite index in $\c G^{\c P}_{<p}$. 

To proceed in the opposite direction note that $H=G(L)$, where $L$ is a 
Lie subalgebra in ${\c L}^{\c P}$ and the index 
$({\c L}^{\c P}:L)$ is a power of $p$. 
Choose   
an increasing sequence of Lie algebras 
$L=L_0\subset L_1\dots \subset L_n={\c L}^{\c P}$  
where 
each $L_{i-1}$ is ideal in $L_i$ and $(L_i:L_{i-1})=p$. 
As a result, we can proceed by induction and it will be sufficient to consider the case 
$n=1$. 

Then $L\supset C_2({\c L}^{\c P})$ and 
$L=C_2({\c L}^{\c P})+L^0$, where 
$$L^0\subset \Hom _{{\c P}\text{-cont}}
(\bar{\c K},\F _p)=\underset{\alpha }\varprojlim\,
\Hom _{{\c P}\text{-cont}}(\bar C_{\alpha },\F _p)\, ,$$ 
cf. notation from \,Sect.\,\ref{S3.2}. Since $L^0$ is $\c P$-open there is an 
index $\alpha _0$ such that 
$L^0$ is the preimage of a subgroup $L^0_{\alpha _0}$ of index $p$ in 
$\Hom _{{\c P}\text{-cont}}(\bar C_{\alpha _0},\F _p)$. Therefore, there is 
$\bar c\in \bar C_{\alpha _0}$ such that $L^0_{\alpha _0}=\op{Ann}\,\bar c$ in 
$\Hom _{{\c P}\text{-cont}}(\bar C_{\alpha _0},\F _p)$. 

Let $\c K'=\c K(T')$, where 
$T^{\,\prime p}-T'=\Pi ^{-1}\bar c$. Then  
$\c K'=\c K_{<p}^{\c H}$, where the subgroup $\c H$ of $\c G_{<p}$ is such that 
$\c H=G(C_2(\c L)+L)$ 
and $L\subset \bar{\c K}^D$ is characterized by 
the trivial action on $T'$. 
Therefore, the corresponding $\c P$-subgroup $\c H^{\c P}=G(C_2(\c L^{\c P})+L^{\c P})$,  
where $L^{\c P}$ consists of finite $\F _p$-linear combinations of 
the elements $D_a^{(i)}$ from \eqref{E3.3} which annihilate $c$. 
Therefore, $L^{\c P}=L^0$ and $\c H^{\c P}=H$. 
\end{proof} 

\subsection{More general $\c P$-groups} \label{S3.6} 

Suppose $G'\subset \Aut \,\c K'$. For example, $\c K'/\c K$ is Galois and 
 $G'=\op{Gal}(\c K'/\c K)$. Consider the group 
$\Gamma '\subset \op{Aut}\,\c K'_{<p}$ of all lifts of the elements of $G'$ 
to $\c K'_{<p}$. These lifts can be treated in terms of 
couples $(C'_g,\c A'_g)$, where $g\in G'$, 
$C'_g\in \c L'_{\c K'}$ and 
$\c A'\in\op{Aut}_{\text{Lie}}\c L'$, cf.\,Sect.\,\ref{S1.6}. 
This description uses the identification $\pi '=\pi _{f'}(e'):\c G'_{<p}\simeq G(\c L')$.   
After applying $\pi '^{-1}$ we obtain the exact sequence 
$$1\To \c G'_{<p}\To \Gamma '\To G'\To 1\,.$$
 
Consider a subgroup $\Gamma '^{\c P}$ of $\Gamma '$ coming from  
$C'_g\in\c L'^{\c P}_{\c K}=\c L'^{\c P}\otimes ^{\c P}\c K'$ and 
$\c P$-continuous $\c A'_g$. (For example, 
use the $\c P$-continuous operators $\c R$ and $\c S$ from Sect.\,\ref{S3.1} 
to recover the corresponding pairs $(C'_g,\c A'_g)$, cf. e.g. Sect.\,\ref{S4.2} below.)
We obtain the following exact sequence 
$$1\To \c G'^{\c P}_{<p}\To \Gamma '^{\c P}\To G'\To 1\,.$$
This construction of the subgroup $\Gamma '^{\c P}$ of $\Gamma '$ 
does not depend on a choice of ``$\c P$-continuous'' lifts of elements of $g\in G'$ 
(all such lifts differ by elements of $\c G'^{\c P}_{<p}$). 

The above construction in the case $G'=\op{Gal}(\c K'/\c K)$ allows us to recover 
(uniquely up to isomorphism) the group 
$\c G^{\c P}_{<p}$ from $\c H^{\c P}=\Theta (\c G'^{\c P}_{<p})$. 
Even more, if $\c K\subset \c K_1\subset \c K_2\subset \c K_{<p}$, 
$\c H_i=\op{Gal}(\c K_{<p}/\c K_i)$,  
$\c H_i^{\c P}=\c H_i\cap\c G^{\c P}_{<p}$ with $i=1,2$, and $\c K_2/K_1$ is 
Galois then $\c H_2^{\c P}=\c H_2\cap\c G_{<p}^{\c P}$ 
is uniquely (up to isomorphism) 
recovered from $\c H_1^{\c P}$.

\section{The groups $\c G^{\c P} _{\omega }$ and 
$\Gamma ^{\c P}_{\omega }$} \label{S4} 

As earlier, $\c K$ is $N$-dimensional local field  of characteristic $p$ 
with fixed system of local  parameters 
$t=\{t_1,\dots ,t_N\}$ and the last residue field $k\simeq\F _{p^{N_0}}$. 
Fix $\alpha _0\in k$ such that $\op{Tr}_{k/\F _p}(\alpha _0)=1$.  
Let $S=S_{t,\alpha _0}$ be the section from Prop.\,\ref{P3.2}. 

\ \ 

Take $e:=e_S=\sum\limits _{a\in\Z _N^+(p)}
t^{-a}D_{a\,0}+\alpha _{\bar 0}D_{\bar 0}$, choose $f\in\c F(e)$ and consider 
$\pi =\pi _f(e):\c G_{<p}\simeq G(\c L)$. 
\medskip

Fix $\bar c^{\,0}=(c^{\,0}_1,\dots ,c_N^{\,0})\in\,p\Z ^{N}$ such that $c_1^0>0$, 
set 
$t^{\bar c^0}:=t_1^{c_1^0}\dots t_N^{c_N^0}$. 
Choose $\omega =\sum\limits _{\iota \geqslant\bar 0}
\beta _{\iota}t^{(\bar c^0/p)+\iota }
\in t^{\bar c_0/p}\c O^*_{\c K}$, where  
all 
$\beta _{\iota }=\beta _{\iota }(\omega )\in k$, $\beta _{\bar 0}\ne 0$.
Let $E(X)=\exp\left (\sum\limits _{i\geqslant 0}X^{p^i}/{p^i}\right )$ 
be the Artin-Hasse exponential.

\subsection{Automorphisms $h_{\omega }^{(m)}$} 
\label{S4.1}

For $1\leqslant m\leqslant N$, 
let $h_{\omega }^{(m)}$ be 
the $\c P$-continuous automorphism   
of $\c K$ such that $h_{\omega }^{(m)}|_k=\id $, $h_{\omega }^{(m)}(t_m)=t_m
E(\omega ^p)$ and for all  
$j\ne m$, $h_{\omega }^{(m)}(t_j)=t_j$. Let $\m _{\c K}$ be 
the maximal ideal in $\c O_{\c K}$. 
Clearly, $\m _{\c K}$ consists of all $\c P$-convergent $k$-linear combinations of 
$t^{a}$, where $a\in\Z ^N_{>\bar 0}$. 

 For $n\in\Z$, let $h_{\omega }^{(m)n}$ be 
 the $n$-th iteration of $h_{\omega }^{(m)}$ 
 and, similarly, 
 denote by $h_{\omega }^{(m_1)}h_{\omega }^{(m_2)}$ the composition of 
 $h_{\omega }^{(m_1)}$ and $h_{\omega }^{(m_2)}$.  

\begin{Prop} \label{P4.1}  \ \ 

{\rm a)} For any $n\geqslant 0$, $h_{\omega }^{(m)n}(t_m)\equiv 
t_mE(n\,\omega ^p)\,\op{mod}\,t^{p\bar c^0}\m _{\c K}$;
\medskip 

{\rm b)} $h_{\omega }^{(m_1)}h_{\omega }^{(m_2)}\equiv h_{\omega }^{(m_2)}
h_{\omega }^{(m_1)}\,\op{mod}\,t^{p\bar c^0}\m _{\c K}$. 
\end{Prop} 

\begin{proof}  Note that 
$h_{\omega }^{(m)}(t_m)\equiv t_m\,\op{mod}\,t^{\bar c^0}\m _{\c K}$ 
and this implies for 
any $\iota \geqslant\bar 0$, that 
$h_{\omega }^{(m)}(t^{\bar c^0+p\iota })\equiv t^{\bar c^0+p\iota }\,
\op{mod}\,t^{p\bar c^0}\m _{\c K}$. As a result,  
\begin{equation} \label{E4.1} h_{\omega }^{(m)}(\omega ^p)
\equiv \omega ^p\,\op{mod}\,t^{p\bar c^0}\m _{\c K}\, .  
\end{equation} 

Apply induction on $n\geqslant 0$ to prove part a) of the proposition. 

If it is proved for some $n\geqslant 0$ then 
$$h_{\omega }^{(m)n+1}(t_m)\equiv  h_{\omega }^{(m)}(t_mE(n\omega ^p))\equiv 
t_mE(\omega ^p)E(n\omega ^p)\equiv 
t_mE((n+1)\omega ^p)$$ 
modulo $t^{p\bar c^0}\m _{\c K}$ 
(use that $E(X+Y)\equiv E(X)E(Y)\,\op{mod}\,\deg p$). 

Similarly, relation \eqref{E4.1} implies part b).
\end{proof}

\begin{remark}
 The above 
proposition can be stated also for the truncated exponential 
$\wt{\exp}(X)=1+X+\dots +X^{p-1}/(p-1)!$  instead of  $E(X)$.  
\end{remark}

\subsection{The groups ${\c G}_{\omega }$ 
and ${\c G}^{\c P}_{\omega }$} \label{S4.2} 

Let $\hat h_{\omega }^{(m)}\in\Aut\,\c K_{<p}$ 
be such that $\hat h_{\omega }^{(m)}|_{\c K}=h_{\omega }^{(m)}$.  
Denote by ${\c G}_{\omega }$  the 
subgroup in $\Aut\,\c K_{<p}$ 
generated by the elements of $\c G_{<p}$ and 
the lifts $\hat h_{\omega }^{(m)}$ with $1\leqslant m\leqslant N$. 
The elements $\hat h_{\omega }$ of ${\c G}_{\omega }$ are characterised by 
the property $\hat h_{\omega }|_{\c K}\in 
\langle h_{\omega }^{(1)}, \dots ,h_{\omega }^{(N)}\rangle \subset \Aut \c K$. 
They can be uniquely specified by the  couples 
$(C,\c A)\in\c L_{\c K}\times \op{Aut}\,\c L$ 
such that 
$$\hat h_{\omega }(f)=
C\circ (\c A \otimes\id )f\, $$
or, equivalently, such that (where $h_{\omega }=\hat h_{\omega }|_{\c K}$) 
\begin{equation} \label{E4.2} (\id _{\c L}\otimes h_{\omega })e=\sigma C
\circ (\c A\otimes\id )e\circ (-C)\, .
\end{equation}

If $\hat h'_{\omega }$ corresponds to $(C',\c A')$ then 
the composition $\hat h_{\omega }'\cdot \hat h_{\omega }$ corresponds to the couple 
$(h_{\omega }'(C)\circ \c A(C'), \c A\,\c A')$, where 
$h_{\omega }'=\hat h'_{\omega }|_{\c K}$. 
With this notation 
the subgroup $\c G_{<p}\subset {\c G}_{\omega }$ is identified 
with the subgroup of couples $(l, \op{Ad}\,l)$, where $l\in\c L$.   
Indeed, under the identification $\pi =\pi _f(e):\c G_{<p}\simeq G(\c L)$ 
from Prop.\,\ref{P1.4}, if $\tau\in\c G_{<p}$ then 
$\tau (f)=f\circ l=l\circ (\op{Ad}\,l\otimes\id )f$. 

Suppose $1\leqslant m\leqslant N$ and 
$\hat h_{\omega }^{(m)}$ is specified via the couple $(C^{(m)},\c A^{(m)})$. 
Relation \eqref{E4.2} can be treated via the following recurrent procedure.

Suppose $s\geqslant 1$ and the couple 
$(C^{(m)}_s,\c A^{(m)}_s)$ satisfies relation \eqref{E4.2} 
modulo $s$-th commutators 
$C_s(\c L_{\c K})$. Use the operators $\c R$ and $\c S$ 
from Sect.\,\ref{S3.1} to obtain 
$C'_s\in C_{s+1}(\c L_{\c K})$ and 
$\c A'_s\in\Hom _{\op{Lie}}(\c L, C_{s+1}\c L)$ 
such that 
$$\sigma C'_s-C'_s+(\c A'_s\otimes\id _{\c K})e\equiv $$
$$(\id _{\c L}\otimes h_{\omega }^{(m)})e-\sigma C^{(m)}_s
\circ (\c A^{(m)}_s\otimes\id _{\c K})e\circ (-C^{(m)}_s)
\,\op{mod}\,C_{s+1}(\c L_{\c K})\, .$$
Then the couple $(C_s+C'_s, \c A_s+\c A'_s)$ satisfies \eqref{E4.2} 
modulo $C_{s+2}(\c L_{\c K})$. 

Denote by $\hat h_{\omega }^{0(m)}$ the lift 
of $h_{\omega }^{(m)}$ which corresponds to the couple 
$(C^{0(m)}, \c A^{0(m)}):=(C^{(m)}_p, \c A^{(m)}_p)$. 
Note that $C^{0(m)}\in\c L_{\c K}^{\c P}$ and 
$\c A^{0(m)}|_{\c L^{\c P}}$ is a $\c P$-continuous 
automorphism of the Lie algebra $\c L^{\c P}$.

Using that 
$(\id _{\c L}\otimes h_{\omega }^{(m)})e\in\c L^{\c P}_{\c K}$ and 
$\c L\cap \c L_{\c K}^{\c P}=\c L^{\c P}$ 
we obtain the following property:

\begin{Prop} \label{P4.2} 
A lift $\hat h_{\omega }^{(m)}$ corresponds to a couple $(C^{(m)}, \c A^{(m)})$ 
with  $C^{(m)}\in\c L_{\c K}^{\c P}$ and 
$\c A^{(m)}\in\op{Aut}_{\c P\text{-cont}}(\c L^{\c P})$,  
if and only if there is 
$l\in\c L^{\c P}$ such that 
$C^{(m)}=C^{0(m)}\circ l$ and $\c A^{(m)}=\Ad\,l\cdot\c A^{0(m)}$.  
\end{Prop}

\begin{definition} ${\c G}^{\c P}_{\omega }\subset {\c G}_{\omega}$ 
is a subgroup 
generated by $\c G_{<p}^{\c P}=\pi ^{-1}\c L^{\c P}$ and the lifts 
$\hat h_{\omega }^{0(m)}$, $1\leqslant m\leqslant N$.
\end{definition}

\begin{remark}
 
 (i) The elements of the group ${\c G}_{\omega }^{\c P}$ are specified via the 
 couples $(C,\c A)\in\c L_{\c K}^{\c P}\times \op{Aut}_{\c P\text{-cont}}\c L^{\c P}$ 
 (which satisfy relation \eqref{E4.2}). 
 
 (ii) The profinite completion of ${\c G}^{\c P}_{\omega }$ 
 coincides with ${\c G}_{\omega }$. 
\end{remark}

Obviously,  we have the following natural short exact sequences: 
\begin{equation} \label{E4.3}
1\To \c G_{<p}\To {\c G}_{\omega }\To \langle h_{\omega }^{(1)}, 
\dots ,h_{\omega }^{(N)}\rangle \To 1\, , 
\end{equation} 

\begin{equation} \label{E4.4}
1\To \c G^{\c P}_{<p}\To {\c G}^{\c P}_{\omega }
\To 
\langle h_{\omega }^{(1)}, \dots , h_{\omega }^{(N)}\rangle\To 1\, , 
\end{equation}

The structure of \eqref{E4.3} can be uniquely recovered from \eqref{E4.4} 
by going to profinite completions. 
\medskip

\subsection{The commutator subgroups 
$C_s({\c G}_{\omega }^{\c P})$} \label{S4.3}

Define the weight function  in $\c L_k^{\c P}$ 
by setting for $s\in\N $ and $(s-1)\bar c^0\leqslant a<s\bar c^0$, 
$$\op{wt} (D_{an})=s\, .$$

Introduce  
the ideal $\c L_{\bar c_0}^{\c P}(s)$ of $\c L^{\c P}$ 
such that $\c L_{\bar c^0}^{\c P}(s)_k$ is generated  
by all $[\dots [D_{a_1n_1},D_{a_2n_2}],\dots ,D_{a_rn_r}]$ 
with 
$\sum\limits _{i}\op{wt} (D_{a_in_i})\geqslant s$.  
Clearly, for any $s_1,s_2$, it holds   
$[\c L_{\bar c^0}^{\c P}(s_1),\c L_{\bar c^0}^{\c P}(s_2)]
\subset\c L_{\bar c^0}^{\c P}(s_1+s_2)$.  

Consider the lifts 
$h_{\omega }^{0(m)}\in {\c G}^{\c P}_{\omega }$ from Sect.\,\ref{S4.2}.
Denote by $\Ad ^{(m)}$ the automorphism of $G(\c L^{\c P})$ 
obtained from conjugation by $\hat h_{\omega }^{0(m)}$ on $\c G_{<p}^{\c P}$ 
with respect to the identification $\pi (=\pi _f(e)):\c G_{<p}^{\c P}\simeq G(\c L^{\c P})$. 

Let for $a\in\Z _N^+(p)$,  
$\op{Ad}^{(m)}_k(D_{a\,0})={D}^{(m)}_{a\,0}$ and 
 $\op{Ad}^{(m)}_k(D_{\bar 0})={D}^{(m)}_{\bar 0}$. 

\begin{Lem} \label{L4.3} For any $1\leqslant m,m'\leqslant N$,  

{\rm a)}\ ${D}^{(m)}_{\bar 0}\equiv D_{\bar 0}\,\op{mod}\,
\c L_{\bar c^0}^{\c P}(3)+\c L_{\bar c^0}^{\c P}(2)\cap C_2(\c L^{\c P})$;
\medskip  
 
{\rm b)}\ if $a=(\bar a^{(1)},\dots ,\bar a^{(N)})\in\Z _N^+(p)$ and 
$\op{wt}(D_{a\,n})=s$ then 
$${D}^{(m)}_{a\,0}\equiv D_{a\,0}-
\sum _{\imath \geqslant \bar 0}
A _{\imath }a^{(m)}D_{a+\bar c^0+p\imath ,0}$$ 
modulo $\c L_{\bar c_0 }^{\c P}(s+2)_k+
\c L_{\bar c_0}^{\c P}(s+1)_k\cap C_2(\c L^{\c P}_k)$, 
where the elements $A _{\iota }\in k$ are such that 
$E(\omega ^p)=1+\sum\limits  _{\iota\geqslant\bar 0}
A _{\iota }t^{\bar c^0+p\iota }$;
\medskip 

{\rm c)} the commutator $(\hat h_{\omega }^{0(m)}, 
\hat h_{\omega }^{0(m_1)})\in \pi ^{-1}G(\c L_{\bar c_0}^{\c P}(2))$. 
\end{Lem} 

We shall prove this lemma after finishing the proof of Prop.\,\ref{P4.4} below. 
 
Note that lemma \ref{L4.3} implies  $\pi C_2({\c G}_{\omega }^{\c P})
\subset G(\c L_{\bar c_0}^{\c P}(2))$. 

Set $\c L_{\omega }^{\c P}(1)=\c L^{\c P}$.

For $s\geqslant 2$, let $\c L_{\omega }^{\c P}(s)\subset\c L^{\c P}$ 
be such that $\pi C_s(\c G_{\omega }^{\c P})=G(\c L_{\omega }^{\c P}(s))$.

\begin{Prop} \label{P4.4} 
 For $1\leqslant s\leqslant p$, 
$\c L_{\omega }^{\c P}(s)=\c L_{\bar c^0 }^{\c P}(s)$. 
\end{Prop} 

\begin{proof} 
Use induction on $s\geqslant 1$.  

Clearly, $\c L^{\c P}_{\omega }(1)=\c L^{\c P}_{\bar c^0}(1)$.

Suppose $s_0\geqslant 1$ and for $1\leqslant s\leqslant s_0$, 
$\c L^{\c P}_{\omega }(s)=\c L^{\c P}_{\bar c^0}(s)$.

Let $\c L_{\op{lin}}^{\c P}=
\left (\sum\limits _{a,n}kD_{a\,n}\right )\cap\c L^{\c P}$ be 
\lq\lq the subspace of linear terms\rq\rq\  in $\c L^{\c P}$.

We have the following properties for all $s\leqslant s_0$:
\medskip 

-- \ $\c L^{\c P}_{\bar c^0}(s+1)=\c L_{\bar c^0}^{\c P}(s+1)
\cap\c L_{\op{lin}}^{\c P}
+\c L^{\c P}_{\bar c^0 }(s+1)\cap C_2(\c L^{\c P})$; 
\medskip 

-- \ $\c L^{\c P}_{\bar c^0}(s+1)\cap C_2(\c L^{\c P})=\sum\limits _{s_1+s_2=s+1} 
\left [\c L_{\bar c^0}^{\c P}(s_1),\c L_{\bar c^0}^{\c P}(s_2)\right ]$;
\medskip 

-- \ $\c L^{\c P}_{\omega }(s+1)$ is the ideal in $\c L^{\c P}$ generated by  
$[\c L^{\c P}_{\omega }(s),\c L^{\c P}]$ and the elements 
$\op{Ad}^{(m)}(l)\circ (-l)$, where $l\in \c L^{\c P}_{\omega }(s)$ 
and $1\leqslant m\leqslant N$. (If $s_0=1$ 
we do need part c) of Lemma \ref{L4.3}.) 
\medskip

Now statements a) and b) of Lemma \ref{L4.3} imply: 

(c1)\ {\it if $l\in \c L_{\bar c^0}^{\c P}(s)$ then 
$\op{Ad}^{(m)}(l)\circ (-l)\in \c L^{\c P}_{\bar c^0}(s+1)$};
\medskip 

(c2)\ {\it if $l\in \c L_{\op{lin}}^{\c P}\cap\c L^{\c P}_{\bar c^0}(s+1)$ 
then there are $m$ and 
$l'\in \c L_{\op{lin}}^{\c P}\cap\c L(s)$ such that} 
$$\op{Ad}^{(m)}(l')\circ (-l')\equiv l\,\op{mod}\,
\c L^{\c P}_{\bar c^0}(s+1)\cap C_2(\c L^{\c P})$$ 
({\it use that $A _{\bar 0}\ne 0$ 
and for any $a=(a^{(1)},\dots ,a^{(N)})\in\Z _N^+(p)$, there is $m$ such that} 
$a^{(m)}\not\equiv 0\,\op{mod}\,p$). 
\medskip

Then $[\c L^{\c P}_{\omega }(s_0),\c L^{\c P}]=
[\c L^{\c P}_{\bar c^0}(s_0), \c L^{\c P}(1)]\subset \c L^{\c P}_{\bar c^0}(s_0+1)$ 
and applying (c1) we obtain  
$\c L^{\c P}_{\omega }(s_0+1)\subset \c L^{\c P}_{\bar c^0}(s_0+1)$. 

For the opposite direction, note that by the inductive assumption, 

$$\c L^{\c P}_{\bar c^0}(s_0+1)\cap C_2(\c L^{\c P})=\sum\limits _{s_1+s_2=s_0+1}
\left [\c L^{\c P}_{\omega }(s_1),\c L_{\omega }^{\c P}(s_2)\right ]
\subset \c L_{\omega }^{\c P}(s_0+1)$$
and then (c2) implies that $\c L_{\op{lin}}^{\c P}\cap\c L^{\c P}_{\bar c^0}(s_0+1)
\subset \c L_{\omega }^{\c P}(s_0+1)$. As a result, 
$\c L_{\bar c^0}^{\c P}(s_0+1)\subset\c L_{\omega }^{\c P}(s_0+1)$ and  
our proposition is proved. 
\end{proof} 

\begin{proof} [Proof of Lemma \ref{L4.3}] 

Let 
$$\c N=\sum_{s\geqslant 1}t^{-\bar c^0s}\c L^{\c P}_{\bar c^0}(s)_{\m _{\c K}},$$
where $\m _{\c K}$ is the maximal ideal of the $N$-valuation ring $\c O_{\c K}$ of $\c K$. 
Clearly, $\c N$ has an induced structure of a Lie algebra over $\F _p$ 
and $e\in\c N$. 

Let 
$e^{(m)}:=(\op{Ad}^{(m)}_k \otimes\id _{\c K})e
=\sum \limits _{a\in\Z _N^+(p)}t^{-a}
{D}^{(m)}_{a\,0}+\alpha _{\bar 0}{D}^{(m)}_{\bar 0}$. 

The recovering of $C^{0(m)}\in G(\c L_{\c K}^{\c P})$ and $e^{(m)}$ 
from relation 
\begin{equation} \label{E4.5} 
(\id _{\c L^{\c P}}\otimes h_{\omega }^{(m)})e\circ C^{0(m)}=
(\sigma C^{0(m)})\circ {e}^{(m)}\,,
\end{equation}
  is a part of the recurrent procedure from Sect.\,\ref{S4.2}.  
Clearly, the operators $\c S$ and $\c R$ from Sect.\,\ref{S3.1}   
map $\c N$ to itself. Therefore, when following the recurrent procedure 
 we remain at each step in $\c N$. As a result, all 
$e^{(m)}, C^{0(m)}, \sigma C^{0(m)}\in\c N$. 

For any $j\geqslant 0$, introduce the ideals $\c N(j):=t^{\bar c^0j}\c N$ of $\c N$. 
The operators $\c R$ and $\c S$ also map the ideals $\c N(j)$ to itself.

The following properties are obtained by direct calculations: 
\medskip 

(i)\  $(\id_{\c L^{\c P}}\otimes h_{\omega }^{(m)})e=e+e^{(m)}_1\,\op{mod}\,\c N(2)$, 
$e^{(m)}_1=e_1^{(m)+}+e_1^{(m)-}\in\c N(1)$,   
$$e_1^{(m)-}=-{\hskip-8pt}\sum _{\substack{\iota\geqslant 
\bar 0\\ a\in\Z_N ^+(p)}}{\hskip -6pt}t^{-a}a^{(m)} 
A _{\iota }D_{a+\bar c^0+p\iota ,0}, \ \ 
e_1^{(m)+}=-{\hskip -10pt}\sum _{\substack{\iota\geqslant \bar 0\\ 
{\bar 0}<a<\bar c^0+p\iota }}{\hskip -8pt}
a^{(m)}A _{\iota }t^{-a+\bar c^0+p\iota }D_{a\,0}\, $$
(note that $e_1^{(m)+}\in\c L^{\c P}_{\m _{\c K}}$ and, 
therefore, $\c S(e_1^{(m)+})=0$);
\medskip 

(ii)\   the congruence $(\id _{\c L^{\c P}}\otimes h_{\omega }^{(m)})e\equiv 
e\,\op{mod}\,\c N(1)$ implies 
that $e^{(m)}\equiv e\,\op{mod}\,\c N(1)$ and 
$C^{0(m)}, \sigma C^{0(m)}\in \c N(1)$. Indeed, in the procedure of 
specification of $\hat h_{\omega }^{0(m)}$ it holds that for all $s$, 
$C^{(m)}_s,\sigma C^{(m)}_s\in\c N(1)$ and 
$(\c A^{(m)}_s\otimes\id _{\c K})e\equiv e\,\op{mod}\,\c N(1)$; 
\medskip 

(iii)\ \ \  
 $e^{(m)}=(-\sigma C^{0(m)})\circ 
 (\id _{\c L^{\c P}}\otimes h_{\omega }^{(m)})e\circ C^{0(m)}
\equiv (C^{0(m)}-\sigma C^{0(m)})+e+e^{(m)}_1\,\op{mod}\,
\c N(2)+t^{\bar c^0}\wt{\c N}^{(2)},$  
where 
$\wt{\c N}^{(2)}:=\sum \limits _{s\geqslant 2}t^{-s\bar c^0}
(\c L^{\c P}_{\bar c^0}(s)\cap C_2(\c L^{\c P}))_{\m _{\c K}}$ 
(use that $[\c N(1),\c N(1)]\subset\c N(2)$ and 
$[\c N(1),\c N]\subset t^{\bar c^0}\wt{\c N}^{(2)}$);
\medskip 

(iv)\  
$\c S(\c N(2)+t^{\bar c^0}\wt{\c N}^{(2)})\subset 
\c N(2)+t^{\bar c^0}\wt{\c N}^{(2)}$, 
$\c S(e^{(m)}-e-e_1^{(m)-})=e^{(m)}-e-e_1^{(m)-}$, 
$\c S(C^{(m)}-\sigma C^{(m)}+e_1^{(m)+})=0$. Therefore, item (iii) implies  
$$e^{(m)}\equiv e+e_1^{(m)-}\,\op{mod}\,\c N(2)+t^{\bar c^0}\wt{\c N}^{(2)}\, .$$  
More explicitly,   
\begin{equation} \label{E4.6}  e^{(m)}\equiv 
\sum \limits _{a\in\Z _N^+(p)}t^{-a}{\hskip-4pt}\left (D_{a\,0}-
a^{(m)}\sum\limits _{\iota\geqslant \bar 0}
A _{\imath}D_{a+\bar c^0+p\iota ,0}\right )+
\alpha _{\bar 0}D_{\bar 0}
\end{equation}
modulo $\c N(2)+t^{\bar c^0}\wt{\c N}^{(2)}$.  

Deduce from this congruence  statements a) and b)  of our lemma.
Consider the presentation of an element 
$l\in\c L^{\c P}_{\c K}$ in the form of a 
$\c P$-convergent series 
$l=\sum \limits _{b\in\Z _N}t^{b}l_{b}$, 
with all $l_{b}\in\c L^{\c P}_k$. 

\ \ 

Suppose $s\geqslant 1$ and $-(s-1)\bar c^0\geqslant b>-s\bar c^0$.

It follows directly from definitions that: 
\medskip 

--- \ if $l\in\c N$ then $l_{b}\in\c L^{\c P}_{\bar c^0}(s)_k$;
\medskip 

---\ if $l\in\c N(2)$ then $l_{b}\in\c L^{\c P}_{\bar c^0}(s+2)_k$;
\medskip 

--- \ if $l\in t^{\bar c^0}\wt{\c N}^{(2)}$ then 
$l_{b}\in \c L^{\c P}_{\bar c^0}(s+1)_k\cap C_2(\c L^{\c P}_k)$. 

As a result, the parts a) and b) of lemma are obtained by comparing 
coefficients in \eqref{E4.6}. 
 
 Now note that for any $m_1$, 
 \begin{equation} \label{E4.7} (\id _{\c L^{\c P}}\otimes 
 h_{\omega }^{(m_1)})e\equiv e+e_1^{(m_1)}\,\op{mod}\,\c N(2)\, .
 \end{equation} 
 
 Let $\c N_{<p}=\sum_{s\geqslant 1}
 t^{-\bar c^0s}\c L^{\c P}_{\bar c^0}(s)_{\m _{<p}},$
where $\m _{<p}$ is the maximal ideal of the $N$-valuation ring $\c O_{\c K_{<p}}$. 
Again, $\c N_{<p}$ has the induced structure 
of a Lie $\F _p$-algebra and 
for any $j\geqslant 0$, $\c N_{<p}(j)=t^{j\bar c_0}
\c N_{<p}$ is  ideal in $\c N_{<p}$. 

As earlier, $f, \sigma f\in \c N_{<p}$, and 
congruence \eqref{E4.7} implies that 
$$(\id _{\c L^{\c P}}\otimes h_{\omega }^{(m_1)}h_{\omega }^{(m))})e\equiv 
 e+e_1^{(m)}+e_1^{(m_1)}\,\op{mod}\,\c N_{<p}(2)\, .$$
where $f_1^{(m_1)}\in\c N_{<p}(1)$ is such that 
$\sigma f_1^{(m_1)}-f^{(m_1)}_1=e_1^{(m_1)}$. 

This implies 
 $$(\id _{\c L^{\c P}}\otimes \hat h_{\omega }^{0(m)}
 \hat h_{\omega }^{0(m_1)})f\equiv 
(\id _{\c L^{\c P}}\otimes \hat h_{\omega }^{0(m_1)}\hat h_{\omega }^{0(m)})f\equiv f+
f_1^{(m_1)}+f_1^{(m)}\,\op{mod}\,\c N_{<p}(2)\,$$
and, therefore, $(\id_{\c L^{\c P}} \otimes (\hat h_{\omega }^{0(m)},
\hat h_{\omega }^{0(m_1)}))f
\equiv f\,\op{mod}\,\c N_{<p}(2)$. 

On the other hand, the commutator $(\hat h_{\omega }^{0(m)}, \hat h_{\omega }^{0(m_1)})$ 
is a lift of $\id _{\c K}$, i.e. it 
coincides with $\pi ^{-1}(l_{mm_1})\in \c G_{<p}^{\c P}$. 
Therefore, $l_{mm_1}\in \c L^{\c P}\cap \c N_{<p}(2)=\c L_{\bar c^0}^{\c P}(2)$. 
The part c) is proved. 
\end{proof} 

\subsection{The group $\Gamma ^{\c P}_{\omega }$} \label{S4.4}

Let $\bar{\c L}=\c L/\c L_{\bar c_0}(p)$ and 
$\bar{\c L}^{\c P}=\c L^{\c P}/\c L_{\bar c_0}^{\c P}(p)$. 
Then $\bar{\c L}^{\c P}$ is dense in $\bar{\c L}$. 
If $\c K(p)=\c K_{<p}^{G(\c L_{\bar c^0}(p))}$ then $\bar{\c G}:=
\op{Gal}(\c K(p)/\c K)$ and the identification $\pi _f(e)$ induces the 
identification $\bar{\pi }:
\bar{\c G}\simeq G(\bar{\c L})$. This identification can be 
obtained via nilpotent Artin-Schreier theory: for $\bar e\in\bar{\c L}_{\c K}$ 
and $\bar f\in\bar{\c L}_{\c K(p)}$, we have 
$\sigma \bar f=\bar e\circ \bar f$ and 
$\bar\pi =\pi _{\bar f}(\bar e)$. However, the algebra $\bar{\c L}_{\c K}$ 
is too big for the process of linearization, cf.\,below. 
This motivates the following definitions. 

Let 
$$\c M:=\sum _{1\leqslant s<p}
 t^{-s\bar c^0}\c L_{\bar c^0}(s)_{\m _{\c K}}+\c L_{\bar c^0}(p)_{\c K}$$
$$\c M^{\c P}:=\sum _{1\leqslant s<p}
 t^{-s\bar c^0}\c L_{\bar c^0}^{\c P}(s)_{\m _{\c K}}+\c L_{\bar c^0}^{\c P}(p)_{\c K}$$
$$\c M_{<p}:=\sum _{1\leqslant s<p}
 t^{-s\bar c^0}\c L_{\bar c^0}(s)_{\m _{<p}}+\c L_{\bar c^0}(p)_{\c K_{<p}}$$  
 where (as earlier) $\m _{<p}$  is the maximal ideal 
of the $N$-valuation ring $\c O_{\c K_{<p}}$.  
Then $\c M$ and $\c M^{\c P}$ have the induced structure of  Lie 
$\F _p$-algebras (use the Lie bracket from $\c L_{\c K}$).  
For $s\geqslant 0$, $\c M(s):=t^{s\bar c^0}\c M$ and 
$\c M^{\c P}(s):=t^{s\bar c^0}\c M^{\c P}$ form a decreasing central 
filtrations of ideals in $\c M$ and $\c M^{\c P}$. 
 Similarly, 
 $\c M_{<p}$ is a Lie $\F _p$-algebra (containing $\c M$ 
 as its subalgebra),   
for $s\geqslant 0$, $\c M_{<p}(s):=t^{s\bar c^0}\c M_{<p}$ is a decreasing 
central 
filtration of ideals in $\c M_{<p}$ and      
$\c M_{<p}(s)\cap \c M^{\c P}=\c M^{\c P}(s)$. 
It can be easily seen that $e\in\c M^{\c P}$ and $f,\sigma f\in\c M_{<p}$.

There is a natural embedding  
$$\bar{\c M}^{\c P}:=\c M^{\c P}/\c M^{\c P}(p-1)\subset   
\bar{\c M}_{<p}:=\c M_{<p}/\c M_{<p}(p-1)\, ,$$
and the induced 
decreasing filtrations 
of ideals $\bar{\c M}^{\c P}(s)$ and $\bar{\c M}_{<p}(s)$ 
(where $\bar{\c M}^{\c P}(p-1)=\bar{\c M}_{<p}(p-1)=0$)  
are compatible with this embedding.  
For all $s\geqslant 0$,  
$(\id _{\bar{\c L}}\otimes h_{\omega }^{(m)}-
\id _{\bar{\c M}^{\c P}})^s\bar{\c M}^{\c P}
\subset \bar{\c M}^{\c P}(s)$.   

The algebras $\bar{\c M}^{\c P}$ and $\bar{\c M}_{<p}$ are naturally 
identified with the following 
subquotients of $\bar{\c L}^{\c P}_{\c K}$ and $\bar{\c L}_{<p}$: 

$$\bar{\c M}^{\c P}=
\left (\sum _{1\leqslant s<p}t^{-s\bar c^0}\bar{\c L}^{\c P}_{\bar c^0}
(s)_{\m }\right )\otimes 
\c O_{\c K}/t^{(p-1)\bar c^0}$$

$$\bar{\c M}_{<p}=
\left (\sum _{1\leqslant s<p}t^{-s\bar c^0}
\bar{\c L}_{\bar c^0}(s)_{\m _{<p}}\right )\otimes 
\c O_{\c K_{<p}}/t^{(p-1)\bar c^0}\,.$$

We can see easily that $\bar e\otimes 1\in\bar{\c M}^{\c P}$, 
$\bar f\otimes 1,\,\sigma \bar f\otimes 1\in\bar{\c M}_{<p}$ and 
$\sigma \bar f\otimes 1=(\bar e\otimes 1)\circ (\bar f\otimes 1)$. 
The following property shows that we still have full control 
of the identification $\bar\pi $.

\begin{Prop} \label{P4.5}  
The correspondence $\tau\mapsto 
(-\bar f\otimes 1)\circ \tau (\bar f\otimes 1)$ induces the 
natural projections $\c G_{<p}\To \bar{\c G}\simeq G(\bar{\c L})$ 
and ${\c G}^{\c P}_{<p}\To \bar{\c G}^{\c P}\simeq G(\bar{\c L}^{\c P})$.  
\end{Prop}

\begin{proof} 
$(-\bar f\otimes 1)\circ \tau \bar f\otimes 1$ comes 
from $(-f)\circ \tau f\in G(\c L)$. 
It remains to notice that $\c L\cap \c M(p-1)=\c L_{\bar c_0}(p)$ and 
$\c L\cap \c M^{\c P}(p-1)=\c L_{\bar c_0}^{\c P}(p)$. 
\end{proof} 

\begin{remark}
 In the above setting we can replace $\bar{\c M}_{<p}$ by its 
 analogue $\bar{\c M}_{\c K(p)}$, where the field $\c K(p)$ is used instead of 
 $\c K_{<p}$ (because $\bar f\in\bar{\c L}_{\c K(p)}$). 
\end{remark}

Let $\Gamma ^{\c P}_{\omega }:={\c G}^{\c P}_{\omega }/
({\c G}_{\omega }^{\c P})^pC_p({\c G}^{\c P}_{\omega })$. 
Then $\Gamma _{\omega }:={\c G}_{\omega }/
{\c G}_{\omega }^pC_p({\c G}_{\omega })$ can be recovered as the 
pro-finite completion of $\Gamma ^{\c P}_{\omega }$. 

\begin{Prop} \label{P4.6} 
Exact sequence \eqref{E4.3} induces the following 
exact sequence of profinite $p$-groups 
$$ 
1\To \bar{\c G}^{\c P}\To 
{\Gamma }^{\c P}_{\omega }\To \prod\limits _{1\leqslant m\leqslant N}
\langle h_{\omega }^{(m)}\rangle /
\langle h_{\omega }^{(m)p}\rangle \To 1\, .
$$
\end{Prop}
 
 \begin{proof}  
Consider the orbit of $\bar f\otimes 1$ with 
 respect to the natural action of ${\c G}_{\omega }
\subset \Aut \,\c K_{<p}$ on $f$ (recall that all ``values'' of $f$ belong to 
$\c K_{<p}\subset \c K_{sep}$). Then the stabilizer $\c H$ of $\bar f\otimes 1$ equals  
 ${\c G}_{\omega }^pC_p({\c G}_{\omega })$. 
This fact and the remaining part of the proof go along the lines 
of Prop.\,3.5 from \cite{Ab12}. 
\end{proof} 

Suppose $\pi _{\omega }:\Gamma _{\omega }^{\c P}\simeq G(L_{\omega }^{\c P})$ 
extends  
$\bar\pi =\pi _{\bar f}(\bar e)$ 
for  
a suitable Lie $\F _p$-algebra $L_{\omega }^{\c P}$ containing 
$\bar{\c L}^{\c P}$. Then the automorphisms $h_{\omega }^{(m)}$ 
give rise to Lie elements $l_{\omega }^{(m)}$ and we obtain the following property. 

\begin{Cor} \label{C4.7}  
There is a natural   
exact sequence of Lie $\F _p$-algebras 
\begin{equation} \label{E4.8} 0\To \bar{\c L}^{\c P}
\To L_{\omega }^{\c P}
\To\prod\limits _{1\leqslant m
\leqslant N} \F _pl_{\omega }^{(m)}\To 0\, .
\end{equation} 
\end{Cor} 

We recover the structure of $L^{\c P}_{\omega }$ below by 
analyzing the orbit of $\bar f$.  

\subsection{Filtered module  
$\bar{\c M}^f$ and the procedure of linearization}\label{S4.5} \ \ 

Let $\bar h_{\omega }^{(m)}\in \Gamma ^{\c P}_{\omega }$ be 
a lift of $h_{\omega }^{(m)}$ to $\c K(p)$. 
We use below the notation $\bar l_{\omega }^{(m)}$ for the corresponding 
element $\pi _{\omega }(\bar h_{\omega }^{(m)})\in L^{\c P}_{\omega }$. 
For example, cf.\,Sect.\,\ref{S4.3}, if $\bar l_{\omega }^{\,0(m)}=\pi _{\omega }
(\hat h_{\omega }^{0(m)})|_{\c K(p)})$ 
then the notation 
$\Ad\,^{(m)}$ appears as $\Ad\,(\bar l_{\omega }^{\,0(m)})$. 

Let $\Gamma ^{(m)\c P}_{\omega }$ be a subgroup in $\Gamma _{\omega }^{\c P}$ 
generated by $\bar h_{\omega }^{(m)}$ and $\bar{\c G}^{\c P}=
\bar{\pi }^{-1}G(\bar{\c L}^{\c P})$ 
(clearly, it does not depend on the choice of $\bar h_{\omega }^{(m)}$).  
Then we have the following exact sequence 
$$1\To \bar{\c G}^{\c P}\To \Gamma _{\omega }^{(m)\c P}\To 
\langle h_{\omega }^{(m)}\rangle /\langle h_{\omega }^{(m)p}\rangle \To 1\, .$$
 
Let $L^{(m)\c P}_{\omega }$ be a Lie subalgebra in $L^{\c P}_{\omega }$ 
such that $\pi _{\omega }(\Gamma ^{(m)\c P}_{\omega })=G(L^{(m)\c P}_{\omega })$. 

We obtain the following exact sequence of $\F _p$-Lie algebras 
$$0\To \bar{\c L}^{\c P}\To L^{(m)\c P}_{\omega }\To 
\F_pl_{\omega }^{(m)}\To 0.$$
obtained from \eqref{E4.8} via the natural embedding $\F_pl_{\omega }^{(m)}\To 
\prod\limits _{1\leqslant m\leqslant N}\F _pl_{\omega }^{(m)}$. 

The structure of the Lie algebras $L^{(m)\c P}_{\omega }$ (as well 
as the groups 
$\Gamma ^{(m)\c P}_{\omega }$) can be studied via the 
"linearization techniques" from \cite{Ab14, Ab13}. 

Namely, the action of $\id _{\bar{\c L}^{\c P}}\otimes h_{\omega }^{(m)}$ on 
$\bar{\c L}^{\c P}_{\c K}$ induces 
the action on $\bar{\c M}^{\c P}$, which can be presented in the form 
$\wt{\exp}(\id _{\bar{\c L}^{\c P}}\otimes dh_{\omega }^{(m)})$, where 
\linebreak 
$\id _{\bar{\c L}^{\c P}}\otimes dh_{\omega }^{(m)}$ is a derivation on 
$\bar{\c M}^{\c P}$. Indeed, the elements of 
$\bar{\c M}^{\c P}$ can be written  
uniquely as sums of elements of the form 
$l\otimes t^{-a}$, where for some $1\leqslant s<p$, 
$(s-1)\bar c^0\leqslant a<s\bar c^0$ and $l\in \bar{\c L}^{\c P}(s)_k$. 
Then this derivation comes from the correspondences  
$l\otimes t^{-a}\mapsto l\otimes (-a^{(m)})t^{-a}\omega ^p$.

Let $\bar{\c M}^f$ be the minimal Lie subalgebra 
in $\bar{\c M}_{\c K(p)}$ obtained by joining to  
$\bar{\c M}^{\c P}$ the element  $\bar f$ 
and its images under the action 
of the group generated by $\id _{\bar{\c L}}\otimes \bar h_{\omega }^{(m)}$. 
The algebra $\bar{\c M}^f$ still reflects all essential information about the structure of 
$\Gamma _{\omega }$. 
Then 
$$(\id _{\bar{\c L}}\otimes \bar h_{\omega }^{(m)})\bar f=\bar C^{(m)}
\circ (\bar{\c A}^{(m)}\otimes\id _{\c K(p)})\bar f\, ,$$ 
where 
$\bar C^{(m)}\in\bar{\c M}^{\c P}$ and 
$\bar{\c A}^{(m)}=\Ad\,\bar l_{\omega }^{(m)}
\in\Aut \bar{\c L}^{\c P}$. So, this relation determines the action of 
the lift $\id _{\bar{\c L}}
\otimes \bar h_{\omega }^{(m)}$ on $\bar{\c M}^f$.  

For any $n\geqslant 1$, we have 
$(\id\otimes \bar h_{\omega }^{(m)n})\,\bar f=\bar C^{(m,n)}
\circ (\bar{\c A}^{(m)n}\otimes\id )\bar f$,  where the element 
$\bar C^{(m,n)}\in\bar{\c M}^{\c P}$ can be presented in the following form  
$$(\id \otimes h_{\omega }^{(m)n-1})\bar C^{(m)}
\circ (\bar{\c A}^{(m)}\otimes h_{\omega }^{(m)n-2})\bar C^{(m)}\circ 
\ldots \circ (\bar{\c A}^{(m)n-1}\otimes\id )\bar C^{(m)}\, .$$ 

Let $\bar c_i^{(m)}\in\bar{\c M}^{\c P}$ be such that for all $1\leqslant n<p$, 
$\bar C^{(m,n)}=\sum\limits _{1\leqslant i<p}n^i\,\bar c_i^{(m)}$. 
(Such elements $\bar c_i^{(m)}$ are unique because $\op{det}((n^i))_{1\leqslant n,i<p}\ne 
0\,\op{mod}\,p$.) 
Summarizing our approach from \cite{Ab14, Ab13} we obtain: 
\medskip 

\begin{Prop} \label{P4.8} \ \ 

{\rm a)}    
$\bar{\c A}^{(m)}=\wt{\exp}\,\bar{\c B}^{(m)}$, where $\bar{\c B}^{(m)}$ is a 
derivation on $\bar{\c L}^{\c P}$;
\medskip 

{\rm b)}  
if $\alpha _p:=\Spec\,\F _p[U]$ with  $U^p=0$,  
is a finite group scheme with coaddition $\Delta U=
U\otimes 1+1\otimes U$ then the correspondence 
$$\bar h^{(m)U}:\bar f\mapsto (U\otimes \bar c^{(m)}_1+\dots +
U^{p-1}\otimes \bar c_{p-1}^{(m)})\circ 
(\bar{\c A}^{(m)U}\otimes\id )\bar f=$$ 
$$=(U\otimes \bar c^{(m)}_1+\dots +
U^{p-1}\otimes \bar c_{p-1}^{(m)})\circ 
\left (\sum _{0\leqslant n<p}U^n\otimes (\bar{\c B}^n/n!
\otimes\id )\right )\bar f$$ 
induces a coaction of $\alpha _p$ on $\bar{\c M}^f$;
\medskip 

{\rm c)}   
$\bar h_{\omega }^{(m)n}(\bar f)=\bar h_{\omega }^{(m)U}|_{U=n}$;
\medskip

{\rm d)}  if for all $a\in\Z _N^0(p)$, 
$V^{(m)}_{a\,0}:=\op{ad}\,\bar l_{\omega }^{(m)}
(D_{a\,0})$ then  
$$
\sigma \bar c^{(m)}_1-\bar c^{(m)}_1+
\sum _{a\in \Z _N^0(p)}t^{-a}V^{(m)}_{a\,0}= 
$$
$$-\sum _{1\leqslant k<p}\,\frac{1}{k!}\,t^{-(a_1+\dots +a_k)}
\omega ^pa_1^{(m)}
[\dots [D_{a_1\,0},D_{a_2\,0}],\dots ,D_{a_k\,0}]$$
$$-\sum _{2\leqslant k<p}\frac{1}{k!}t^{-(a _1+\dots +a_k)}
[\dots [V_{a_1\,0},D_{a_2\,0}],\dots ,D_{a_k\,0}]$$
\begin{equation} \label{E4.9} -\sum _{1\leqslant k<p}\,
\frac{1}{k!}\,t^{-(a_1+\dots +a_k)}
[\dots [\sigma \bar c^{(m)}_1,D_{a_10}],\dots ,D_{a_k0}]
\end{equation} 
(the indices  
$a_1,\dots ,a_k$ in all above sums run over $\Z _N^0(p)$);  
\medskip 

{\rm e)}   the  solutions 
$\{\bar c^{(m)}_1, V^{(m)}_{a0}\ |\ a\in\Z _N^0(p)\}$  
of \eqref{E4.9} are in bijection with the lifts 
$\bar h_{\omega }^{(m)}$ of $h_{\omega }^{(m)}$ to $\c K(p)$; 
\medskip 

{\rm f)}\   suppose $\bar c^{(m)}_1=\sum _{\iota\in\Z ^N}\bar c^{(m)}_1
(\iota)t^{\iota }$,
where all $\bar c^{(m)}_1(\iota )\in \bar{\c L}^{\c P}_k$; 
then different solutions of \eqref{E4.9} have different 
$\bar c^{(m)}_1(\bar 0)$, i.e. 
$\bar c^{(m)}_1(\bar 0)\in\bar{\c L}^{\c P}_k$ are strict invariants of 
the lifts $\bar h_{\omega }^{(m)}$. 
\end{Prop}  

\begin{proof}
 -- a) is just a general fact about the structure of unipotent automorphisms on 
 modules with filtration of length $<p$; 
 \medskip 
 
 -- b) this is also a sufficiently general interpretation of unipotent 
 additive action on modules with filtration of length $<p$  
 (Sect.3 of \cite{Ab14} contains  
 necessary background of the specification of this 
 situation to the Campbell-Hausdorff composition law.);
 \medskip 
 
 -- c) this follows obviously from b);  
 \medskip 
 
 -- d) note that the relations 
 $$(\id \otimes h_{\omega }^{(m)n})\bar e=(\sigma \bar C^{(m,n)})\circ 
 (\Ad\,^n\bar l_{\omega }^{(m)}\otimes\id )\bar e\circ (-\bar C^{(m,n)})$$ 
 imply that 
 $$
  (\id \otimes h_{\omega }^{(m)U})\bar e=(\sigma \bar C^{(m)U})\circ 
 (\Ad\,^U\bar l_{\omega }^{(m)}\otimes\id )\bar e\circ (-\bar C^{(m)U})\, ,
 $$
where $\bar C^{(m)U}=U\bar c_1^{(m)}+\dots +U^{p-1}\bar c^{(m)}_{p-1}$. 
This implies  
 $$(\bar c_1^{(m)}U)\circ (\id \otimes h_{\omega }^{(m)}U)\bar e
 \equiv (\sigma \bar c_1^{(m)}U)\circ 
 (\Ad ^U\bar l_{\omega }^{(m)}\otimes\id )\bar e
 \,\op{mod}\,U^2\,,$$
 and we need just to follow the coefficents for $U$; 
 
 \begin{remark} 
  Relation \eqref{E4.9} can be uniquely lifted to 
  $\bar{\c L}^{\c P}_{\c K}=\c L^{\c P}_{\c K}\,\op{mod}\,
  \c L^{\c P}_{\bar c_0}(p)_{\c K}$ by taking  suitable unique lifts 
  of $\bar c_1^{(m)}$ (use that 
  $\sigma $ is nilpotent on 
  $\c M(p-1)_{\c K}\,\op{mod}\,\c L^{\c P}_{\bar c_0}(p)_{\c K}$). 
  In other words, 
  we have unique lifts of $\bar c_1^{(m)}$ to $\bar{\c L}_{\c K}$ such that  
  \eqref{E4.9} is still an equality in 
  $\bar{\c L}^{\c P}_{\c K}$. 
  We will use the same notation $\bar c_1^{(m)}$ for these lifts.   
 \end{remark}

 -- e) note, first, the corresponding data 
 $\{\bar c_1^{(m)}, V^{(m)}_{\bar a0}\ |\ \bar a\in\Z _N^0(p)\}$ are in a bijection 
with the lifts 
$\bar h_{\omega }^{(m)}$ of $h_{\omega }^{(m)}$ to $\c K(p)$, cf. Sect.\,\ref{S1.6}. 
Therefore, we should verify that $\bar c_1^{(m)}$ 
determines uniquely the whole vector $\bar c^{(m)}$. 
This follows formally from b) and can be verified as follows 
 (we used a different approach in \cite {Ab13}, cf. Remark in Sect.\,\ref{S3.5}). 
 
Since $\bar h_{\omega }^{(m)U}:\bar{\c M}^f\To\F _p[U]\otimes\bar{\c M}^f$ 
is the coaction of the group scheme 
$\alpha _p$, we have in $\F _p[U_1,U_2]\otimes \bar{\c M}^f$ that 
$$(\id _{U_1}\otimes \bar h_{\omega }^{(m)U_2})\cdot 
\bar h_{\omega }^{(m)U_1}=\bar h_{\omega }^{(m)U_1+U_2}\, .$$  
Then we obtain in $\bar{\c L}_{\c K(p)}$, 
$$(\id \otimes  h_{\omega }^{(m)U_2})\bar C^{(m)U_1}
\circ (\Ad ^{U_1}\bar l_{\omega }^{(m)}\otimes \id )\bar C^{(m)U_2}
\circ (\Ad ^{U_1+U_2}\bar l_{\omega }^{(m)}\otimes\id )\bar f$$
$$=\bar C^{(m)U_1+U_2}\circ (\Ad ^{U_1+U_2}\bar l_{\omega }^{(m)}\otimes \id )\bar f$$ 
implies 
$$\left (\sum _{n\geqslant 1}U_1^{n}
(\id\otimes h_{\omega }^{(m)U_2})\bar c^{(m)}_{n}\right )\circ 
\left (\sum _{n\geqslant 1}U^{n}_2(\bar{\c A}^{U_1}
\otimes\id )\bar c^{(m)}_{n}\right )$$
$$=\sum _{n_1,n_2}\bar c^{(m)}_{n_1+n_2}(U_1+U_2)^{n_1+n_2}$$
For $n\geqslant 1$, the coefficient for $U_1U^n_2$ in the RHS equals 
$(n+1)\bar c_{n+1}^{(m)}$. The corresponding coeffcient in the LHS 
coincides with the coeffcient in   
$$\left (U_1
(\id\otimes h_{\omega }^{(m)U_2})\bar c^{(m)}_{1}\right )
\circ\left ((\id +
U_1\bar{\c B}^{(m)}\otimes\id )\sum _{n\geqslant 0}U_2^{n}\bar c^{(m)}_{n}\right )$$
and, therefore, equals 
$(\bar{\c B}^{(m)}\otimes\id )\bar c^{(m)}_n$ plus 
$\F _p$-linear combination of 
the elements of the following form (cf. Remark below)   
$$[\ldots [(\id\otimes d^{n_1}h_{\omega }^{(m)}/n_1!)c^{(m)}_1,\bar c_{n_2}^{(m)}]
,\cdots ],\bar c^{(m)}_{n_s}]\,,$$
where $n_1+\dots +n_s=n$, $s\geqslant 1$ and $n_1\geqslant 0$.
As a result, $(n+1)\bar c_{n+1}^{(m)}$ can 
be uniquely recovered from $\bar c^{(m)}_1, \dots ,\bar c_n^{(m)}$.

\begin{remark} 
We used well-known relation, 
$$
X+UY\equiv X\circ 
\left (U\sum _{k\geqslant 1}\frac{1}{k!}
[\dots [Y,\underset{k-1\text{ times }}{\underbrace{X],
\dots ,X]}}\right )\,\op{mod}\,U^2
$$
with $U=U_1$ and $X=\sum _{n}U_2^n\bar c^{(m)}_n$, cf. references in \cite{Ab14}, 
Sect.\,3.2. 
\end{remark} 

f) follows by induction on $\op{mod}\,C_i(\bar{\c L}^{\c P}_{\c K})$, 
$i\geqslant 1$, from relation 
\eqref{E4.9}. 
\end{proof}

\subsection{The structure of $\Gamma _{\omega }^{\c P}$} \label{S4.6}  

We are going to determine the structure of the Lie algebra $L^{\c P}_{\omega }$. 
In the above section we indicated the way how to specify the lifts 
$\bar h_{\omega }^{(m)}=
\pi _{\omega }^{-1}(\bar l_{\omega }^{(m)})$, $1\leqslant m\leqslant N$. 
This can be done by applying recurrent procedure 
\eqref{E4.9} to find the elements $\bar c_1^{(m)}$ and $V_{a\,0}^{(m)}=
\ad _k\, \bar l_{\omega }^{(m)}(D_{a0})$, $a\in\Z_N ^0(p)$. 
In addition, we should specify the commutators 
$[\bar l_{\omega }^{(i)},\bar l_{\omega }^{(j)}]:=
\bar l_{\omega }[i,j]\in \bar{\c L}^{\c P}$, 
$1\leqslant i,j\leqslant N$. 

\begin{remark} 
From Lemma \ref{L4.3}\,c) it follows that all $\bar l_{\omega }[i,j]
\in C_2(L_{\omega }^{\c P})=\bar{\c L}_{\bar c^0}(2)$.
It would be very interesting to find explicit expression for the elements 
$\bar l_{\omega }[i,j]$ in terms of the involved parameters 
$A_{\iota }=A_{\iota }(\omega )$ 
(recall that $E(\omega ^p)=1+\sum _{\iota }
A_{\iota }t^{\bar c^0+p\iota }$). 
We verified by a direct calculation with relations \eqref{E4.9} 
that for any lifts 
$\bar h_{\omega }^{(i)}$ and $\bar h_{\omega }^{(j)}$, 
the corresponding elements $\bar l_{\omega }[i,j]\in C_4(L_{\omega }^{\c P})$. 
\end{remark} 

The following property could be useful to study 
the properties of the elements $\bar l_{\omega }[i,j]$.  
To simplify the notation set for all $m$, 
$\id\otimes dh_{\omega }^{(m)}=d^{(m)}$ and 
$\ad\,\bar l_{\omega }^{(m)}\otimes\id =\ad ^{(m)}$. 

\begin{Prop} \label{P4.9} For all $1\leqslant i,j\leqslant N$, 
 $$\bar l_{\omega }[i,j]=(d^{(j)}-\ad ^{(j)})\bar c_1^{(i)}
 -(d^{(i)}-\ad ^{(i)})
 \bar c_1^{(j)}+[\bar c_1^{(i)},\bar c_1^{(j)}]\,.$$ 
\end{Prop}

\begin{proof} With above assumption we have in $\bar{\c L}^{\c P}_{\c K}[U]$ 
for all $m$, 
$$(\id +Ud^{(m)})\bar e\equiv (U\sigma\bar c_1^{(m)})\circ (\id +U\ad ^{(m)})
\bar e\circ (-U\bar c^{(m)}_1)\,\op{mod}\,U^2.$$

Let $\bar E=\wt{\exp}\,\bar e$ in 
the enveloping algebra $\bar{\frak A}^{\c P}_{\c K}$ 
of $\bar{\c L}^{\c P}_{\c K}$. 
Then we have the following congruence modulo $(U^2, (\bar{\c J}^{\c P}_{\c K})^p)$, 
where $\bar{\c J}^{\c P}_{\c K}$ is the augmentation ideal in 
$\bar{\frak A}^{\c P}_{\c K}$, 
$$(\id +Ud^{(m)})\bar E\equiv (1+\sigma\bar c^{(m)}_1U)\cdot (\id +U\ad ^{(m)})
\bar E\cdot (1-\bar c^{(m)}_1U)\, .$$
Comparing the coefficients for $U$ we obtain  
\begin{equation} \label{E4.10} 
(d^{(m)}-\ad ^{(m)})\bar E\equiv \sigma \bar c^{(m)}_1
\cdot \bar E-\bar E\cdot \bar c^{(m)}_1\, 
\op{mod}\,(\bar{\c J}^{\c P}_{\c K})^p\,. 
\end{equation}
Note that $(d^{(i)}d^{(j)}-d^{(j)}d^{(i)})\bar E\equiv 0$ because   
$$d^{(m)}\bar E=-\sum _{\bar\iota\geqslant\bar 0}A_{\bar\iota }
\sum _{s\geqslant 0}(1/s!)(a_1+\dots +a_s)^{(m)}
t^{-(a_1+\dots +a_s)+\bar c^0+p\bar\iota }D_{a\,0}\, .$$
In addition, for any $i,j$, $\ad ^{(i)}$ and $d^{(j)}$ commute one with another.   

Therefore, \eqref{E4.10} implies 
$$\ad\,\bar l_{\omega }[i,j](\bar E)\equiv (d^{(i)}-\ad ^{(i)})(d^{(j)}-\ad ^{(j)})
\bar E-(d^{(j)}-\ad ^{(j)})(d^{(i)}-\ad ^{(i)})\bar E$$ 
$$\equiv \sigma (\c X)\bar E-\bar E\c X\, ,$$ 
where $\c X=(d^{(i)}-\ad ^{(i)})\bar c^{(j)}_1-(d^{(j)}-\ad ^{(j)})\bar c^{(i)}_1+
[\bar c^{(j)}_1,\bar c^{(i)}_1]$. 
\medskip 

Let $\c X_0=\c X-l_{\omega }[i,j]$, then 
\begin{equation} \label{E4.11} \sigma (\c X_0)\bar E
\equiv\bar E\c X_0\,\op{mod}(\bar{\c J}_{\c K}^{\c P})^p\,. 
\end{equation}

It remains to prove that 
\end{proof} 

\begin{Lem} \label{L4.10} 
 $\c X_0=0$. 
\end{Lem}

\begin{proof} [Proof of Lemma] 
As earlier, let $\bar{\frak A}^{\c P}$ be the enveloping algebra for  
$\bar{\c L}^{\c P}$. Let $\bar{\frak A}^{\c P}_{\bar c^0}(s)$, $s\geqslant 1$, 
be the ideal in $\bar{\frak A}^{\c P}$ 
generated by the monomials $D_{a_1n_1}\dots D_{a_un_u}$ of weight $\geqslant s$, 
i.e.\,such that if $s_i\in\N $ for $1\leqslant i\leqslant u$, 
are such that 
$(s_i-1)\bar c^0\leqslant a_i<s_i\bar c_0$ then 
$s_1+\dots +s_u\geqslant s$. For each $s$, 
$\bar{\frak A}^{\c P}_{\bar c^0}(s)\cap\bar{\c L}^{\c P}_{\bar c^0}
=\bar{\c L}^{\c P}_{\bar c^0}(s)$, in particular, 
$\bar{\frak A}^{\c P}_{\bar c^0}(p)\cap\bar{\c L}^{\c P}_{\bar c^0}=0$.

Note that $\c X_0\in\bar{\c L}^{\c P}_{\c K}$ and let 
$\c X_0 =\c X^{(1)}+\dots +\c X^{(p-1)}_0$, 
where each  
$\c X_0^{(s)}$ is a $\c K$-linear combination of the Lie monomials of weight $s$, 
$$[\dots [D_{a_1n_1},D_{a_2n_2}],\dots ,D_{a_un_u}]\, .$$  
Clearly,  $\c X_0^{(s)}\in\bar{\c L}^{\c P}_{\bar c^0}(s)$ and if 
$\c X_0^{(s)}\in\bar{\frak A}^{\c P}_{\bar c^0}(s+1)$ then 
$\c X_0^{(s)}=0$. 
It will be enough to prove by induction on $1\leqslant s\leqslant p$ 
that 
$$\c X_0\in\bar{\c L}_{\bar c^0}^{\c P}(s)+
\bar{\frak A}^{\c P}_{\bar c^0}(s+1)_{\c K}\,.$$

If $s=1$ then \eqref{E4.11} implies that $\sigma (\c X_0)\equiv 
\c X_0\,\op{mod}\,\bar{\frak A}_{\bar c^0}^{\c P}(2)_{\c K}$ 
(use that $\bar E\equiv 1\,\op{mod}\,\bar{\frak A}^{\c P}_{\bar c^0}(1)_{\c K}$). 
So, 
$\sigma\c X_0^{(1)}=\c X_0^{(1)}$ in $\bar{\c L}^{\c P}_{\c K}$, i.e. 
$\c X_0^{(1)}\in\bar{\c L}^{\c P}=\bar{\c L}^{\c P}_{\bar c^0}(1)$ 
and $\c X_0\in\bar{\c L}_{\bar c^0}^{\c P}(1)+
\bar{\frak A}^{\c P}_{\bar c^0}(2)_{\c K}$. 

Suppose the lemma is proved for some $1\leqslant s\leqslant p-1$. 

So, $\c X_0=\c X_0^{(s)}+\dots +\c X_0^{(p-1)}
\in\bar{\c L}^{\c P}_{\bar c^0}(s)+\bar{\frak A}_{\bar c^0}^{\c P}(s+1)_{\c K}$, 
and \eqref{E4.11} implies

$$\sigma\c X_0^{(s+1)}-\c X_0^{(s+1)}\equiv 
\sum _{a\in\Z _N^0(p)}t^{-a}[D_{a0},\c X^{(s)}_0]\,
\op{mod}\,\bar{\frak A}_{\bar c^0}^{\c P}(s+2)_{\c K}\,.$$

Thus (compare with Prop.\,\ref{P3.1}) 
all $[D_{a0},\c X_0^{(s)}]
\in\bar{\frak A}^{\c P}_{\bar c^0}(s+2)_{\c K}$ 
and 
$$\c X_0^{(s+1)}\in \bar{\c L}_{\bar c^0}^{\c P}(s+1)+
\bar{\frak A}^{\c P}_{\bar c^0}(s+2)_{\c K}
\,.$$
If weight of $D_{a0}=1$ then $[D_{a0},\c X_0^{(s)}]$ has weight 
$s+1$, therefore, $[D_{a0}, \c X_0^{(s)}]=0$ and  
$\c X_0^{(s)}=0$. Similarly, 
$\c X_0^{(s+1)}\in\bar{\c L}^{\c P}_{\bar c^0}(s+1)$. 

The lemma is proved. 
\end{proof}

Using the notation from item f) of Prop.\,\ref{P4.8} we obtain 

\begin{Cor}\,\label{C4.11}
 For $1\leqslant i,j\leqslant N$, 
 $$\bar l_{\omega }[i,j]=\ad ^{(i)}(\bar c_1^{(j)}(\bar 0))-
 \ad ^{(j)}(\bar c^{(i)}_1(\bar 0))+
 \sum _{\iota\in\Z ^N}[\bar c_1^{(i)}(\iota ), \bar c_1^{(j)}(-\iota )]\,.$$
\end{Cor}

\subsection{The structure of $\Gamma _{\omega }^{\c P}$ 
modulo third commutators} \label{S4.7}  

Consider the lift of relation \eqref{E4.9} 
to $\bar{\c L}^{\c P}_{\c K}$ taken 
modulo  
$C_2(L_{\omega }^{\c P})_{\c K}=\bar{\c L}^{\c P}_{\bar c^0}(2)_{\c K}$ 
(cf. Remark in the part d) of the proof of Prop.\,\ref{P4.8}). 

\begin{equation} \label{E4.12} \sigma\bar c^{(m)}_1-\bar c^{(m)}_1+
\sum _{a\in\Z _N^0(p)}t^{-a}V^{(m)}_{a0}\equiv 
-\sum _{\substack{a\in \Z _N^0(p)\\ \iota\geqslant \bar 0}}
A_{\iota }(\omega )t^{\bar c^0+p\iota -a}a^{(m)}
D_{\bar a0}\, .
\end{equation}

Here $V_{a0}^{(m)}=\ad\,\bar l_{\omega }^{(m)}(D_{a0})$ and 
$a^{(m)}$ is the $m$-th component of $ a\in\Z _N^0(p)$. 
Applying to \eqref{E4.12} the operator $\c R$ from Sect.\,\ref{S3.1} we obtain:
\medskip 

1)\ $V^{(m)}_{\bar 0}=\op{ad}\,\bar l_{\omega }^{(m)}(D_{\bar 0})
\in C_2(L_{\omega }^{\c P})$;
\medskip 

2)\ for all $a\in\Z _N^+(p)$, 
$$V^{(m)}_{a0}=\op{ad}\,l_{\omega }^{(m)}(D_{a0})\in 
\,\op{mod}\,C_2(L_{\omega ,k}^{\c P})\, .$$

Property 2) means that 
all generators of ${L}_{\omega ,k}^{(m)\c P}$ of the form 
$D_{an}$ with $ a>\bar c^0$ and $a^{(m)}\ne 0\,\op{mod}\, p$, 
can be eliminated from 
the system 
$$\{D_{an}\ |\ a\in\Z _N^+(p)\}\cup\{D_{\bar 0}\}\cup \{l_{\omega }^{(m)}\}$$
of  generators of $L^{(m)\c P}_{\omega ,k}$. Indeed,  
because $A_{\bar 0}\ne 0$ and $ a^{(m)}\ne 0\,\op{mod}\,p$, all 
$D_{ a+\bar c^0,0}$ belong to the 
ideal of second commutators 
$$C_2(L^{(m)\c P}_{\omega ,k})=(\op{ad}\,\bar l_{\omega }^{(m)})\bar{\c L}_k^{\c P}+
C_2(\bar{\c L}_k^{\c P})\, ,$$ 
and  
for any $n\in\Z /N_0$, all 
$D_{a+\bar c^0,n}=\sigma ^nD_{a+\bar c^0,0}$ 
also belong to $C_2(L^{(m)\c P}_{\omega ,k})$. 
Property 1) means that $L^{(m)\c P}_{\omega }$ has only one relation 
with respect to any minimal $\c P$-topological set of generators. 
 Therefore, 
$\Gamma _{\omega }^{(m)\c P}$ can be treated as a 
$\c P$-topological group with one relation.

On the other hand, the Lie algebra $L^{\c P}_{\omega ,k}$ 
has the system of generators 
$$\{D_{an}\ |\ a\in\Z _N^+(p)\}\cup\{D_{\bar 0}\}
\cup \{\bar l_{\omega }^{(m)}\ |\ 1\leqslant m\leqslant N\}$$ 
with the corresponding system of relations 
$$[D_{an}, \bar l_{\omega }^{(m)}]=V^{(m)}_{an}, \ 
[D_{\bar 0},\bar l_{\omega }^{(m)}]=V^{(m)}_{\bar 0},\ [\bar l_{\omega }^{(i)},
\bar l_{\omega }^{(j)}]=\bar l_{\omega }[i,j]\,,$$
where 
$a\in\Z _N^+(p)$ and  $1\leqslant m,i,j\leqslant N$. 

Choose for every $a\in\Z _N^+(p)$, $m_{a}\in [1,N]$ such that 
$a^{(m_{a})}\ne 0\,\op{mod}\,p$. Then the relations 
$[D_{an},\bar l_{\omega }^{(m_{a})}]=V_{an}^{(m_{a})}$, 
$a\in\Z _N^+(p)$, 
can be used to eliminate extra generators $\{D_{an}\ |\ a>\bar c_0\}$ 
and to present the structure of $L_{\omega ,k}^{\c P}$ in terms of 
the corresponding minimal system of generators 
$\{D_{an}\ |\ a\in \Z_N^+(p), a<\bar c^0\}\cup
\{D_{\bar 0}\}\cup\{\bar l_{\omega }^{(m)}\}$.

Consider second central step to obtain explicitly the above relations 
modulo $C_3(L^{\c P}_{\omega ,k})$.  

\begin{Prop} \label{P4.12} For $1\leqslant m\leqslant N$ 
and $a\in\Z ^+_N(p)$, 
there are the following congruences 
$\op{mod}\,C_3(L_{\omega ,k}^{\c P})$: 

$$V^{(m)}_{\bar 0}\equiv 
-\frac{1}{2}\sum _{\substack{\iota\geqslant\bar 0 
\\ 0\leqslant n<N_0}}
\sigma ^n\left (A_{\iota }
\sum _{\substack{a_1+a_2=\\
\bar c^0+p\iota }}a_1^{(m)}[D_{a_1,0},D_{a_2,0}]\right )\, ,
$$
$$V^{(m)}_{a0}\equiv 
-\sum _{\substack {n\geqslant 1 \\ \iota\geqslant 0}}
\sigma ^{n}\left (A_{\iota }\hskip -8pt 
\sum _{\substack{a_1+a_2/p^{n}\\ 
=\bar c^0+p\iota +a/p^{n}}}\hskip -8pt a_1^{(m)} 
[D_{a_1,0}, D_{a_2,-n}]\right )$$

$$-\sum _{\imath\geqslant\bar 0}A_{\iota }a^{(m)}
D_{\bar c^0+p\imath +a,0}-\frac{1}{2}
\sum _{\substack{n\geqslant 0\\ \iota\geqslant\bar 0 }}
\sigma ^{-n}\left (A_{\iota }\hskip -8pt 
\sum _{\substack{a_1+a_2=\\ 
\bar c^0+p\iota +ap^n}}\hskip -8pt 
a_1^{(m)}[D_{a_1,0}, D_{a_2,0}]\right )
\,.$$
\end{Prop} 
\medskip

\begin{proof}  
From \eqref{E4.11} we obtain 
(apply the operator $\c S$ from Sect.\,\ref{S3.1})  

$$\bar c^{(m)}_1\equiv  
\sum _{\substack{a,  
\iota\geqslant\bar 0 \\ n\geqslant 0}}
\sigma ^n(A_{\iota })t^{p^n(\bar c^0+p\iota -a)}
a^{(m)}D_{an}\,\op{mod}\,C_2({L}^{\c P}_{\omega ,\c K})\,.$$  

Then the right-hand side of \eqref{E4.9} modulo 
$C_3(L_{\omega ,\c K}^{\c P})$ appears as 

$$-\sum _{
\iota }A_{\iota }
t^{\bar c^0+p\iota -a}a^{(m)}D_{a0}-
\frac{1}{2}\sum _{\substack{a_1,a_2,  
\iota }}
A_{\imath }
t^{\bar c^0+p\iota -a_1-a_2}a^{(m)}_1
[D_{a_10},D_{a_20}]$$

$$-\sum _{\substack{a_1, a_2,  
\iota \\ n\geqslant 1}}\sigma ^{n}(A_{\iota })
t^{p^{n}(\bar c^0+p\iota -a_1)-a_2}
a^{(m)}_1[D_{a_1,n},D_{a_20}]$$

Applying the operators $\c R$ and $\c S$ we obtain our proposition. 
\end{proof} 

\begin{Cor} \label{C4.13} 
 $L^{\c P}_{\omega ,k}\,\op{mod}\,C_3(L^{\c P}_{\omega ,k})$ 
 is the maximal $\sigma $-invariant 
 quotient of nilpotent class 2 of a free Lie $k$-algebra 
 with generators 
 $$\{D_{an}\ |\ a\in\Z _N^+(p), a<\bar c^0, n\in\Z /N_0\}\cup\{D_{\bar 0}\}
 \cup \{\bar l_{\omega }^{(m)}\ |\ 1\leqslant m\leqslant N\}$$ 
satisfying for $1\leqslant m_1,m_2\leqslant N$ and $a\in\Z _N^+(p)$,  
the relations:  

$$\c R(m_1,m_2): [\bar l_{\omega }^{(m_1)}, \bar l_{\omega }^{(m_2)}]=0\, ,$$
$$\c R_{\bar 0}(m):\ \ \ [D_{\bar 0},\bar l_{\omega }^{(m)}]+\frac{1}{2}
\sum _{\substack{\iota \geqslant\bar 0
\\ 0\leqslant n<N_0}}
\sigma ^n\left (A_{\iota }
\sum _{\substack{a_1+a_2=\\
\bar c^0+p\iota }}a_1^{(m)}[D_{a_1,0},D_{a_2,0}]\right )=0\, ,
$$  
 
$$\c R_{a}(m_1,m_2): 
\sum _{\substack{n\geqslant 1\\ \iota \geqslant\bar 0} }
\sigma ^{n}\left (A_{\iota }\hskip -14pt 
\sum _{\substack{a_1+a_2/p^{n}\\ 
=\bar c^0+p\iota +a/p^{n}}}\hskip -12pt\left (a_1^{(m_1)}a_2^{(m_2)}-
a_1^{(m_2)}a_2^{(m_1)}\right )
[D_{a_1,0}, D_{a_2,-n}]\right )$$

$$+\frac{1}{2}
\sum _{\substack{n\geqslant 0\\  \iota \geqslant \bar 0}}
\sigma ^{-n}\left (A_{\iota }\hskip -8pt 
\sum _{\substack{a_1+a_2=\\ 
\bar c^0+p\iota +ap^n}}\hskip -8pt 
\left (a_1^{(m_1)}a_2^{(m_2)}-
a_1^{(m_2)}a_2^{(m_1)}\right )[D_{a_1,0}, D_{a_2,0}]\right )=0\,. $$
\end{Cor}

\begin{remark}
 a) If $N=1$ then there is only one (Demushkin) relation 
 $\c R_{\bar 0}(1)$, cf. \cite{Ab14, Ab13}.
 
 b) The simplest example can be obtained by choosing $\bar c^0=(p,0,\ldots ,0)$, 
 $\omega ^p=t^{\bar c^0}$ and $N_0=1$, cf. Sect.\,\ref{S5.6}.
 
 c) The structure of $\Gamma ^{\c P}_{\omega }$ 
 depends only on $\omega $, more precisely, only on 
 $\omega ^p\,\op{mod}\,t^{(p-1)\bar c^0}$, i.e. on the constants 
 $A_{\imath }$ with $\iota <(p-2)\bar c^0/p$. 
 
 d) The structure of $\Gamma ^{\c P}_{\omega }$ 
 modulo $C_s(\Gamma ^{\c P}_{\omega })$, $s\leqslant p$, depends only 
 on the constants $A_{\iota }$ with $\iota <(s-2)\bar c^0/p$.
\end{remark}

\subsection{The simplest example} \label{S4.8} $N=2$, $N_0=1$, $\bar c^0=(p,0)$, 
$A_{\bar 0}=1$, all remaining $A_{\iota }=0$. 

The minimal generators: 

$\{D_a\ |\ a\in\Z _2^+(p), a<(p,0)\}
\cup \{D_{\bar 0}\}\cup \{\bar l^{(1)}, \bar l^{(2)}\}$. 

The relations:

$$\c R(1,2): [\bar l^{(1)}, \bar l^{(2)}]\, ,$$

$$\c R_{\bar 0}(1)=[D_{\bar 0}, \bar l^{(1)}]+
\sum _{\substack{1\leqslant\alpha \leqslant \frac{p-1}{2}\\ \gamma\in\Z }}
\alpha [D_{(\alpha ,\gamma )}, D_{(p-\alpha ,-\gamma )}]
$$

$$\c R_{\bar 0}(2)=[D_{\bar 0}, \bar l^{(2)}]+
\sum _{\substack{0\leqslant\alpha \leqslant \frac{p-1}{2}\\ \gamma\in\Z }}
\gamma [D_{(\alpha ,\gamma )}, D_{(p-\alpha ,-\gamma )}]$$
 
$$\c R_{a}(1,2):\ \ [D_{a}, a^{(2)}\bar l^{(1)}-
a^{(1)}\bar l^{(2)}]-\delta _{a^{(1)},0}\sum _{\gamma }a^{(2)}[D_{(p-1,\gamma )}
,D_{(p,a^{(2)}-p\gamma )}]$$
$$+a^{(1)}\hskip -10pt
\sum _{\substack{n\geqslant 1\\ \beta\in\Z ^+(p)}}\hskip -10pt \beta 
[D_{(p,-\beta )}, D_{a+(0,p^n\beta )}]+\frac{1}{2}\hskip -6pt  
\sum _{\substack{a_1+a_2=\\ (p,0)+a}}\hskip -8pt 
\left (a_1^{(1)}a_2^{(2)}-
a_1^{(2)}a_2^{(1)}\right )[D_{a_1}, D_{a_2}]\,. $$

\section{Characteristic 0 case} \label{S5} 

In this section $K$ is an $N$-dimensional local field 
of characteristic 0. We assume that the first residue field $K^{(1)}$ of $K$ has 
characteristic $p$. The last residue field $k$ of $K$ 
is isomorphic to $\mathbb F _{p^{N_0}}$, $N_0\in\mathbb N $. We also fix a system 
of local parameters 
$\pi =\{\pi _1,\cdots ,\pi _N\}$ of $K$, denote by $v^1$ the first valuation 
of $K$ such that $v^1(p)=1$ and by $O_K^1$ the corresponding valuation ring. 
Starting Sect.\,\ref{S5.2} we assume that 
$K$ contains a primitive $p$-th root of unity $\zeta _1$.

\subsection{The field-of-norms functor}\label{S5.1}

$N$-dimensional local fields are  
special cases of $(N-1)$-big fields 
used by Scholl \cite{Sch}  
to construct a higher dimensional analogue of 
the Fontaine-Wintenberger field-of-norms functor, \cite{Wtb1}. 
We can apply this construction  to the case of higher local fields 
due to the fact that 
the structure of an $N$-dimensional field is  
uniquely extended to its finite field extensions. 
We don't use here the construction of the 
field-of-norms functor from \cite{Ab10}: 
it is based on essentially close ideas but works in the category of 
higher local fields with additional structure given by 
``subfields of constants'' (because the whole theory in \cite{Ab10} 
is based only on the concept of ramification for higher local fields). 

Let $K_{alg}$ be an algebraic closure of $K$. 
Denote by the same symbol a unique 
extension of the valuation $v^{1}$ to $K_{alg}$. 
For any $0<c\leqslant 1$, let $\frak p ^c=
\{x\in K_{alg}\ |\ v^1(x)\geqslant c\}$. If $L$ is a 
field extension of $K$ in $K_{alg}$, we use 
the simpler notation $O^1_L/\frak p ^c$ instead of 
$O^1_L/(\frak p ^c\cap O^1_L)$.

An increasing fields tower $K_\d=(K_n)_{n\geqslant 0}$, where $K_0=K$,   
is strictly deeply ramified (SDR) with parameters $(n_0,c)$,  if for 
$n\geqslant n_0$, we have 
$[K_{n+1}:K_n]=p^N$ and 
there is a surjective map 
$\Omega^1_{O^1_{F_{n+1}}/O^1_{F_n}} \To (O^1_{F_{n+1}}/\frak p ^c)^{N}$ 
or, equivalently, 
the $p$-th power map induces epimorphic maps 
$$ 
i^1_n(K_\d ):O^1_{K_{n+1}}/\frak p ^c\longrightarrow O^1_{K_n}/\frak p ^c.
$$
This implies that for all $n\geqslant n_0$, the last residue fields   
$K_n^{(N)}$ are the same and 
there are systems of local parameters $\{\pi _{n1},\dots ,\pi _{nN}\}$ in 
$K_n$ such that for all $1\leqslant m\leqslant N$, 
$\pi _{n+1,m}^{p}\equiv \pi _{nm}
\,\mathrm{mod}\,\frak p ^c$, where $\pi _m=\pi _{0m}$. Equivalently, on the level of the 
$N$-valuation rings $\c O_{K_n}$, 
the $p$-th power map induces epimorphic maps 
\begin{equation} \label{E5.1}
i_n(K_\d ):\c O_{K_{n+1}}/\frak p ^c\longrightarrow \c O_{K_n}/\frak p ^c.
\end{equation}

Let 
$\c O=\underset{n}\varprojlim\c O_{K_n}/\frak p ^c$. Then $\c O$ is an 
integral domain and we can introduce 
the field of fractions $\c K$ of $\c O$. 
The field-of-norms functor $X$ associates to the  SDR tower $K_\d $  
the field $\c K=X(K_{\d })$. This field has characteristic $p$,  
it inherits a structure 
of $N$-dimensional local field such that the elements 
$t_m:=\underset {n}\varprojlim\,\pi_{nm}$, $1\leqslant m\leqslant N$, 
form a system of local parameters in $\c K$. 
Then the  
$N$-dimensional valuation ring $\c O_{\c K}$ of $\c K$ coincides with $\c O$,  
and for $n\geqslant n_0$, 
the last residue fields of $\c K$ and $K_n$ coincide. Since the 
identification $\c O_{\c K}=\varprojlim \c O_{K_n}/\frak p^c$ relates 
the appropriate power series in given systems of local parameters the 
field-of-norms 
functor is compatible with $\c P$-topological structures on the fields $K_n$ and $\c K$. 
\medskip 

Suppose $L$ is a finite extension of $K$ in $K_{alg}$. Then the tower 
$L_\d =(LK_n)_{n\geqslant 0}$ is again SDR and $X(L_{\d })=\c L$ 
is a separable extension of $\c K$ of degree $[LK_n:K_n]$, 
where $n\gg 0$. The extension $\c L/\c K$ is Galois if and only if   
for $n\gg 0$, $LK_n/K_n$ is Galois. From the definition of $\c L$ and 
$\c K$ it follows that we have a natural identification of 
groups 
$\op{Gal}(\c L/\c K)=\op{Gal}(LK_n/K_n)$. As a result, 
$X(K_{alg}):=\underset{\substack{\longrightarrow\\L}}\lim X(L_{\d})$ 
is a separable closure $\c K_{sep}$ of $\c K$ and the functor 
$X$ identifies 
$\Gal (\c K_{sep}/\c K)$ with $\Gal (K_{alg}/K_{\d }^{\infty })$, 
where $K_{\d }^{\infty }=\underset{\substack{\longrightarrow \\ n}}\lim K_n$.

Similarly to the classical 1-dimensional situation there is the following 
interpretation of the functor $X$. 
Let $\mathbb C_p(N)$ be the $v^1$-adic completion of $K_{alg}$ and let 
$R_0(N)=\underset{n\geqslant 0}\varprojlim \,\mathbb C_p(N)$ with respect to 
the $p$-th power maps on $\mathbb C_p(N)$. The operations on 
$R_0(N)$ are defined as follows: 
if $a=\{a_n\}_{n\geqslant 0}$ and $b=\{b_n\}_{n\geqslant 0}$ 
belong to $R_0(N)$ then $ab=\{a_nb_n\}_{n\geqslant 0}$ and 
$a+b=\{c_n\}_{n\geqslant 0}$, where $c_n=\underset{m\to\infty }
\lim (a_{n+m}+b_{n+m})^{p^n}$. So, $R_0(N)$ is a 
field of characteristic $p$, there is a natural embedding $\c K_{sep}\subset R_0(N)$, 
and $R_0(N)$ appears as the completion of $\c K_{sep}$ with 
respect to the first valuation. 

We denote by $R(N)\subset R_0(N)$  
the completion of the $N$-valuation ring of $\c K_{sep}$, 
and let $\m _{R(N)}$ be the maximal ideal of $R(N)$. 
Clearly, 
$$R(N)=\underset{n\geqslant 0}\varprojlim\, \c O_{\mathbb C_p(N)}=
\underset{n\geqslant 0}\varprojlim\,\c O_{K_{alg}}/p\,.$$ 
Note that 
the $\c P$-topology on $K$ is uniquely extended to $R_0(N)$ and 
this extension coincides with the extension of the 
$\c P$-topological structure of $\c K$ to $R_0(N)$ 
(as the completion of the $\c P$-topology on $\c K_{sep}$).

\subsection{} \label{S5.2}
Suppose $u\in\N $ and 
$\phi _1,\dots ,\phi _u\in \m _{K}\setminus\{0\}$ 
are independent modulo $p$-th powers, i.e. 
 for any $n_1,\dots ,n_u\in\Z $, the product 
 $\phi _{1}^{n_1}\dots \phi _{u}^{n_u}\in K^{*p}$ 
 iff all $n_i\equiv\,0\,\op{mod}\,p$. 
 Let ${K}^{\phi }=
 \underset{n\geqslant 0}\bigcup K(\phi _{1n}, \dots ,\phi _{un})$, where 
 for $1\leqslant i\leqslant u$, $\phi _i=\phi _{i0}$ and for all $n\in\N $, 
 $\phi _{i,n-1}=\phi _{in}^p$. 
 
 For $n\in\N $, let $\zeta _{n+1}\in K_{alg}$ 
 be such that $\zeta _{n+1}^{p}=\zeta _{n}$. Here $\zeta _1\in K$ is 
 our given $p$-th primitive root of unity. 
 Let $\wt{K}^{\phi }=K^{\phi }(\{\zeta _n\,|\, n\in\N\})$. 
 Then $\wt{K}^{\phi }/K$ is normal and let  
 $\Gamma ^{\phi }$ be its Galois group.

\begin{Lem} \label{L5.1} 
If $\Gamma ^{\phi }_{ <p}$ is the 
maximal quotient of $\Gamma ^{\phi }$ 
of period $p$ and nilpotent class $<p$ then there is 
a natural exact sequence of groups 
$$\op{Gal}(\wt{K}^{\phi }/K^{\phi })\To \Gamma ^{\phi}_{<p}\To 
\op{Gal}(K(\phi _{11},\dots ,\phi _{u1})/K)\To 1\, .$$
\end{Lem}

\begin{proof} 
Clearly,  
$\Gamma _{\wt{K}_{\phi }/K}=\langle \sigma, \tau _1,\dots ,\tau _u\rangle $, where 
for any $1\leqslant i,m\leqslant u$ and 
some $s_0\in \Z$, 
$\sigma \zeta _n=\zeta _n^{1+ps_0}$, 
$\sigma\phi _{in}=\phi _{in}$, 
$\tau _m(\zeta _n)=\zeta _n$, $\tau _m(\phi _{in})=
\phi _{in}\zeta ^{\delta _{mi}}_n$, and 
$\sigma ^{-1}\tau _m\sigma =\tau _m^{(1+ps_0)^{-1}}$. 

Therefore, $(\Gamma ^{\phi})^p=\langle \sigma ^{p},
\tau _1^{p}, \dots ,\tau _u^{p}\rangle $ and 
for the subgroup of second commutators we have 
$C_2(\Gamma ^{\phi })\subset 
\langle\tau _1^{p},\dots ,\tau _u^{p}\rangle
\subset (\Gamma ^{\phi})^p$. As a result, it holds  
$(\Gamma ^{\phi})^pC_p(\Gamma ^{\phi })=
\langle \sigma ^{p}, \tau _1^{p},\dots ,\tau _u^{p}\rangle $  
 and the lemma is proved. 
 \end{proof}

We are going to apply the above lemma to our field $K$ and the set 
of local parameters 
$\pi =\{\pi _1,\dots ,\pi _n\}$. 
The lemma provides us with the field extensions 
$\wt{K}^{\pi }\supset K^{\pi }\supset K$. Let     
$\Gamma _{<p}:=\Gamma /\Gamma ^{p}C_p(\Gamma )$ and ${\Gamma }_{K^{\pi }}=
\Gal (K_{alg}/K^{\pi })$. The embedding ${\Gamma }_{K^{\pi }}\subset\Gamma $ 
induces a continuous homomorphism 
$\iota  ^{\pi }:{\Gamma }_{K^{\pi }}\To \Gamma _{<p}$. Denote by $\kappa ^{\pi }$ 
the natural surjection    
$\Gamma _{<p}\To \op{Gal}(K(\root p\of {\pi _1}, \dots ,\root p\of {\pi _N})/K)=
\underset{1\leqslant m\leqslant N}\prod \langle\tau _m\rangle ^{\Z /p}$, 
where $\tau _m(\root p\of{\pi _i})=\root p\of {\pi _i}\zeta _1^{\delta _{im}}$. 

\begin{Prop} \label{P5.2}
 The following sequence of profinite groups 
$$\Gamma _{K^{\pi }}\overset{\iota ^{\pi }}\To \Gamma _{<p}\overset{\kappa ^{\pi }}
\To \prod\limits _{1\leqslant m\leqslant N}
\langle \tau _m\rangle ^{\Z /p}\To 1\, $$ 
is exact. 
\end{Prop}
\begin{proof} 
Note that the elements of the group 
$\Gamma _{\wt{K}^{\pi }}=\op{Gal}(K_{alg}/\wt{K}^{\pi })$ together with a lift 
$\hat\sigma \in {\Gamma}_{K^{\pi }}$ of $\sigma $  
generate the group $\Gamma _{K^{\pi }}$. Now   
the exact sequence from Lemma \ref{L5.1} implies that 
$\op{Ker}\,\kappa ^{\pi }$ 
is generated by $\hat\sigma $ and the image of 
$\Gamma _{\wt{K}^{\pi }}\subset {\Gamma }_{K^{\pi }}$. 
As a result, this kernel coincides with the image of 
${\Gamma }_{K^{\pi }}$ in $\Gamma _{<p}$\,.  
\end{proof}

Let $R(N)$ be the ring from Sect.\,\ref{S5.1}. Recall, there is a natural embedding 
$k\subset R$ and for $1\leqslant m\leqslant N$, 
$t_m:=\underset{n}\varprojlim\,\{\pi _{mn}\}_{n\geqslant 0}\in R(N)$, 
where $\pi _{m0}=\pi _m$ and $\pi _{m, n+1}^p=\pi _{mn}$. 

Following Sect.\,\ref{S2.3} set $t=(t_1,\dots ,t_N)$ and $\c K=k((t))$. 
Then $\c K$ is a closed subfield of $R_0(N)=\op{Frac}\,R(N)$ 
with a system of local parameters $t$. 
The tower $K^{\pi }_{\d }=\{K(\pi _{1n},\dots ,\pi _{Nn})\}_{n\geqslant 0}$ 
is SDR and ${K}^{\pi }=K_{\d }^{\pi\,\infty }$. Therefore,  
the field-of-norms functor $X$ from Sect.\,\ref{S5.1}  
identifies $X(K^{\pi }_{\d })$ with $\c K$ 
and $R_0(N)$ with the completion of $\c K_{sep}$. 
In particular, there is a natural inclusion $\Gamma \To\Aut\, R_0(N)$ 
which induces the identification of $\c G=\op{Gal} (\c K_{sep}/\c K)$ and 
$\Gamma _{K^{\pi }}$. 
 
We are going to 
apply below the results of the previous sections and will use the appropriate 
notation related to our field $\c K$, e.g. $\c G_{<p}=\op{Gal}(\c K_{<p}/\c K)$, where 
$\c K_{<p}=\c K_{sep}^{\c G^pC_p(\c G)}$.  
The field-of-norms identification $\c G\simeq \Gamma _{K^{\pi }}$ 
composed with the morphism $\iota ^{\pi }$ from 
Prop.\,\ref{P5.2} induces a  
group homomorphism 
$\iota ^{\pi }_{<p}:\c G_{<p}\To\Gamma _{<p}$ and 
we obtain the following property. 

\begin{Prop} \label{P5.3} 
The following sequence of profinite groups is exact 
$$\c G_{<p}\overset{\iota ^{\pi }_{<p}}\To \Gamma _{<p}
\overset{j}\To \prod\limits _{1\leqslant m\leqslant N}
\langle \tau _m\rangle ^{\Z /p}\To 1\, .$$  
\end{Prop}

\subsection{Isomorphism $\kappa _{<p}$} \label{S5.3} 
Let $\bar c^1\in\Z ^N_{>\bar 0}$ be such that 
$p=\pi ^{\bar c^1}u$, where $u\in\c O_K^*$ (as earlier, $\pi ^{\bar c^1}=
\pi _1^{c_1^1}\dots 
\pi _N^{c_N^1}$ with $\bar c^1=(c_1^1, \dots ,c_N^1)$). 

Set $\bar c^0=p\bar c^1/(p-1)$. Note that $\bar c^0\in\Z _{>\bar 0}^N$, because 
$\zeta _1\in K$, $\zeta _1-1\in\pi ^{\bar c^2}\c O^*_K$ 
with $\bar c^2\in\Z _{>\bar 0}^N$ and 
$(p-1)\bar c^2=\bar c^1$. 

Consider the auxillary Lie algebra $\bar{\c M}_{<p}$ 
from Sect.\,\ref{S4.4} and its analogue 
$$\bar{\c M}_{R_0(N)}=\left (\sum _{1\leqslant s<p}t^{-s\bar c^0}
\bar{\c L}_{\bar c^0}(s)_{\m _{R(N)}}\right )\otimes R(N)/t^{(p-1)\bar c^0}\, .$$ 
 
Clearly,   
$\bar f\in\bar{\c L}_{R_0(N)}$ and $\bar f\otimes 1\in\bar{\c M}_{R_0(N)}$. 
(Recall that $\bar\pi =\pi _{\bar f}(\bar e):\bar{\c G}\simeq G(\bar{\c L})$.)

Consider a $v^1_{\c K}$-continuous 
embedding $\eta $ of the field $\c K$ into $R_0(N)$ 
such that the image of 
$(\id _{\bar{\c L}}\otimes \eta )\bar e$ in 
$\bar{\c M}_{R_0(N)}$ coincides with $\bar e\otimes 1$. 

Such field embedding satisfies $\eta |_k=\id $ and 
is uniquely determined by a choice 
of the elements $\eta (t_m)\in R(N)_0$, 
where $1\leqslant m\leqslant N$, such that 
$\eta (t_m)\equiv t_m\,\op{mod}\,t^{(p-1)\bar c^0}\m _{R(N)}$.

\begin{Prop} \label{P5.4}  
There is a unique lift $\bar\eta $ of $\eta $ to $\c K(p)$ such that 
$$(\id _{\bar{\c L}}\otimes \bar\eta )\bar f\otimes 1=
\bar f\otimes 1\,.$$ 
\end{Prop} 

\begin{proof}
Let $\hat\eta $ be an extension of $\eta $ to $\c K_{sep}$.
 Clearly, $\sigma ((\id _{\bar{\c L}}\otimes \hat\eta )\bar f\otimes 1)=
 (\bar e\otimes 1)\circ ((\id _{\bar{\c L}}\otimes \hat\eta )\bar f\otimes 1)$. So,  
 $$(-(\bar f\otimes 1))\circ 
 ((\id _{\bar{\c L}}\otimes \hat\eta )\bar f\otimes 1)
 \in \bar{\c M}_{R_0(N)}|_{\sigma =\id }=\bar{\c L}\,.$$
 In other words, there is $l\in\bar{\c L}$ such 
 that for $g =\bar\pi ^{-1}(l)$, it holds 
 $$(\id _{\bar{\c L}}\otimes\hat\eta )\bar f\otimes 1=
 (\bar f\otimes 1)\circ l=g(\bar f)\otimes 1\,.$$
 As a result, we can take $\bar\eta =\hat\eta \cdot g^{-1}$. 
 The uniqueness of $\bar\eta $ is obvious because 
 any two such lifts differ by an element of $\bar{\c G}$ but 
 $\bar{\c G}$ acts strictly on $\bar f\otimes 1$. 
\end{proof}

Let $\varepsilon =(\zeta _n\,\op{mod}\,p)_{n\geqslant 0}\in R\subset R(N)$ 
be Fontaine's element 
(here $\zeta _1\in K$ is our $p$-th root of unity and for all $n$, 
$\zeta _n^p=\zeta _{n-1}$). 

Let $\zeta _1=1+\pi ^{\bar c^2}\sum _{\iota\geqslant \bar 0}
[\beta _{\iota }]\pi ^{\iota }$,  
where all $[\beta _{\iota }]$ are the 
Teichmuller representatives of $\beta _{\iota }\in k$ and $\beta _{\bar 0}\ne 0$. 
Here $\pi ^{\bar c^2}\c O_{K}=
(\zeta _1-1)\c O_{K}=p^{1/(p-1)}\c O_{K}$, i.e. $p\bar c^2=\bar c_0$.

Consider the identification of rings  
$R(N)/t^{\bar c^1}\simeq \c O_{K_{alg}}/p$, coming from the 
projection 
$R(N)=\underset{n\geqslant 0}
\varprojlim \,(\c O_{K_{alg}})_n$ to 
$(\c O_{K_{alg}})_0\,\op{mod}\,p$.   
This implies  $\sigma ^{-1}\varepsilon \equiv 1+\sum _{\imath\geqslant \bar 0}
\beta _{\imath}t^{\bar c^2+\imath }\,\op{mod}\,t^{\bar c^1}R(N)$ and, 
 therefore, 
 $$\varepsilon\equiv 1+\sum _{\imath \geqslant\bar 0}\beta ^p_{\imath }
 t^{\bar c^0+p\imath }\,\op{mod}\,t^{(p-1)\bar c^0}R(N)\,.$$

Assume the morphisms  $h^{(m)}_{\omega }\in\op{Aut}\c K$ 
from Sect.\,\ref{S4.1} are determined by $\omega $ such that 
$E(\omega ^p)=1+\sum _{\iota \geqslant \bar 0}
\beta _{\iota }^pt^{\bar c^0+p\iota }$, 
i.e. for all $\iota $, $A_{\iota}(\omega )=\beta ^p_{\iota }$. 
As a result, for any $m$, 
$\tau _m(t)\equiv h_{\omega }^{(m)}(t)\op{mod}\,t^{(p-1)\bar c^0}\m _{R(N)}$. 

Suppose $\tau\in\Gamma $ (recall that $\Gamma \subset\Aut\,R_0(N)$). 
Then for some integers $m_1, \dots ,m_N$, 
we have the following congruence modulo $t^{(p-1)\bar c^0}\m _{R(N)}$ 
$$\tau (t)=\{t_1\varepsilon ^{m_1}, \dots ,t_N\varepsilon ^{m_N}\}\equiv 
\{h_{\omega }^{(1)m_1}(t_1),\dots , 
h_{\omega }^{(N)m_N}(t_N)\}\, .$$
Let 
$h_{\tau }=h_{\omega }^{(1)m_1}\dots h_{\omega }^{(N)m_N}\in\Aut\,{\c K}$. 
Then (use Prop.\,\ref{P4.1})
$$\tau |_{\c K}(t)\equiv h_{\tau }(t)
\,\op{mod}\,t^{(p-1)\bar c^0}\m _{R(N)}\, .$$ 

 This means that 
$\eta :=\tau ^{-1}|_{\c K}\cdot h_{\tau }:\c K\To R_0(N)$ satisfies the assumption 
of Prop.\,\ref{P5.4} and we can consider the corresponding 
lift $\bar\eta :\c K(p)\To R_0(N)$. Let $\hat\eta $ be a lift of $\bar\eta $ 
to $\c K_{sep}$. 

Set $\bar h_{\tau }:=(\tau\cdot\hat\eta )|_{\c K(p)}$. 

Then $\bar h_{\tau }|_{\c K}=(\tau\cdot \hat\eta )|_{\c K}=h_{\tau }$ 
and by Galois theory $\bar h_{\tau }\in\Gamma _{\omega }\subset \Aut\c K(p)$. 
\medskip 

As a result, we obtained the map 
(of sets) $\kappa :\Gamma \To \Gamma _{\omega }$ uniquely characterized by 
the property $(\id _{\bar{\c L}}\otimes \tau )\bar f = 
(\id _{\bar{\c L}}\otimes \kappa (\tau ))\bar f\,$.

\begin{Prop} \label{P5.5} 
$\kappa $ induces a group isomorphism $\kappa _{<p}:\Gamma _{<p}\To\Gamma _{\omega }$. 
 \end{Prop}

\begin{proof} Suppose $\tau_1,\tau \in\Gamma $. 
Let $\bar C\in\bar{\c L}_{\c K}$ and $\bar{\c A}\in\Aut _{\bar{\c L}}$ be such that 
 $(\id _{\bar{\c L}}\otimes \kappa (\tau ))\bar f=
\bar C\circ (\bar{\c A}\otimes \id _{\c K(p)})\bar f$. Then 
$$(\id _{\bar{\c L}}\otimes\kappa (\tau_1\tau ))\bar f= 
(\id _{\bar{\c L}}\otimes \tau _1\tau )\bar f
= (\id _{\bar{\c L}}\otimes \tau _1)
(\id _{\bar{\c L}}\otimes \tau )\bar f$$
$$=(\id _{\bar{\c L}}\otimes \tau _1)
(\id _{\bar{\c L}}\otimes\kappa (\tau ))\bar f= 
(\id _{\bar{\c L}}\otimes \tau _1)
(\bar C\circ (\bar{\c A}\otimes \id _{\c K(p)})\bar f)= $$
$$(\id _{\bar{\c L}}\otimes \tau _1)
\bar C\circ (\bar{\c A}\otimes \tau _1)\bar f
= (\id _{\bar{\c L}}\otimes \kappa (\tau _1))
\bar C\circ (\bar{\c A}\otimes\id _{\c K(p)})
(\id _{\bar{\c L}}\otimes \tau _1)\bar f=$$
$$(\id _{\bar{\c L}}\otimes \kappa (\tau _1))
\bar C\circ (\bar{\c A}\otimes\id _{\c K(p)})(\id _{\bar{\c L}}
\otimes \kappa (\tau _1))\bar f= (\id _{\bar{\c L}}\otimes \kappa (\tau _1))
(\bar C\circ (\bar{\c A}\otimes\id _{\c K(p)})\bar f$$
$$=(\id _{\bar{\c L}}\otimes \kappa (\tau _1))
(\id _{\bar{\c L}}\otimes\kappa (\tau ))\bar f
= (\id _{\bar{\c L}}\otimes \kappa (\tau _1)\kappa (\tau ))\bar f$$
and, therefore, $\kappa (\tau _1\tau )=\kappa (\tau _1)\kappa (\tau )$ 
(use that $\Gamma _{\omega }$ acts strictly on the orbit of $\bar f$).  
In particular, 
$\kappa $ factors through the natural projection $\Gamma \to\Gamma _{<p}$  
and defines the group homomorphism 
$\kappa _{<p}:\Gamma _{<p}\to\Gamma _{\omega }$.

Recall that we have the identification of 
$\Gamma _{K^{\pi }}=\op{Gal}(K_{alg }/K^{\pi })$ with 
$\c G=\op{Gal}(\c K_{sep}/\c K)$ and, therefore, $\kappa _{<p}$ identifies   
the groups $\kappa (\Gamma _{K^{\pi }})$ and 
$G(\bar{\c L})\subset\Gamma _{\omega }$. 
Besides, $\kappa _{<p}$ induces a group isomorphism of the group 
$\op{Gal}(K(\pi _{11},\dots,\pi _{1N})/K)=\langle \tau _1\rangle ^{\Z /p}
\times\dots\times\langle \tau _m\rangle ^{\Z /p}$ 
and the quotient $\langle h_{\omega }^{(1)}\rangle ^{\Z /p}\times 
\dots \times \langle h^{(N)}_{\omega }\rangle ^{\Z /p}$ of $\Gamma _{\omega }$. 
Now Proposition \ref{P5.3} 
implies that $\kappa _{<p}$ is a group isomorphism. 
\end{proof} 

\subsection{Groups $\Gamma ^{\c P}_{<p}$} \label{S5.4} 

Consider the group isomorphism 
$\kappa _{<p}:\Gamma _{<p}\To \Gamma _{\omega }$ from Sect.\,\ref{S5.3}. 

\begin{definition} $\op{cl}^{\c P}\Gamma _{<p}$ is the class of 
conjugated subgroups in $\Gamma _{<p}$ containing 
$\Gamma _{<p}^{\c P}:=\kappa ^{-1}(\Gamma _{\omega }^{\c P})$.   
\end{definition}

This definition involves a choice of local parameters in $K$. 

\begin{Prop} \label{P5.6}
 The class $\op{cl}^{\c P}\Gamma _{<p}$ does not depend on a choice of 
 a system of local parameters in $K$.
\end{Prop}

\begin{proof} 
Let $\pi =\{\pi _1,\dots ,\pi _N\}$ and $\pi '=\{\pi '_1,\dots ,\pi '_{1N}\}$ 
be two systems of 
local parameters in $K$. Let $K^{\pi }_1=K(\pi _{11}, \dots ,\pi _{1N})$ and 
$K_1^{\pi '}=K(\pi '_{11}, \dots ,\pi '_{1N})$, where (as earlier) for all $m$, 
$\pi _{1m}^p=\pi _m$ and $\pi '^p_{1m}=\pi _m'$. Denote by $K^{(1)}$ 
the composite of $K^{\pi }_1$ and $K^{\pi '}_1$.

Consider the SDR-towers $K^{\pi }_{\d }$ and $K_{\d }^{\pi '}$, 
the fields-of-norms $\c K=X(K_{\d }^{\pi })$ and $\c K'=X(K_{\d }^{\pi '})$ 
with the corresponding systems of local parameters 
$t=\{t_1,\dots ,t_N\}$ and $t'=\{t'_1,\dots ,t'_N\}$. 

Then we have the appropriate field extensions 
$$\c K\subset \c K^{(1)}\subset \c K(p)\subset \c K_{<p}\subset R_0(N)$$
$$\c K'\subset \c K'^{(1)}\subset \c K'(p)\subset \c K'_{<p}\subset R_0(N)$$
where $\c K^{(1)}=X(K^{(1)}K_{\d }^{\pi })$, 
$\c K'^{(1)}=X(K^{(1)}K_{\d }^{\pi '})$, the fields $\c K_{<p}$ and $\c K(p)$ 
were defined earlier, $\c K'_{<p}$ and $\c K'(p)$ are their analogs if $\c K$ is replaced by  
$\c K'$. 
The exact sequence  from Sect.\,\ref{S4.5} 
$$1\To \bar{\c G}\To \Gamma _{\omega }\To H_{\omega }\To 1\,,$$ 
with  $H_{\omega }:=
\prod\limits _{1\leqslant m\leqslant N}
\langle h_{\omega }^{(m)}\rangle /
\langle h_{\omega }^{(m)p}\rangle $, 
gives rise to the exact sequence 
\begin{equation} \label{E5.2} 1\To \bar{\c G}^{(1)}\To \Gamma _{\omega }\To 
\op{Gal}(\c K^{(1)}/\c K)\times H_{\omega }\To 1\, ,
\end{equation} 
where $\bar{\c G}^{(1)}=\op{Gal}(\c K(p)/\c K^{(1)})$. 
We have a similar sequence for $\c K'$   
\begin{equation} \label{E5.3} 1\To \bar{\c G}'^{(1)}\To \Gamma '_{\omega '}\To 
\op{Gal}(\c K'^{(1)}/\c K')\times H'_{\omega '}\To 1\, ,
\end{equation} 
where $H'_{\omega '}$ is an analog of $H_{\omega }$. 
The isomorphisms  
$\kappa _{<p}:\Gamma _{\omega }\simeq\Gamma _{<p}$ and   
$\kappa '_{<p}:\Gamma '_{\omega '}\simeq\Gamma _{<p}$ 
induce the isomorphisms 
$$\op{Gal}(K^{(1)}/K)\simeq\op{Gal}(\c K^{(1)}/\c K)\times H_{\omega }
\simeq\op{Gal}(\c K'^{(1)}/\c K')\times H'_{\omega '}\,,$$  
\begin{equation} \label{E5.4} \op{Gal}(K_{<p}/K^{(1)})\simeq 
\bar{\c G}^{(1)}\simeq\bar{\c G}'^{(1)}\,. 
\end{equation} 
We want to study relation between the $\c P$-structures on 
$\bar{\c G}^{(1)}$ and $\bar{\c G}'^{(1)}$ 
(induced from $\bar{\c G}$ and $\bar{\c G}'$). 

Recall, cf. Sect.\,\ref{S4.4}, there is a Lie algebra 
$$\bar{\c M}=\left (\sum _{s}t^{-s\bar c^0}\m _{\c K}\bar{\c L}_{\bar c^0}(s)
\right )\otimes\c O_{\c K}/t^{(p-1)\bar c^0}$$
and its analogue $\bar{\c M}_{\c K(p)}$, where $\c K$ is replaced by $\c K(p)$. 
There are $\bar e\otimes 1\in\bar{\c M}$ and 
$\bar f\otimes 1\in\bar{\c M}_{\c K(p)}$ such that 
$\sigma (\bar f\otimes 1)=(\bar e\otimes 1)\circ (\bar f\otimes 1)$ 
and the identification 
$\pi _{\bar f}(\bar e):\bar{\c G}\simeq G(\bar{\c L})$ is given by 
$\tau\mapsto (-\bar f\otimes 1)\circ (\tau\bar f\otimes 1)$. 

The algebra 
$\bar{\c L}^{\c P}$ (as well as the group 
$\bar{\c G}^{\c P}$) is defined in terms of the $\c P$-topological 
structure on  $\bar{\c M}$ coming from the corresponding structures on 
$\c O_{\c K}/t^{(p-1)\bar c^0}$ and $\bar{\c L}$. Note that 
$\bar{\c L}$ is the image of $\c L$ and the $\c P$-topology on $\c L$  
is induced from $L/C_2(L)=\Hom (\bar{\c K},\F _p)$. 
Therefore, the $\c P$-structure on 
$\bar{\c L}$ comes from 
$$\bar{\c L}/C_2(\bar{\c L})=
\Hom (t^{-(p-1)\bar c^0}\m _{\c K}/(\sigma -\id )\c K,\F _p)\,.$$ 

As a result, 
$$\bar e\otimes 1\in\bar{\c M}^{\c P}=
\left (\sum _{s}t^{-s\bar c^0}\m _{\c K}\bar{\c L}^{\c P}_{\bar c^0}(s)
\right )\otimes ^{\c P}\c O_{\c K}/t^{(p-1)\bar c^0}$$
and $\pi _{\bar f}(\bar e)^{-1}G(\bar{\c L}^{\c P})=\bar{\c G}^{\c P}$. 

Similar construction is used to obtain (in the context of $\c K'$) that 
$$\bar e'\otimes 1\in\bar{\c M}'^{\c P}=
\left (\sum _{s}t'^{-s\bar c'^0}\m _{\c K'}\bar{\c L}'^{\c P}_{\bar c'^0}(s)
\right )\otimes ^{\c P}\c O_{\c K'}/t'^{(p-1)\bar c'^0}$$
and $\pi _{\bar f'}(\bar e')^{-1}G(\bar{\c L}'^{\c P})=\bar{\c G}'^{\c P}$. 

The $\c P$-subgroups $\bar{\c G}^{\c P}$ and $\bar{\c G}'^{\c P}$ 
``live in different worlds``, but the field-of-norms functor $X$ 
identifies  them with subgroups in $\Gamma _{<p}=\op{Gal}(K_{<p}/K)$.  
This procedure can be specified as follows. 

Let $\Gamma _{<p}=G(L)$, where $L$ is a suitable Lie $\F _p$-algebra. 
Introduce 
$$\bar{M}=\left (\sum _{1\leqslant s<p}
(\zeta _1-1)^{-s}C_s(L)_{\m _{K}}
\right )\otimes\c O_{K}/p\,.$$
$$\bar{M}_{<p}=\left (\sum _{1\leqslant s<p}
(\zeta _1-1)^{-s}C_s(L)_{\m _{K_{<p}}}
\right )\otimes\c O_{K_{<p}}/p\,.$$

The projection $\op{pr}_1:R(N)=
\underset{n\geqslant 0}\varprojlim (\c O_{K_{alg}}/p)_n\To (\c O_{K_{alg}}/p)_1$ establishes 
the ring isomorphism 
$R(N)/t^{p\bar c^1}=R(N)/t^{(p-1)\bar c^0}\simeq \c O_{K_{alg}}/p\,$,  
cf. Sect.\,\ref{S5.4} (note that $\Ker \,\op{pr}_1$ is generated by 
$t^{(p-1)\bar c^0}$). 

Consider the induced by $\op{pr}_1$ isomorphism 
$\c O_{\c K(p)}/t^{(p-1)\bar c^0}\simeq 
\c O_{K_{<p}}/p$. This gives  
for each $1\leqslant s<p$, the  
compatible identifications of the 
\linebreak 
$\c O_{\c K(p)}/t^{(p-1)\bar c^0}$-module  
$t^{-s\bar c^0}\m _{\c K(p)}\otimes \c O_{\c K(p)}/t^{(p-1)\bar c^0}$ with  
$\c O_{K_{<p}}/p$\,-module 
$(\zeta _1-1)^{-s}\m _{K_{<p}}\otimes \c O_{K_{<p}}/p$. 
In addition, the field-of-norms functor identifies the group 
$\bar{\c G}$ with $\Gamma ^{\pi }_1=
\op{Gal}(K_{<p}/K_1^{\pi })\subset\Gamma $ 
and the Lie algebra $\bar{\c L}$ with the Lie subalgebra 
$L_1^{\pi }\subset L$, where $G(L_1^{\pi })=\Gamma ^{\pi }_1$. 
As a result, we obtain the embedding 
$F_{<p}:\bar{\c M}_{\c K(p)}\To \bar{M}_{<p}$ and the induced 
embedding 
$F=F_{<p}|_{\bar{\c M}}:\bar{\c M}\To \bar{M}$. 

Set $e^{\pi }=F(\bar e\otimes 1)$ and 
$f^{\pi }=F_{<p}(\bar f\otimes 1)$. Then 
$\sigma f^{\pi }=e^{\pi }\circ f^{\pi }$ and the 
map $\tau\mapsto (-f^{\pi })\circ \tau (f^{\pi })$ recovers  
the identification 
$\Gamma _{1}^{\pi }\simeq \bar{\c G}$ or, equivalently, 
$\kappa ^{\pi }:L_1^{\pi }\simeq \bar{\c L}$.    
(We used it earlier when constructing the isomorphism $\kappa _{<p}$.) 

The identification $\kappa ^{\pi }$ is compatible with the $\c P$-topology. 

Indeed, 
the $\c P$-topological structure on $L_1^{\pi }$ 
comes (via tensor topology)  
from $L_1^{\pi }/C_2(L_1^{\pi })
=\Hom (p^{-1}\m _K/(\sigma -\id )K,\F _p)$ and the field-of-norms  
identification of 
$p^{-1}\m _K/(\sigma -\id )K$ with $t^{-(p-1)\bar c^0}\m _{\c K}
/(\sigma -\id )\c K$.   

Repeating the above argiments in the context of the system of 
parameters $\pi '$ we obtain the $\c P$-continuous identification 
$\kappa ^{\pi '}:L_1^{\pi '}\simeq \bar{\c L'}$. 

Let $\bar{\c G}^{(1)}=G(\bar{\c L}^{(1)})$ and 
$\bar{\c G}'^{(1)}=G(\bar{\c L}'^{(1)})$. 
Then isomorphism \eqref{E5.4} appears as compatible with $\c P$-structures 
isomorphisms $L^{(1)}\simeq \c L^{(1)}\simeq \c L'^{(1)}$. 
Therefore, the conjugacy classes of $(\kappa ^{\pi})^{-1}\bar{\c L}^{(1)\c P}$ and 
$(\kappa ^{\pi '})^{-1}\bar{\c L}'^{(1)\c P}$ coincide. 

Finally, applying the interpretation from Sect.\ref{S3.6} we obtain 
from exact sequences \eqref{E5.2} and \eqref{E5.3} that 
the conjugate classes of $\kappa _{<p}(\Gamma ^{\c P}_{\omega })$ and 
$\kappa '_{<p}(\Gamma '^{\c P}_{\omega '})$ coincide. 
\end{proof} 

As a result,  $\Gamma ^{\c P}_{<p}$ is provided with the $\c P$-topology induced from  
the subgroup $\bar{\c G}^{\c P}$ and this topology does not depend on 
a choice of such subgroup (i.e. on the choice of local parameters in $K$. 

For any (local $N$-dimensional) 
subfield $K\subset K'\subset K_{<p}$, set 
$H^{\c P}=H\cap\Gamma _{<p}^{\c P}$, where $H=\op{Gal}(K_{<p}/K')$.  
As earlier (in the case of local fields of characteristic $p$), 
we easily obtain the following property. 

\begin{Cor} \label{P5.7} a) 
The profinite completion of $H^{\c P}$ is $H$;
\medskip 

b) $(\Gamma _{<p}:H)=
(\Gamma _{<p}^{\c P}:H^{\c P})=[K':K]$;  

c) for all subfields $K'$, 
the subgroups $H^{\c P}$ of $\Gamma _{<p}^{\c P}$ can be 
characterized as all $\c P$-open subgroups of finite index in $\Gamma _{<p}^{\c P}$. 
\end{Cor}

\subsection{Explicit structure of $\Gamma ^{\c P}_{<p}$} \label{S5.6} 

Recall that $K$ is an $N$-dimensional local field, with the first residue field 
of characteristic $p$ and the last residue field $k\simeq \F _{p^{N_0}}$. 
Review the above results about $\Gamma _{<p}$. 

Suppose $\pi =\{\pi _1,\dots ,\pi _N\}$ is a system of local parameters 
in $K$ and $\zeta _1\in K$ is a primitive $p$-th root of unity. 
Then 
$$\zeta _1=E\left (\sum _{\bar\iota \geqslant\bar 0}[\beta _{\iota }]
\pi ^{\bar c_2+\iota }\right )\,,$$ 
where $E(X)\in\Z _p[[X]]$ is the Artin-Hasse exponential, 
all $\iota\in\Z ^N_{\geqslant\bar 0}$,  
$[\beta _{\iota }]$ are the Teichmuller representatives 
of $\beta _{\iota }\in k$, 
$\beta _{\bar 0}\ne 0$. (Recall that 
$\pi ^{\bar c^2}:=\pi _1^{c_1}\dots \pi _N^{c_N}$, where 
$\bar c^2=(c_1, \dots ,c_N)$.)

Associate with $K$ the topological $\F _p$-algebra $L^{\c P}$ 
as follows.

Consider an $N$-dimensional local field $\c K$ of characteristic 
$p$ with finite residue field $k\simeq \F _{p^{N_0}}$. 
Then we have the toplogical module 
$\bar{\c K}^{\c PD}:=\Hom _{\c P\text{-cont}}(\c K/(\sigma -\id )\c K,\F _p)$, which 
generate the ''maximal`` $\F _p$-Lie algebra 
$\c L^{\c P}$ of nilpotent class $<p$, i.e. the 
quotient of free Lie algebra with module of generators 
$\bar{\c K}^{\c PD}$ by the ideal of $p$-th commutators. 
(Note that 
$\c L^{\c P}$ appears as the projective limit of ''maximal`` 
Lie algebras 
$\c L_{\bar C_{\alpha }}$ (of nilpotent class $<p$) generated by 
the elements of open subsets $\bar C^{D}_{\alpha }\subset\c M$.) 
If $t=\{t_1,\dots ,t_N\}$ is a system of local parameters in $\c K$ then 
$\c L^{\c P}_k$ is provided with natural system of 
$\c P$-topological generators 
$\{D_{an}\ |\ a\in\Z _N^+(p), n\in\Z /N_0\}\cup \{D_0\}$. 
Recall that 
$$\Z _N^+(p)=\{a\in\Z ^N_{>\bar 0}\ |\ \op{gcd}(a,p)=1\}\, , 
\Z _N^0(p)=\Z ^+(p)\cup\{\bar 0\}\,.$$

Let $\bar c^0=p\bar c^2$ and introduce the weight function $\op{wt}$ 
on $\c L^{\c P}$ by setting $\op{wt}(D_{an})=s$ 
if $(s-1)\bar c^0\leqslant a<s\bar c^0$, 
$\op{wt}(D_0)=1$. 
Let $\c L^{\c P}_k(p)$ be the ideal 
in $L^{\c P}_k$ of elements of weight $\geqslant p$,   
$\bar{\c L}_k=\c L_k/\c L(p)$. 
Let $\sigma $ be the Frobenius automorphism 
on $k$ and denote by the same symbol  
the $\sigma $-linear automorphism of $\bar{\c L}_k$ such that 
$D_{an}\mapsto D_{a,n+1}$, $D_{\bar 0}
\mapsto D_{\bar 0}$. (We also set $D_{\bar 0n}=
\sigma ^n(\alpha _0)D_{\bar 0}$, where 
$\alpha _0$ is a fixed element from 
$W(k)\subset K$ with absolute trace 1.) 

Let $\bar{\c L}^{\c P}=\c L^{\c P}_k/\c L^{\c P}_k(p)|_{\sigma =\id }$ 
with induced $\c P$-topological structure. 

Introduce the Lie $\F _p$-algebra $L^{\flat }$ as the maximal 
algebra of nilpotent class $<p$ 
containing $\bar{\c L}$ and the generators 
$\{\bar l^{(m)}\ |\ 1\leqslant m\leqslant N\}$. 
Introduce the ideal of 
relations $\c R^{\flat }_k$ in $L^{\flat }_k$ as follows.

Specify recurrent relation \eqref{E4.9} to our situation: 
$$
\sigma \bar c^{(m)}_1-\bar c^{(m)}_1+
\sum _{a\in \Z _N^0(p)}t^{-a}V^{(m)}_{a\,0}= 
$$
$$-\sum _{\iota\geqslant \bar 0}\beta _{\iota }^{p}
\sum _{1\leqslant k<p}\,\frac{1}{k!}\,t^{-(a_1+\dots +a_k)+p\iota }a_1^{(m)}
[\dots [D_{a_10},D_{a_20}],\dots ,D_{a_k0}]$$
$$-\sum _{2\leqslant k<p}\frac{1}{k!}t^{-(a _1+\dots +a_k)}
[\dots [V^{(m)}_{a_10},D_{a_20}],\dots ,D_{a_k0}]$$
\begin{equation} \label{E5.5} -\sum _{1\leqslant k<p}\,
\frac{1}{k!}\,t^{-(a_1+\dots +a_k)}
[\dots [\sigma \bar c^{(m)}_1,D_{a_10}],\dots ,D_{a_k0}]
\end{equation} 
(the indices  
$a_1,\dots ,a_k$ in all above sums run over $\Z _N^0(p)$).   
Recall that   
$\bar c_1^{(m)}\in\bar{\c L}^{\c P}_{\c K}$ and $V_{a0}\in\bar{\c L}^{\c P}_k$.  
Then the ideal $\c R^{\flat }_k\subset L^{\flat}_k$ is 
generated by the following elements: 

1) $[D_{an},\bar l^{(m)}]-\sigma ^n(V^{(m)}_{a0})$, $a\in\Z ^0_N(p)$;
\medskip 

2) $[\bar l^{(i)}, \bar l^{(j)}]-\bar l[i,j]$, 
where $\bar l[i,j]\in\bar{\c L}$ is given in notation of  
Cor.\,\ref{C4.11} 
$$\bar l[i,j]=\ad ^{(i)}(\bar c_1^{(j)}(\bar 0))-\ad ^{(j)}
(\bar c^{(i)}_1(\bar 0))+
 \sum _{\iota\in\Z ^N}[\bar c_1^{(i)}(\iota ), \bar c_1^{(j)}(-\iota )]\,.$$
 
Let $L_k=L_k^{\flat}/\c R^{\flat }_k$ and $L=L_k|_{\sigma =\id }$. 

\begin{remark}
 Taking another solutions $\{\bar c_1^{(m)}, V^{(m)}_{a0}\ |\ a\in\Z _N^0(p), 
 1\leqslant m\leqslant N\}$ of \eqref{E5.5} is equivalent to 
 replacing the generators $\bar l^{(m)}$ modulo $\bar{\c L^{\c P}}$. 
\end{remark}

Summarizing the above results we state the following theorem.  

\begin{Thm} \label{T5.8}
a) $\Gamma _{<p}^{\c P}\simeq G(L^{\c P})$;

b) $\Gamma _{<p}$ is the profinite completion of $G(L^{\c P})$. 
\end{Thm}

\subsection{Simplest example}  \label{S5.7}

The above description of $\Gamma _{<p}$ essentially uses the equivalence 
of the category of $p$-groups and $\F _p$-Lie algebras (there 
is no operation of extension of scalars in the category of $p$-groups). 
It could be also verified that the study of involved $p$-groups at  
the level of their Lie algebras gives much simpler form of the 
corresponding relations. At the same time the above presentation of 
$L^{\c P}$ as the quotient of $L^{\flat }$ by appropriate 
relations is not the simplest one. Some relations, e.g. 
$[D_{an},\bar l^{(m)}]-V_{an}$ with $a>\bar c^0$ can be used to 
exclude extra generators $D_{an}$ with $a>\bar c^0$. Ideally, the whole description 
should be done in terms of, say, the minimal system of generators 
$\{D_{an}\ |\ a<\bar c^0\}\cup\{\bar l^{(m)}\ |\ 1\leqslant m\leqslant N\}$. 
This was done in Sect.\,\ref{S4.7} where we presented our description 
modulo third commutators. 

In the case of the last residue field $k\simeq\F _p$ we do not need 
the operation of extension of scalars and can express the answer 
directly in terms of groups. This will not give any simplifications, but 
can be easily obtained when working modulo third commutators. 

Namely, if $k\simeq \F _p$ then $\Gamma ^{\c P}_{<p}\,
\op{mod}C_3(\Gamma ^{\c P}_{<p})$ appears as the group 
with $\c P$-topological generators 
$\{\tau _a\ |\ a\in\Z ^0_N(p), a<\bar c^0\}\cup 
\{\bar h^{(m)}\ |\ 1\leqslant m\leqslant N\}$ and the 
subgroup of relations generated (as normal subgroup) 
by following relations:  
\medskip 

$\bullet $\  $R(i,j)=(\bar h^{(i)}, \bar h^{(j)})$, here $1\leqslant i< j\leqslant N$;

$$\bullet \ R_{\bar 0}(m)=(\tau _{\bar 0},\bar h^{(m)})\prod _{\iota\geqslant\bar 0} 
\prod _{\substack{b+c=\\
\bar c^0+p\iota }}
(\tau _{b},\tau _{c})^{b^{(m)}\beta _{\iota }/2}\,; 
\qquad\qquad\qquad\qquad\qquad\ \ 
$$ 
 
$$\bullet \ R_a(i,j)= 
(\tau _a,\bar h^{(i)a^{(j)}}/\bar h^{(j)a^{(i)}})
\prod _{\substack{n\geqslant 1\\ \iota\geqslant 0}}  
\prod_{\substack{b+c/p^{n}=\\ \bar c^0+p\iota +a/p^{n}}}
(\tau _{b}, \tau_{c})^{(b^{(i)}c^{(j)}-
b^{(j)}c^{(i)})\beta _{\iota }}\times $$

$$\qquad\qquad\qquad\qquad \times 
\prod _{\substack{n\geqslant 0\\ \iota\geqslant\bar 0}}
\prod_{\substack{b+c=\\ 
\bar c^0+p\iota +ap^n}} 
(\tau _{b}, \tau _{c})^{(b^{(i)}c^{(j)}-
b^{(j)}c^{(i)})\beta _{\iota }/2} \,;$$
here $1\leqslant m\leqslant N$, $1\leqslant i<j\leqslant N$, $a\in\Z ^+_N(p)$. 
\medskip 

The above example could be simplified if we take $2$-dimensional 
$K=\Q _p(\zeta _1)\{\{\pi _2\}\}$ with the system 
of local parameters $\pi =\{\pi _1,\pi _2\}$, where 
$E(\pi _1)=\zeta _1$. In this case $\bar c^0=(p,0)$, 
$\beta _{\bar 0}=1$ and all remaining $\beta _{\iota }=0$. 

We have the system of minimal generators 
$$\{\tau _{a}\ |\ a\in\Z _2^+(p), 
a<(p,0)\}\cup\{\tau _{\bar 0}\}\cup\{\bar h^{(1)},\bar h^{(2)}\}$$ 
and the following relations: 
\medskip

$\bullet $\  $\c R(1,2)=(\bar h^{(1)}, \bar h^{(2)})$;

$$\bullet\ \c R_{\bar 0}(1)=(\tau_{\bar 0}, \bar h^{(1)})
\prod _{1\leqslant\alpha \leqslant \frac{p-1}{2}}
\prod_{\gamma }
(\tau _{(\alpha ,\gamma )}, \tau _{(p-\alpha ,-\gamma )})^{\alpha }
\qquad\qquad\qquad\qquad 
$$

$$\bullet\ \c R_{\bar 0}(2)=(\tau_{\bar 0}, \bar h^{(2)})
\prod _{1\leqslant\alpha \leqslant \frac{p-1}{2}}\hskip 6pt 
\prod_{\gamma }
(\tau _{(\alpha ,\gamma )}, \tau _{(p-\alpha ,-\gamma )})^{\gamma }
\qquad\qquad\qquad\qquad 
$$
 
$$\bullet\ \c R_{a}(1,2):\ (\tau _{a},\bar h^{(1)a^{(2)}}/
\bar h^{(2)a^{(1)}})\times 
\prod _{\gamma }(\tau _{(p-1,\gamma )},
\tau _{(p,a^{(2)}-p\gamma )})^{-a^{(2)}\delta _{0,a^{(1)}}}\ \ $$

$$\times 
\prod _{\substack{n\geqslant 1,\beta }} 
(\tau _{(p,-\beta )}, \tau _{a+(0,p^n\beta )})^{a^{(1)}\beta }\times  
\prod _{\substack{b+c=\\(p,0)+a}} 
(\tau _{b}, \tau _{c})^{(b^{(1)}c^{(2)}-
b^{(2)}c^{(1)})/2}\,$$

\end{document}